\documentclass[12pt,a4paper]{amsart}
\usepackage{amsmath,amsfonts,amssymb}
\usepackage{url}

\def\Z{{\Bbb Z}}
\def\Q{{\Bbb Q}}    
\def\R{{\Bbb R}}
\def\C{{\Bbb C}}

\def\cL{{\mathcal L}}
\def\cM{{\mathcal M}}

\def\fM{{\mathfrak M}}
\def\hh{{\widetilde{h}}}
\def\hg{{\widehat{g}}}
\def\tg{{\widetilde{g}}}
\def\hf{{\widetilde{\phi}}}

\theoremstyle{plain}
\newtheorem{thm}{Theorem}[section]
\newtheorem{lemma}[thm]{Lemma}
\newtheorem{prop}[thm]{Proposition}
\newtheorem{cor}[thm]{Corollary}

\begin{document}

\title[Dimensions of paramodular forms with involutions]{Dimensions of paramodular forms and compact twist modular forms with involutions}
\author{Tomoyoshi Ibukiyama}
\address{Department of Mathematics, Graduate School of Science, Osaka University,
Machikaneyama 1-1, Toyonaka, Osaka, 560-0043 Japan.}
\email{ibukiyam@math.sci.osaka-u.ac.jp}
\keywords{Paramodular forms, algebraic modular forms}
\thanks
{This work was supported by JSPS KAKENHI Grant Number JP19K03424, 
JP20H00115, JP23K03031, and AIM SQuaRE ``Computing paramodular forms''}
\subjclass{11F46,11F55,11F72}
\maketitle

\begin{abstract}
We give an explicit dimension formula for paramodular forms of degree two of 
prime level with plus or minus sign of the Atkin--Lehner involution of 
weight $\det^k\operatorname{Sym}(j)$ with $k\geq 3$, 
as well as a dimension formula for algebraic modular forms of any weight associated with the binary quaternion hermitian maximal lattices in non-principal genus of 
prime discriminant with fixed sign of the involution.
These two formulas are essentially equivalent 
by a recent result of N. Dummigan,  A. Pacetti. G. Rama and G.  Tornar\'ia 
on correspondence between algebraic modular forms  
and paramodular forms with signs. So we give the formula by calculating the latter.
When $p$ is odd, our formula for the latter 
is based on a class number formula of  
some quinary lattices by T. Asai and its interpretation to the type number of 
quaternion hermitian forms given in our previous works. 
On paramodular forms, we also give a dimensional bias between plus 
and minus eigenspaces, some list of palindromic 
Hilbert series, numerical examples for small $p$ and $k$, and the complete list of 
primes $p$ such that there is no paramodular cusp form of level $p$ of weight 3 with plus sign. 
This last result has geometric meaning on moduli of Kummer 
surface with $(1,p)$ polarization. 
\end{abstract}

\section{Introduction}\label{sec:intro}
\allowdisplaybreaks[4]

The paramodular forms of degree $2$ of level $p$ are Siegel modular forms 
associated with the moduli space of abelian surfaces with polarization of type 
$(1,p)$. These forms have recently gathered significant attentions from 
researchers. 
For example, denote by $B$ the definite quaternion algebra with discriminant $p$
and $B_A^{\times}$ its adelization. Then  
Eichler's classical work established  
a well-known Hecke equivariant 
bijection between elliptic new cusp forms of level $p$ 
and algebraic modular forms on $B_A^{\times}$
which can be represented as vectors of some harmonic polynomials of $4$ variables.
Paramodular cusp forms arise in a natural extension 
of this correspondence to the case of 
symplectic groups of rank two a la Langlands.
In fact, the present author proposed a conjectural correspondence 
between paramodular forms of prime level $p$ and algebraic modular forms 
on $G_A$ which is the adelization of the quaternion hermitian group 
\[
G=\{g\in M_2(B); gg^*=n(g)1_2\}, \qquad g^*=(\overline{g_{ji}}) \text{ for } g=(g_{ij})
\]
with respect to the stabilizers of 
a maximal quaternion hermitian lattice in the non-principal genus,
where $\overline{*}$ is the main involution of $B$. 
Similarly as in the classical case, such algebraic modular forms in case of weight $0$  
have relations to some arithmetics of 
supersingular abelian surfaces as written for example 
in \cite{ibulattice}. 
The conjecture regarding this correspondence was initially proposed in \cite{ibuparamodular} 
for the scalar valued prime level 
case, later extended to the vector valued case in \cite{ibuparavector},
and also to the square-free level case in \cite{ibukitayama}.
These conjectures 
were based on dimensional equalities and numerical examples. 
In these works, new forms were defined in a slightly different
way from the elliptic modular case,  
and certain concrete lifting behaviours were also conjectured.
By the way, later the theory of new forms 
based on paramodular fixed vectors has been 
extensively developed in Robert--Schmidt 
\cite{robertschmidt} and Schmidt \cite{schmidtiwahori}.

Recently, the above conjectural correspondence has been proved by van Hoften \cite{vanhoften} and R\"osner--Weissauer \cite{weissauer} independently. 
Subsequently, the correspondence has been refined to the correspondence 
between plus and minus eigenspaces of Atkin--Lehner type involutions
in Dummigan--Pacetti--Rama--Tornar\'ia \cite{dummigan}.

The main objective of this paper is to give explicitly  
dimensions of the above algebraic modular forms and 
dimensions of paramodular forms of prime level with 
the Atkin--Lehner plus or minus sign. In these formulas,
weights $\det^kSym(j)$ of paramodular forms are restricted to 
the case $k\geq 3$ since there exists no algebraic modular forms 
corresponding to the case $k<3$. 
But the formula for $k\geq 3$ is also useful for determining  
dimensions for weight $2$ for concrete cases 
as explained in Poor--Yuen \cite{poor}. 
Their motivation behind \cite{poor} is a 
Shimura--Taniyama type conjecture for abelian surfaces, which is a generalization
of the fact that the zeta functions of elliptic curves defined over $\Q$ are given 
by those of 
elliptic cusp forms of weight $2$. Examples of this type of abelian surfaces  
were first explored in Yoshida \cite{yoshida}. It is conjectured more precisely in Brumer--Kramer \cite{brumerkramer} that 
for any abelian surface $A$ defined over $\Q$ with 
$End(A)=\Z$ of conductor $p$, there should exist a paramodular form of weight 
$2$ of level $p$ whose $L$ function gives the $L$ function of $A$. 
Explicit calculations of paramodular 
forms of weight $2$ of small concrete levels were carried out in \cite{poor}, 
which also includes some calculation of dimensions of higher weights of plus and minus eigenvalues. 
The examples for $k\geq 3$ 
obtained from their calculations coincide with the data
derived from our general formula. 

The outline of our paper is as follows.
In section \ref{sec:theorem}, we state our main theorems \ref{compactdim} and \ref{paramodular} on dimensions without going into details.
In section \ref{proof2}, we prove Theorem \ref{paramodular}.  
In section \ref{proof1}, we prove Theorem \ref{compactdim}
for $p\neq 2$, $3$. 
In section \ref{p23}, using a more direct method, 
we prove Theorem \ref{compactdim} for $p=2$ and $3$. 
We also give explicit generating functions of dimensions for 
scalar valued case for $p=2$, $3$ as examples.
In section \ref{appendix}, we quote concrete dimension formulas of algebraic modular forms used in this paper 
from \cite{hashimotoibu} II
and explicitly describe the group characters we require.
 In section \ref{bias}, we show that 
$(-1)^k\bigl(\dim S_k^+(K(p))-\dim S_k^-(K(p))\bigr)\geq 0$ for any $k\geq 3$
in Theorem \ref{plusminusdiff}.
In section \ref{example}, we give several numerical examples of dimensions  
for small $p$ as well as a certain list of 
Hilbert series with palindromic property, and also 
tables of dimensions 
for $k=3$, $4$, $5$, $6$, $7$, $8$, $10$ 
for small $p$. In particular, we determine all primes $p$ such that 
$\dim S_3^{+}(K(p))=0$.

\section{Main Theorems}\label{sec:theorem}
In this section, we will state our results on 
explicit dimension formulas for algebraic modular forms and 
paramodular forms that have the Atkin--Lehner plus or minus sign without going into details.
Precise definitions and most proofs will be given in later sections.

First we start from algebraic modular forms.  
Let $B$ be the definite quaternion algebra over $\Q$ with prime discriminant $p$.
For a natural number $n$, define a positive definite quaternion hermitian form $h(x,y)$ of 
$B^n$ by 
\[
h(x,y)=\sum_{i=1}^{n}x_i\overline{y_i} \qquad x=(x_i), \quad y=(y_i) \in B^n,
\]
where $\overline{*}$ is the main involution of $B$. 
This is the unique positive definite quaternion hermitian form on $B^n$ up to 
base change of $B^n$ over $B$ (See \cite{shimura} Lemma 4.4). For $g=(g_{ij})\in M_n(B)$, we put 
$g^*=\,^{t}\overline{g}=(\overline{g_{ji}})$. 
We denote by $G$ the group of quaternion hermitian similitudes of degree $n$ 
defined by 
\[
G=\{g\in M_n(B); gg^*=n(g)1_n, n(g)\in \Q^{\times}_{+}\}.
\]
We denote by $G_A$ the adelization of $G$, and by $G_v$ 
the local factor of $G_A$ at a place $v$.
In particular, if we define the subgroup $G^1$ of $G$ 
by taking elements with $n(g)=1$, then $G^1_{\infty}$ is 
the compact twist $Sp(n)$ of the split symplectic group 
$Sp(n,\R)\subset SL_{2n}(\R)$ of real rank $n$. 
For an irreducible representation $(\rho,V)$ 
of $G^1_{\infty}$ such that $\rho(\pm 1_n)=1$, we define 
a representation of $G_A$ by 
\begin{equation}\label{definerep}
G_A\rightarrow G_{\infty}\rightarrow G_{\infty}/\{center\}\cong G^1_{\infty}/\{\pm 1_n\}
\stackrel{\rho}{\longrightarrow} GL(V).
\end{equation}    
We denote this representation also by $\rho$ by abuse of language.
For an open subgroup $U$ of $G_A$, we define 
the space $\fM_{\rho}(U)$ of modular forms on $G_A$ 
with respect to $U$ of weight $\rho$
by 
\[
\fM_{\rho}(U)=\{f:G_A\rightarrow V; f(uga)=\rho(u)f(g) \text{ for all } g \in G_A, u\in U, a\in G\}.
\]
(The modular forms of this type are nowadays called algebraic modular forms in \cite{gross}. Classically the Brandt matrices that appeared in the theory of Eichler (and of Jacquet-Langlands) are of this sort. For higher degrees, see also \cite{ihara}, \cite{hashimoto}, \cite{hashimotoibu}.)
When $n=2$, as was shown in \cite{hashimotoibu} I Theorem 3 and II Theorem
and \cite{ibuparavector}, \cite{ibumiddle} Theorem 5.1, 
the dimension of $\fM_{\rho}(U)$ is
explicitly known for any $\rho$ and 
for any $U=U_{pg}(p)$ or $U_{npg}(p)$ corresponding to the principal genus, or the non-principal 
genus of maximal lattices in $B^2$.
For a precise definition, see the first paragraph of Section \ref{proof1}.
The irreducible representation $\rho$ of the compact group $G_{\infty}^1=Sp(2)$ corresponds 
with a Young diagram parameter $(f_1,f_2)$ with $f_1\geq f_2\geq 0$ 
with $f_1\equiv f_2 \bmod 2$ (coming from the assumption $\rho(\pm 1_2)=1$).
We write this representation by $\rho_{f_1,f_2}$ and we write 
\[
\fM_{\rho_{f_1,f_2}}(U)=\fM_{f_1,f_2}(U).
\]
Here in this paper we are interested in the case $U=U_{npg}(p)$ 
since this case has a good correspondence with paramodular forms. 
Our first problem is to decompose $\fM_{f_1,f_2}(U_{npg}(p))$ under the 
Atkin--Lehner involution for $U_{npg}(p)$ and give a dimension formula for each eigenspace. So we explain what the involution is in this case. 
Let $O$ be a maximal order of $B$. 
We denote by $\pi$ a local prime element at $p$ 
of $O_p=O\otimes_{\Z}\Z_p$. We may assume that $\pi^2=-p$. 
Then the Atkin--Lehner involution for $\fM_{f_1,f_2}(U_{npg}(p))$ is 
defined to be the action of the Hecke double coset $U_{npg}(p)\pi U_{npg}(p)
=U_{npg}(p)\pi=\pi U_{npg}(p)$.
This action gives an involution (under a natural normalization), since $\pi^2=-p$. 
The representation matrix of the action of $U_{npg}(p)\pi$ 
on $\fM_{f_1,f_2}(U_{npg}(p))$ 
as a Hecke algebra is denoted by $R_{f_1,f_2}(\pi)$.  
The $+1$ and $-1$ eigen subspaces of $\fM_{f_1,f_2}(U_{npg}(p))$ of the matrix $R_{f_1,f_2}(\pi)$ 
will be denoted by $\fM^{\pm}_{f_1,f_2}(U_{npg}(p))$, respectively. We would like to give dimension formulas for 
these spaces. 
By definition, we have 
\[
Tr(R_{f_1,f_2}(\pi))
=\dim \fM^{+}_{f_1,f_2}(U_{npg}(p))-\dim \fM^{-}_{f_1,f_2}(U_{npg}(p))
\]
and since a formula for 
\[
\dim \fM_{f_1,f_2}(U_{npg}(p))=\dim \fM^+_{f_1,f_2}(U_{npg}(p))+
\dim \fM^-_{f_1,f_2}(U_{npg}(p))
\]
has been known in \cite{hashimotoibu} II Theorem and \cite{ibuparavector}
(and also reproduced in section \ref{appendix}),  
all we should do is to give a formula for $Tr(R_{f_1,f_2}(\pi))$.
This will be given below as Theorem \ref{compactdim}.

Before stating the result, we explain some notation that we need in the formula. 

For any $g \in G_{\infty}$, the characteristic polynomial $\phi(x)$ 
of degree $4$ of the image of the embedding 
$g \in G_{\infty}\subset M_2(B\otimes_{\Q}\R)\subset M_4(\C)$ 
is said to be a principal polynomial of $g$.
The character $Tr(\rho_{f_1,f_2}(g))$ of $g \in G_{\infty}$ depends only on the 
principal polynomial $\phi(x)$ of $g/\sqrt{n(g)}\in G^1_{\infty}=Sp(2)$, 
and by our assumption that 
$\rho_{f_1,f_2}(\pm 1_2)=1$, 
the characters for $\phi(x)$ and $\phi(-x)$ are the same. The general formula for the characters 
is well known and found in \cite{weyl} Theorem 7.8 E.
Here, we need the following principal polynomials
\begin{align*}
&\phi_2(x) =(x-1)^2(x+1)^2,  \qquad \phi_6(x) = (x^2+1)^2, 
\quad \phi_9(x) = x^4+x^2+1,  
\\ & \phi_{11}(x)  = x^4+1, \quad \phi_{13}(x)=x^4+\sqrt{5}x^3+3x^2+\sqrt{5}x+1, 
\\ & \phi_{14}(x) = (x^2+\sqrt{2}x+1)^2, \quad
\phi_{15}(x)= x^4+\sqrt{2}x^3+x^2+\sqrt{2}x+1, \\
&\phi_{16}(x)=(x^2+\sqrt{2}x+1)(x^2+1), 
\quad \phi_{17}(x) = (x^2+\sqrt{3}x+1)(x^2+1) 
\end{align*}
and denote by $\chi_i$ the character of elements corresponding to $\phi_i(\pm x)$. 
(We use this strange numbering to maintain consistency 
with our previous works.) 
The polynomials $\phi_{14}$ to $\phi_{17}$ appear only when $p=2$ or $3$. 
We denote by $h(\sqrt{-d})$ the class number of an imaginary quadratic field 
$\Q(\sqrt{-d})$ and by $B_{2,\chi}$ the second generalized Bernoulli number for 
the character $\chi$ corresponding to the real quadratic extension 
$\Q(\sqrt{p})/\Q$. By definition, we have 
\[
B_{2,\chi}=\frac{1}{f}\sum_{a=1}^{f}\chi(a)a^2-\sum_{a=1}^{f}\chi(a)a, 
\]
where $f$ is the conductor of $\chi$, i.e. $f=p$ if $p\equiv 1 \bmod 4$ and $f=4p$ if $p\equiv 3 \bmod 4$, and the latter sum
is always $0$ in this case since $\chi(-1)=1$ (\cite{bernoulli} p.54). 
Our first theorem is given below.
Here the cases $p=2$, $3$ are slightly exceptional and 
the formula in these cases will be reproduced 
as Theorem \ref{compactp2p3} in section \ref{p23} 
with a different proof.
We put $\delta_{ab}=1$ if $a=b$ and $=0$ otherwise. We define the quadratic residue symbol $\left(\frac{d}{p}\right)$ for a prime $p$ to be 
$1$, $-1$ and $0$ if $p$ splits unramified, remains prime, 
and is ramified in $\Q(\sqrt{d})$, respectively. 

\begin{thm}\label{compactdim}
We assume that $n=2$. Then an explicit formula for $Tr(R_{f_1,f_2}(\pi))$
is given for any $f_1\geq f_2\geq 0$ with $f_1\equiv f_2 \bmod 2$ as follows. \\
For $p\equiv 1 \bmod 4$, we have 
\begin{align*}
Tr R_{f_1,f_2}(\pi) & =\frac{\chi_2}{2^5\cdot 3}\left(9-2\left(\frac{2}{p}\right)\right)B_{2,\chi} 
+ \frac{h(\sqrt{-p})}{2^4}\chi_6 
\\ & + \frac{h(\sqrt{-2p})}{2^3}\chi_{11}+\frac{h(\sqrt{-3p})}{2^2\cdot 3}
\left(3+\left(\frac{2}{p}\right)\right) \chi_{9} +\frac{\delta_{p5}}{5}\chi_{13}.
\end{align*}
For $p\equiv 3 \bmod 4$ and $p>3$, we have 
\begin{align*}
Tr R_{f_1.f_2}(\pi) & =\frac{\chi_2}{2^5\cdot 3}B_{2,\chi} 
+ \frac{h(\sqrt{-p})}{2^4}\left(1-\left(\frac{2}{p}\right)\right)\chi_6
\\ & + \frac{h(\sqrt{-2p})}{2^3}\chi_{11} + 
\frac{h(\sqrt{-3p})}{2^2\cdot 3}\chi_{9}. 
\end{align*}
For $p=2$, we have 
\[
Tr(R_{f_1,f_2}(\pi))=\frac{1}{48}\chi_2+\frac{1}{16}\chi_6+\frac{1}{6}\chi_9 
+\frac{5}{16}\chi_{11}
+\frac{1}{48}\chi_{14}+\frac{1}{6}\chi_{15} + \frac{1}{4}\chi_{16}.
\]
For $p=3$, we have 
\[
Tr(R_{f_1,f_2}(\pi))=\frac{1}{24}\chi_2+\frac{1}{24}\chi_6+\frac{1}{3}\chi_9
+\frac{1}{4}\chi_{11}+\frac{1}{3}\chi_{17}.
\]
Here we have $\chi_i=Tr(\rho_{f_1,f_2}(g_i))$ for any $g_i \in G^1$ 
whose principal polynomials
are given by $\phi_i(\pm x)$ for $i=2$, $6$, $9$, $11$, $13$, $14$, $15$, $16$, $17$  defined above.
\end{thm}

The proof of Theorem \ref{compactdim} will be given in section \ref{proof1}. 

The explicit value of $\chi_i$ for each $(f_1,f_2)$ for the case $p\neq 2$, $3$ 
is given as follows 
as can be easily deduced from the classical result in \cite{weyl} Theorem 7.8 E
(See also section \ref{appendix}).
For the formulas for $\chi_{14}$ to $\chi_{17}$, see section 5. 
We use notation 
$[a_0,\ldots,a_m;m]_{b}$ that means the number $a_i$  
when $b\equiv i \bmod m$. 
\begin{align*}
\chi_2 & = (-1)^{f_1}\frac{(f_1+2)(f_2+1)}{2}, \\
\chi_6 & =  \frac{(-1)^{(f_1+f_2)/2}}{2}
\times \left\{\begin{array}{ll} f_1+2 & \text{ if } f_2\equiv 0 \bmod 2, \\
-(f_2+1) & \text{ if } f_2\equiv 1 \bmod 2, 
\end{array}\right.
\\
\chi_{9} & = \left\{
\begin{array}{ll}
[1,0,0,-1,0,0;6]_{f_2} & \text{ if } f_1-f_2\equiv 0 \bmod 6, \\{} 
[-1,1,0,1,-1,0;6]_{f_2} & \text{ if } f_1-f_2\equiv 2 \bmod 6, \\{} 
[0,-1,0,0,1,0:6]_{f_2}  & \text{ if } f_1-f_2\equiv 4 \bmod 6,
\end{array} \right.  
\\
\chi_{11} & = \left\{
\begin{array}{ll}
(-1)^{(f_1-f_2)/4}[1,-1,0,0;4]_{f_2} & \text{ if } f_1-f_2\equiv 0 \bmod 4, \\
(-1)^{(f_1-f_2-2)/4}[0,1,-1,0;4]_{f_2} & \text{ if }f_1-f_2\equiv 2 \bmod 4, 
\end{array}\right. \\
\chi_{13} & = 
\left\{\begin{array}{ll}
[1,2,1,0,0,-1,-2,-1,0,0;10]_{f_2} & \text{ if } f_1- f_2\equiv 0 \bmod 10, \\{}
[2,0,-1,1,0,-2,0,1,-1,0;10]_{f_2} & \text{ if } f_1- f_2\equiv 2 \bmod 10, \\{}
[-2,-2,2,2,0,2,2,-2,-2,0;10]_{f_2} & \text{ if } f_1- f_2 \equiv 4 \bmod 10, \\{}
[-1,1,0,-2,0,1,-1,0,2,0;10]_{f_2} & \text{ if } f_1- f_2\equiv 6 \bmod 10, \\{}
[0,-1,-2,-1,0,0,1,2,1,0;10]_{f_2} & \text{ if } f_1-f_2 \equiv 8 \bmod 10.
\end{array}\right.
\end{align*}
By definition of $\fM^{\pm}_{f_1,f_2}(U_{npg}(p))$, we have 
\begin{align}
\label{compactplus}
\dim \fM^+_{f_1,f_2}(U_{npg}(p)) & =
\frac{1}{2}\biggl(\dim \fM_{f_1,f_2}(U_{npg}(p))+Tr(R_{f_1,f_2}(\pi))\biggr),\\
\label{compactminus}
\dim \fM^-_{f_1,f_2}(U_{npg}(p)) & = \frac{1}{2}\biggl(
\dim \fM_{f_1,f_2}(U_{npg}(p))-Tr(R_{f_1,f_2}(\pi))\biggr).
\end{align}
But the dimension for $\dim \fM_{f_1,f_2}(U_{npg}(p))$ 
is explicitly known for any $\rho_{f_1,f_2}$ 
as given in \cite{hashimotoibu} II 
(reproduced in Theorem \ref{nonprincipaldim} in 
section \ref{appendix}). 
So as a corollary of the above theorem, we have 
explicit dimension formulas for $\dim \fM^{+}_{f_1,f_2}(U_{npg}(p))$ and 
$\dim \fM^{-}_{f_1,f_2}(U_{npg}(p))$. 
(By the same sort of calculation given in this paper, we can give an
explicit  formula also for $\fM_{f_1,f_2}^{\pm}(U_{pg}(p))$, 
but we omit it here.)

Now we proceed to our next theme. 
Since $G^1_{\infty}=Sp(2)$ is the compact twist of the split symplectic group 
$Sp(2,\R)\subset SL_4(\R)$ of real rank $2$,
we may expect a nice correspondence between algebraic 
modular forms on $G_A$ and 
Siegel cusp forms of degree $2$. (A general principle by Langlands, and for this special case 
asked also by Y. Ihara \cite{ihara}.) 
An explicit correspondence for the 
non-principal genus for $n=2$ was conjectured in our previous works 
\cite{ibuparamodular}, \cite{ibuparavector}, \cite{ibukitayama}
with precise comparison of explicit dimension formulas, 
and this conjecture has been proved by van Hoften in \cite{vanhoften} Theorem 3 and
by R\"{o}sner and Weissauer in \cite{weissauer} Proposition 12.3, independently by 
a completely different method. (Other parahoric cases 
different from the above case 
have been conjectured in \cite{ibumiddle},\cite{ibueuler},\cite{hashimotoibudimII},
but this is another story.) 
The corresponding Siegel cusp forms here 
are so called paramodular forms. 
Now there also exists the Atkin--Lehner type involution on 
paramodular forms of level $p$. 
Recently, Dummigan, Pacetti, Rama and Tornar\'ia generalized 
the above correspondence to the case between 
paramodular forms and algebraic modular forms with given sign of the Atkin--Lehner 
involution (See \cite{dummigan} Theorem 9.6 etc.). (Their theorem includes 
some general level cases, but here we are concerned only with 
prime level. See also \cite{ladd}.) We will explain more details below.
For any positive integer $N$, we denote by $K(N)$ the paramodular subgroup of $Sp(2,\Q)$ of level $N$ defined by 
\[
K(N)=Sp(2,\Q)\cap \begin{pmatrix} \Z & N\Z & \Z & \Z 
\\ \Z & \Z & \Z & N^{-1}\Z \\
\Z & N\Z & \Z & \Z \\
N\Z & N\Z & N\Z & \Z 
\end{pmatrix}.
\]
Let $H_n$ be the Siegel upper half space of degree $n$.  
For any $g=\begin{pmatrix} A & B \\ C & D \end{pmatrix} \in Sp(2,\R)$ and 
a $V_{k,j}$-valued function $F$ of $Z \in H_2$,
we put 
\[
F|_{k,j}[g]=\rho_{k,j}(CZ+D)^{-1}F(gZ),
\]
where $(\rho_{k,j},V_{k,j})$ is the irreducible representation 
$\det^k Sym(j)$ of $GL_2(\C)$ and $Sym(j)$ is the symmetric tensor representation of degree $j$. (When compared with algebraic modular forms explained  
later, we will put $f_1=k+j-3$ and $f_2=k-3$, assuming $k\geq 3$.) 
We denote by $A_{k,j}(K(N))$ the space of paramodular forms belonging to $K(N)$ of weight $\rho_{k,j}$,
and by $S_{k,j}(K(N))$ its subspace of cusp forms. 
By definition, $S_{k,j}(K(N))$ means the vector space of 
$V_{k,j}$-valued holomorphic functions on $H_2$ such that $F|_{k,j}[\gamma]=F$
for any $\gamma \in K(N)$ and that vanish at all the cusps.
When $j$ is odd, we always have $A_{k,j}(K(N))=0$ by $\rho_{k,j}(-1_2)=(-1)^{2k+j}=-1$.
When $j=0$, we simply write $A_{k,0}(K(N))=A_k(K(N))$ and $S_{k,0}(K(N))=S_{k}(K(N))$. 
The formula for $\dim S_{k,j}(K(N))$ is known 
for square free $N$ for any $k\geq 3$, $j\geq 0$ 
(See \cite{ibuparamodular}, \cite{ibuweightthree}  for $j=0$, $N=$prime,
\cite{ibuparavector} for $j>0$, $k>4$, $N=$prime,  \cite{ibukitayama}  
for square free $N$ with $j=0$, $k\geq 3$ and $j>0$ and $k>4$,  
and by Dan Petersen (colloquial communication) for $k=3$, $4$, $j>0$. 
The last case has been also reproved by van Hoften \cite{vanhoften}.   

For a prime $p$, we put 
\[
\eta=\frac{1}{\sqrt{p}}\begin{pmatrix} 0 & 0 & 0 & -1 \\ 0 & 0 & -1 & 0 \\
0 & p & 0 & 0 \\
p & 0 & 0 & 0 
\end{pmatrix}.
\]
Then $\iota:F\rightarrow F|_{k,j}[\eta]$ induces an involution on $S_{k,j}(K(p))$. (This can be regarded also as the action of
the Hecke operator associated with $K(p)\eta$.) 
We denote by $S_{k,j}^{\pm}(K(p))\subset S_{k,j}(K(p))$
the eigenspaces of $\iota$ belonging to eigenvalues $+1$ and $-1$, respectively. 
To adjust the lifting part in the correspondence, we need the space 
$S_k(\Gamma_0(p))$ of 
elliptic cusp forms of weight $k$ belonging to the group 
\[
\Gamma_0(p)=SL_2(\Z)\cap \begin{pmatrix} \Z & \Z \\ p\Z& \Z\end{pmatrix}.
\]
We denote by $S_{j+2}^{\pm}(\Gamma_0(p))$ 
the eigenspaces of $S_{j+2}(\Gamma_0(p))$ belonging to the eigenvalues $+1$ and $-1$ of 
the Atkin--Lehner involution $W_p$ defined by 
\[
f(z)|_{j+2}W_p=p^{-(j+2)/2}z^{-(j+2)}f(-1/pz)
\]
on $S_{j+2}(\Gamma_0(p))$. 
We put $S_{j+2}^{\pm,new}(\Gamma_0(p))=S_{j+2}^{\pm}(\Gamma_0(p))\cap S_{j+2}^{new}(\Gamma_0(p))$
where $S_{j+2}^{new}(\Gamma_0(p))$ is the space of new forms. 

Now by Theorem \ref{compactdim} and 
the above mentioned result of Dummigan, Pacetti, Rama, Tornaria in \cite{dummigan}, 
together with \cite{vanhoften}, \cite{weissauer}, and other results, 
we have the following explicitly calculable formula for $\dim S_{k,j}^{\pm}(K(p))$. 
\begin{thm}\label{paramodular}
Let $p$ be any prime. For  
$k\geq 3$ and even $j\geq 0$, we have an explicit formula for 
$\dim S_{k,j}^{\pm}(K(p))$. It is given by 
\begin{align*}
& \dim S_{k,j}^+(K(p)) = \dim S_{k,j}(Sp(2,\Z))+\dim \fM^{-}_{j+k-3,k-3}(U_{npg}(p))
\\
& -\dim S_{j+2}^{new,+}(\Gamma_0(p))\times \dim S_{2k+j-2}(SL_2(\Z)) 
\\
& \dim S_{k,j}^{-}(K(p)) = \dim S_{k,j}(Sp(2,\Z))-\delta_{j0}\dim S_{2k-2}(SL_2(\Z))
-\delta_{j0}\delta_{k3} \\ 
& +\dim \fM_{j+k-3,k-3}^{+}(U_{npg}(p))
-\dim S_{j+2}^{new,-}(\Gamma_0(p))\times S_{2k+j-2}(SL_2(\Z)).
\end{align*}
where the main part of RHS is explicitly given by Theorem \ref{compactdim}, 
Theorem \ref{compactp2p3}, Theorem \ref{nonprincipaldim},
and the other parts are explained below.
\end{thm}

In particular, $S_3^+(K(p))$ has a geometric meaning 
on the moduli of the Kummer surfaces associated to 
$(1,p)$ polarization (See \cite{gritsenkohulek} Theorem 1.5), 
and by using the above formula and estimating bounds of the class numbers and 
the Bernoulli numbers, 
we can give the complete list of 
primes such that $S_3^+(K(p))=0$ in Proposition \ref{weight3}. 
This includes a partial result in \cite{breedingpooryuen} and \cite{gritsenkopooryuen}
for primes, though they also treated composite levels. 

The meaning of the dimensional relations in Theorem 
\ref{paramodular} and its proof will be explained later in section \ref{proof2}.
Here we will explain why this is an explicit formula. 
The formula for $\dim \fM^{\pm}(U_{npg}(p))$ is deduced by Theorem \ref{compactdim}
and \eqref{compactplus} and \eqref{compactminus} by virtue of 
\cite{hashimotoibu} (reproduced as Theorem \ref{nonprincipaldim} in section 
\ref{appendix}). 
The formula for $\dim S_{k,j}(Sp(2,\Z))$ is in 
\cite{igusa}, \cite{igusa2}, \cite{tsushima}, \cite{petersen}
(not known for the case $k=2$ and big $j>0$ but this case   
is also excluded in Theorem \ref{paramodular}). 
The dimension 
$\dim S_{j+2}^{\pm}(\Gamma_0(p))$ is essentially in \cite{yamauchi} 1.6 Theorem, 
and easily obtained by the formula  
\eqref{dimgamma0}, \eqref{dimgamma0pm}, \eqref{dimgamma0p2}, \eqref{dimgamma0p3} explained below (See also \cite{martin}).
The formula for $\dim S_k(SL_2(\Z))$ is well known and given by \eqref{dimsl} below. 
So the above Theorem
\ref{paramodular} gives a really calculable formula 
for any given prime $p$, $k\geq 3$ and even $j\geq 0$.

We review the formula for $S_{j+2}^{\pm,new}(\Gamma_0(p))$ 
in \cite{yamauchi} for readers' convenience
(see also \cite{martin}). 
For any prime $p$ and even $k\geq 2$, as well known we have 
\begin{align}\label{dimgamma0}
& \dim S_{k}^{new}(\Gamma_0(p)) =
\dim S_{k}^{+,new}(\Gamma_0(p))+\dim S_{k}^{-,new}
(\Gamma_0(p)) \\
& \notag = \frac{(p-1)(k-1)}{12}+\frac{1}{4}(-1)^{k/2+1}\left(1-\left(\frac{-1}{p}\right)\right)
\\ \notag & \qquad \qquad \qquad \qquad
+\frac{1}{3}[-1,0,1;3]_k\left(1-\left(\frac{-3}{p}\right)\right)-\delta_{k2}.
\end{align}
For $p>3$ and even $k\geq 2$, we have 
\begin{align}
& \dim S_{k}^{+,new}(\Gamma_0(p))-\dim S_{k}^{-,new}(\Gamma_0(p))  = 
(-1)^{k/2}\frac{a_p\cdot h(\sqrt{-p})}{2}+\delta_{k2},
\label{dimgamma0pm}
\end{align}
where $h(\sqrt{-p})$ is the class number of $\Q(\sqrt{-p})$ and 
$a_p=1$ if $p\equiv 1 \bmod 4$, $a_p=2$ if $p\equiv 7 \bmod 8$, and 
$a_p=4$ if $p\equiv 3 \bmod 8$. 
We can also write $a_p h(\sqrt{-p})=h(-p)+h(-4p)$, where $h(-d)$ denotes the class number of quadratic order of discriminant $-d$ (not necessarily maximal), regarding $h(-d)=0$ if $-d\equiv 2,3 \bmod 4$. \\
The case $p=2$ and $3$ for even $k\geq 2$ is given by 
\begin{align}
\label{dimgamma0p2}
\dim S_k^{+,new}(\Gamma_0(2))-\dim S_{k}^{-,new}(\Gamma_0(2))
& =
\frac{(-1)^{k/2}-(-1)^{(k-4)(k-2)/8}}{2}+\delta_{2k},
\\
\label{dimgamma0p3}
\dim S_k^{+,new}(\Gamma_0(3))-\dim S_k^{-.new}(\Gamma_0(3))
& =
\delta_{2k}+\left\{
\begin{array}{ll}
-1 & \text{ if } k\equiv 2,6 \bmod 12, \\
0 & \text{ if } k\equiv 4,10 \bmod 12, \\
1 & \text{ if } k\equiv 0, 8 \bmod 12.
\end{array}\right.
\end{align}
It is also well known that 
\begin{equation}\label{dimsl}
\dim S_k(SL_2(\Z))=\frac{k-1}{12}+\frac{1}{4}(-1)^{k/2}+\frac{1}{3}[1,0,-1;3]_k-\frac{1}{2}+\delta_{k2}.
\end{equation}

As a special case of Theorem \ref{paramodular} for $(k,j)=(3,0)$, 
we have  
\[
\dim S_{3}^{+}(K(p))=H-T \quad \dim S_3^{-}(K(p))=T-1
\]
where $H$ is the class number and $T$ is the ($G$-)type number of the non-principal genus $\cL_{npg}$ in $B^2$ defined in \cite{ibutypenumber} p. 370.  The formula and numerical tables for $H$ and $T$ were given in \cite{hashimotoibu} II and \cite{ibuquinary} p.218.
Numerical examples of $\dim S_4^{\pm}(K(p))$ for many primes $p$ 
have been already given in 
\cite{poor} Table 4 and of course our results coincide with these values.
More numerical tables will be given in 
section \ref{example}. 

\section{Proof of Theorem \ref{paramodular}}\label{proof2}
In this section, we prove Theorem \ref{paramodular}, since 
this is much shorter than the proof of Theorem \ref{compactdim}.
We also explain the meaning of the dimensional relations in the theorem. 
Of course this relation can be read as a reflection of the Hecke equivariant 
bijection, but since this is obvious by \cite{vanhoften}, \cite{weissauer} and 
\cite{dummigan}, 
we do not explain such details. 

The local completions of groups $K(p)$ and $Sp(2,\Z)$ are maximal compact subgroups of 
$Sp(2,\Q_p)$ and there is no inclusion relation between $K(p)$ and $Sp(2,\Z)$ 
even if we take conjugacy.
But still we may consider paramodular forms in $S_{k,j}(K(p))$ 
coming from $Sp(2,\Z)$ as images 
of $S_{k,j}(Sp(2,\Z))$ and 
$S_{k,j}(Sp(2,\Z))|_{k}[\eta]=S_{k,j}(\eta^{-1}Sp(2,\Z)\eta)
\cong S_{k,j}(Sp(2,\Z))$ by trace operators. These are old forms in 
the sense of \cite{robertschmidt} and also explained in the much earlier paper \cite{ibuparamodular}. In general, old forms 
of $Sp(2,\Z)$ doubly appear in $S_{k,j}(K(p))$, but the 
Saito-Kurokawa lift from $S_{2k-2}(SL_2(\Z))$ 
(that exists only when $j=0$ and $k$ is even) 
exceptionally appears only once. 
Indeed, by \cite{royschmidtyi} Table 1 and \cite{schmidt}, \cite{schmidtpacket}, 
the possibility of local 
representations at $p$ which have both $Sp(2,\Z_p)$ fixed vectors and 
$K(p)$ fixed vectors are (I) and (IIb) 
in the notation of \cite{robertschmidt}. 
Here (IIb) corresponds to the Saito--Kurokawa type (P), where $K(p)$ fixed vector is 
unique up to scalar with plus Atkin--Lehner sign, 
while (I) corresponds to the general type (G), where there are two 
$K(p)$ fixed vectors, one is of Atkin--Lehner plus and the other is 
of Atkin--Lehner minus. This can be seen in \cite{robertschmidt} Table A.15.
(See also \cite{gritsenko}, \cite{schmidtiwahori}, \cite{schmidtSK}.)
So the dimensions of new forms in $S_{k,j}^{\pm}(K(p))$ in the sense of 
\cite{robertschmidt} is given by 
\begin{align*}
& \dim S_{k,j}^{+}(K(p)) -\dim S_{k,j}(Sp(2,\Z))  \\
& \dim S_{k,j}^{-}(K(p)) -(\dim S_{k,j}(Sp(2,\Z))-
\delta_{k,even}\dim S_{2k-2}(SL_2(\Z))). 
\end{align*}
Now, \cite{vanhoften} and \cite{weissauer} claim that the general part
(i.e. non-lift part) of  $\fM_{k+j-3,k-3}(U_{npg}(p))$ and 
that of new forms  
of $S_{k,j}(K(p))$ have Hecke equivariant bijection.
By \cite{dummigan} Theorem 10.1 (i), the general new part of 
$S_{k,j}^{\pm}(K(p))$ corresponds to 
the general part of $\fM_{k+j-3,k-3}^{\mp}(U_{npg})$
(where double sign corresponds). So we have to see the Atkin--Lehner signs 
of the remaining lifting part.
When $k$ is odd, then there is no Saito Kurokawa lift from $S_{2k-2}(SL_2(\Z))$ 
to $S_k(Sp(2,\Z))$, but there exists an Ihara lift from $S_{2k-2}(SL_2(\Z))$ 
to algebraic modular forms 
and this lift is injective to $\fM_{k-3,k-3}(U_{npg}(p))$ by  
\cite{vanhoften} Theorem 8.2.1 (3) and \cite{weissauer} Proposition 12.2
(, though not known if this lift is exactly as constructed in \cite{ihara}, 
\cite{iharaibu}, \cite{ibuparavector}).  
By virtue of \cite{dummigan} Theorem 10.1 (ii), we see that this appears in 
$\fM_{k-3,k-3}^{+}(U_{npg}(p))$. Also, when $k=3$ and $j=0$, 
$\fM_{0,0}(U_{npg}(p))$ contains a constant function of $G_A$ which 
belongs to Atkin--Lehner plus, and 
this does not correspond to a cusp form.
So we must subtract 
\[
\delta_{k,odd}\dim S_{2k-2}(SL_2(\Z))+\delta_{j0}\delta_{k3}.  
\]
from $\fM_{k+j-3,k-3}^{+}(U_{npg})$.
Now by \cite{vanhoften} Theorem 8.2.1 (3) and \cite{weissauer} Proposition 12.1, 
there is an injective Yoshida lifting from 
$S_{j+2}^{new}(\Gamma_0(p))\times S_{2k+j-2}(SL_2(\Z))$ to 
$\fM_{k+j-3,k-3}(U_{npg}(p))$. 
The signature of the image of the Yoshida lift is given in 
\cite{dummigan} Theorem 10.1 (iii). This gives the following injection
\begin{align*}
& S_{j+2}^{+,new}(\Gamma_0(p))\times S_{2k+j-2}(SL_2(\Z))\hookrightarrow 
\fM_{k+j-3,k-3}^{-}(U_{npg}(p)), \\
&  S_{j+2}^{-,new}(\Gamma_0(p))\times S_{2k+j-2}(SL_2(\Z))
\hookrightarrow \fM_{k+j-3,k-3}^{+}(U_{npg}(p)).
\end{align*}
Note that there is no Yoshida lift in $S_{k,j}(K(p))$ by \cite{schmidtpacket}
Lemma 2.2.1.
The only remaining part now is a lift from $S_{2k-2}^{\pm,new}(\Gamma_0(p))$. 
For paramodular forms, this is the Gritsenko lift 
(paramodular version of Saito-Kurokawa lift) from 
$S_{2k-2}^{\epsilon,new}(\Gamma_0(p))$ to $S_k(K(p))$. 
Here we must have $\epsilon=(-1)^k$, since originally the elliptic modular forms here
should correspond with Jacobi forms of index $p$ of weight $k$ and the 
sign of the functional equation should be $-1=(-1)^{k-1}\epsilon$.  
The Atkin--Lehner sign of Gritsenko lifts is $(-1)^k$ (\cite{gritsenko} (12), 
\cite{schmidtSK} Theorem 5.2 (i)). 
So we have the following injections.
\begin{align*}
& S_{2k-2}^{+,new}(\Gamma_0(p))\hookrightarrow S_k^{+}(K(p)) \quad 
\text{ for even $k$},\\
& S_{2k-2}^{-,new}(\Gamma_0(p))\hookrightarrow S_k^{-}(K(p)) \quad 
\text{ for odd $k$}.
\end{align*}
(By the way, the fact that the sign of the Saito-Kurokawa lift from 
$S_{2k-2}(SL_2(\Z))$ to $S_{k}(K(p))$ is plus can be proved also 
by \cite{gritsenko} Theorem 3 Corollary.)
For the compact twist, we have also a lift from $S_{2k-2}^{(-1)^k,new}(\Gamma_0(p))$
to $\fM_{k-3,k-3}(U_{npg}(p))$ by \cite{weissauer} Proposition 12.2. 
The Atkin--Lehner sign is determined by \cite{dummigan} Theorem 10.1 (ii), and we have 
\begin{align*}
& S_{2k-2}^{+,new}(\Gamma_0(p))\hookrightarrow \fM_{k-3,k-3}^{-}(U_{npg}(p)) \quad 
\text{ for even $k$,} \\
& S_{2k-2}^{-,new}(\Gamma_0(p))\hookrightarrow \fM_{k-3,k-3}^{+}(U_{npg}(p))\quad 
\text{ for odd $k$}.
\end{align*}
Note that the sign is reversed compared with paramodular case. 
So in the comparison between dimensions of 
$S_{k,j}^{\pm}(K(p))$ and $\fM_{k+j-3,k-3}^{\mp}(U_{npg}(p))$,
this part apparently need not appear.
Gathering above considerations, Theorem 1.2 is proved.

\section{Proof of Theorem \ref{compactdim}} 
\label{proof1}
First we explain the definition of the non-principal genus. 
From now on, we consider only the case $n=2$ for simplicity.
Let $B$ be the definite quaternion algebra of 
prime discriminant $p$ as before. We fix a maximal order $O$ of $B$.
A choice of $O$ is not essential but we fix $O$ that contains 
a prime element $\pi$ with 
$\pi^2=-p$ for the sake of simplicity (such $O$ always exists).
A $\Z$ lattice $L\subset B^2$ (a free $\Z$ submodule of $B^2$ of rank $8$) 
is said to be a left $O$ lattice if $L$ is a left $O$ module.
As in Shimura \cite{shimura}, we define a norm $N(L)$ of $L$ as 
the two sided $O$ ideal spanned by $h(x,y)$ for $x$, $y\in L$. 
A left $O$ lattice $L$ is said to be maximal if  any $O$-lattice $M$ 
with $L\subset M$
and $N(M)=N(L)$ satisfies $M=L$. 
For any left $O$ lattice $L$ and any prime $q$, we write 
$L_q=L\otimes_{\Z}\Z_q$. 
A genus $\cL$ is a set of left $O$ lattices in $B^2$ such that for any 
$L_1$, $L_2 \in \cL$, we have $L_{1,q}=L_{2,q} g_q$ for some $g_q\in G_q$ for any prime $q$.
If a lattice in $\cL$ is maximal, then so is any other lattice in $\cL$.   
The set of norm $N(L)$ for $L\in \cL$ is determined up to 
a multiplication by $\Q^{\times}$. In our case of discriminant $p$,
there are two genera of maximal lattices, one is called the principal genus $\cL_{pg}$ containing 
$O^2$, and the other  is the one called non-principal genus $\cL_{npg}$
containing a maximal lattice $L$ with $N(L)=\pi O=O \pi$ (See \cite{shimura} Proposition 4.6).

Now we fix any genus $\cL$.  
For $L\in \cL$, a set of lattices $\{Lg;g\in G\}\subset \cL$ is called a class. 
The number of classes in $\cL$ is called the class number of $\cL$.   
For a fixed left $O$ lattice $L\in \cL$ and any prime $q$, we put 
\[
U(L_q)=\{g_q\in G_q; L_qg_q=L_q\},
\]
and define an open subgroup $U(L)$ of $G_A$ by 
\[
U=U(L)=G_{\infty}\prod_{q}U(L_q).
\]
For any $g_A=(g_v)\in G_A$, we define a left $O$ lattice $Lg_A\subset B^2$ by 
\[
Lg_A=\bigcap_{v<\infty}(L_vg_v\cap B^2).
\]
Then by $L\rightarrow Lg_A$, we see that 
the class number $H$ is equal to the number of double cosets in  
$U\backslash G_A/G$, and if we write the representatives of double cosets as 
\begin{equation}\label{doublecoset}
G_A=\coprod_{i=1}^{H}Ug_i G \qquad (disjoint),
\end{equation}
then the set $\{Lg_i;i=1,\ldots, H\}$ gives a complete set of representatives of 
classes in $\cL$. If we put $\Gamma_i=g_i^{-1}Ug_i\cap G$, then 
these are finite groups and are (metric preserving) 
automorphism groups of $Lg_i$ for 
$1\leq i\leq H$, respectively. 
Let $\rho_{f_1,f_2}$ be the irreducible representation of $Sp(2)$ corresponding to 
$(f_1,f_2)$ with $f_1\equiv f_2 \bmod 2$ with $f_1\geq f_2\geq 0$. We define 
$\fM_{f_1,f_2}(U)$ as in the introduction and call an element of $\fM_{f_1,f_2}(U)$ 
an algebraic modular form of weight $\rho_{f_1,f_2}$ belonging to $U$. 
For later use, we review another non-adelic realization of $\fM_{f_1,f_2}(U)$ (see \cite{hashimoto} p. 230).  
Let $V$ be a representation space of $\rho_{f_1,f_2}$.
Then we have 
\[
\fM_{f_1,f_2}(U)\cong \oplus_{i=1}^{H}V^{\Gamma_i}
\qquad (\text{direct sum}), 
\]
where we put 
\[
V^{\Gamma_i}=\{v\in V; \rho_{f_1,f_2}(\gamma)v=v \text{ for all }\gamma \in \Gamma_i\}.
\]
The above isomorphism is given by the mapping 
\[
\fM_{f_1,f_2}(U)\ni f \rightarrow \sum_{i=1}^{H}\rho(g_i)^{-1}f(g_i)\in 
\oplus_{i=1}^{H}V_i^{\Gamma_i}.
\] 
Next we consider an action of the Hecke algebra. 
For $g \in G_A$, we define an action of $UgU=\coprod_{j}z_jU$ (disjoint) 
on $\fM_{f_1,f_2}(U)$ by 
\[
(R_{f_1,f_2}(UgU)f)(x)=\sum_{j}\rho_{f_1,f_2}(z_j)f(z_j^{-1}x).
\]
(Note that the representation $\rho_{f_1,f_2}$ is defined on 
$G_A$ by \eqref{definerep}.)
We interpret the action of $UgU$ on $\fM_{f_1,f_2}(U)$ into an $H\times H$ matrix action 
on $\oplus_{i=1}^{H}V^{\Gamma_i}$. 
We put $T_{ij}=G\cap g_i^{-1}UgUg_j$. Then it is clear by definition that we have
$\Gamma_iT_{ij}\Gamma_j=T_{ij}$. So we regard $T_{ij}$ as a formal sum 
\[
T_{ij}=\sum_{h}\Gamma_i h\Gamma_j=\sum_{m}h_m\Gamma_j.
\] 
Then this can be regarded as an operator of $V^{\Gamma_j}$ 
to $V^{\Gamma_i}$ by defining the action on 
$v_j \in V^{\Gamma_j}$ by 
\begin{equation}\label{Tijaction}
T_{ij}v_j = \sum_{m}\rho_{f_1,f_2}(h_m)v_j, \qquad (T_{ij}=\sum_{h}\Gamma_ih\Gamma_j = \sum_{m}h_m\Gamma_j).
\end{equation}
Then by \cite{hashimoto} Lemma 1, the action of $UgU$ is identified 
with the action of the matrix $(T_{ij})_{1\leq i, j \leq H}$ on $\oplus_{i=1}^{H}V^{\Gamma_i}$ 
by $(v_1,\ldots,v_H)\rightarrow (\sum_{j=1}^{H}T_{ij}v_j)_{1\leq i\leq H}$. Here we may also write   
\[
T_{ij}=\sum_{a\in G\cap g_i^{-1}UgUg_j/\Gamma_{j}}\rho_{f_1,f_2}(a)|V^{\Gamma_j}.
\]
We describe the trace of the action of $UgU$ in this setting.
For $x \in G$, we define $\rho_{f_1,f_2}(x)$ by \eqref{definerep}
through the diagonal embedding $G \rightarrow G_A$. 
We denote by $Tr(\rho_{f_1,f_2}(x))$ the trace of the representation
$\rho_{f_1,f_2}(x)$ on whole $V$.  
The following Lemma \ref{discrettrace} and Corollary \ref{maincor} are almost trivial by definition and has been used many times 
for the calculation of dimension formulas of algebraic modular 
forms such as \cite{hashimotoibu}, \cite{ibuparamodular}, \cite{hashimotoibudimII}, \cite{ibuparavector}, \cite{ibumiddle}, 
but to see its connection to $SO(5)$ clearly, we explain some details
for safety. 
\begin{lemma}\label{discrettrace}
For $i=1$, \ldots, $h$, define $T_{ii}=G\cap g_i^{-1}UgUg_i$ as before. 
Then we have 
\[
Tr(R_{f_1,f_2}(UgU))=\sum_{i=1}^{H}\frac{\sum_{x\in T_{ii}}
Tr(\rho_{f_1,f_2}(x))}{\#(\Gamma_i)}, 
\]
where $\#(\Gamma_i)$ is the cardinality of $\Gamma_i$.
\end{lemma}
\begin{proof}
Notation being as in \eqref{Tijaction}, 
 the action of $T_{ii}$ on $v_i \in V^{\Gamma_i}$ is given by 
\begin{align*}
T_{ii}v_i & =\sum_{m}\rho_{f_1,f_2}(h_m)v_i = \sum_{m}\sum_{\gamma \in \Gamma_i}
\frac{\rho_{f_1,f_2}(h_m\gamma)v_i}{\#(\Gamma_i)}
= \frac{\sum_{x\in T_{ii}}\rho_{f_1,f_2}(x)v_i}{\#(\Gamma_i)}.
\end{align*}
Now we must consider a relation 
between the trace of the above 
action of $T_{ii}$ on $V^{\Gamma_i}$ and 
$Tr(\rho_{f_1,f_2}(x))$ on $V$.
We define an action of $T_{ii}$ on $v\in V$ by 
\[
T_{ii}v=  \sum_{x\in T_{ii}}\rho_{f_1,f_2}(x)v.
\]
Since $\Gamma_iT_{ii}=T_{ii}$, we see that 
$T_{ii}v\in V^{\Gamma_i}$ for any $v \in V$.
So considering the representation matrix of $T_{ii}$ 
with respect to a basis of $V$ obtained by prolonging 
a basis of $V^{\Gamma_i}$, it is clear that
\[
\sum_{x\in T_{ii}}Tr(\rho_{f_1,f_2}(x)|V)
=
\sum_{x\in T_{ii}}Tr(\rho_{f_1,f_2}(x)|V^{\Gamma_i}).\qedhere
\]
\end{proof}
By definition \eqref{definerep} of the representation $\rho$, 
if we put $\hat{g}=g/\sqrt{n(g)}$ for $g \in G$, then 
we have 
$Tr(\rho_{f_1,f_2}(g))=Tr(\rho_{f_1,f_2}(\hg))$.
Here if $\Phi(x)$ is the principal polynomial of $g\in G$
(defined as the characteristic polynomial of 
$g\in G\subset M_2(B)\subset M_4(\C)$), 
the principal polynomial of $\hg$ is given by 
$\phi(x)=n(g)^{-2}\Phi(\sqrt{n(g)}x)$.  
It is well known that the character 
$Tr(\rho_{f_1,f_2}(\hg))$ depends only on the principal 
polynomial $\phi(x)$ of $\hg$. 
For any such polynomial $\phi(x)$, 
we write $\chi_{f_1,f_2}(f)=Tr(\rho_{f_1,f_2}(\hg))
=Tr(\rho_{f_1,f_2}(g))$ where $g\in G$ is as above.
Since we assumed $\rho_{f_1,f_2}(\pm 1_2)=1$, 
we have $Tr(\rho_{f_1,f_2}(\hg))=Tr(\rho_{f_1,f_2}(-\hg))$,  so 
we have $\chi_{f_1,f_2}(f(x))=\chi_{f_1,f_2}(f(-x))$. 
Let $G(\phi)$ be the set of all elements $g$ of $G$ such that 
$\phi(x)$ or $\phi(-x)$ is the principal polynomial of $\hg$.
We also put 
\[
Tr(R_{f_1,f_2}(UgU),\phi)=\sum_{i=1}^{H}
\frac{\sum_{g\in T_{ii}\cap G(\phi)}Tr(\rho_{f_1,f_2}(g))}{\#(\Gamma_i)}. 
\]
Then by Lemma \ref{discrettrace}, 
we have the following corollary. 
\begin{cor}\label{maincor}
We have 
\begin{equation}\label{corshiki}
Tr(R_{f_1,f_2}(UgU),\phi)=\chi_{f_1,f_2}(\phi)\times 
\sum_{i=1}^{h}\frac{\#(T_{ii}\cap G(\phi))}{\#(\Gamma_i)}.
\end{equation}
\end{cor}
This corollary is very useful since RHS is a product of the term depending on the representation and the term independent of the representation.
We have $\chi_{0,0}(\phi)=1$ for any $\phi$, so  
the general formula is reduced to the case of the 
trivial representation as far as the formula for $Tr(R_{0,0}(UgU))$ 
is given as a sum of 
explicit terms for each $\phi$. (This is usually true for any trace formula).   

Now we specialize our consideration to the case 
when $\cL=\cL_{npg}$. 
We consider the double coset 
$R(\pi)=U_{npg}(p)\pi U_{npg}(p)=U_{npg}(p)\pi=\pi U_{npg}(p)$,
where $\pi$ is identified with an element of $G_A$ by the diagonal 
embedding.
Our aim is to give a formula for $Tr(R_{f_1,f_2}(\pi))$ on $\fM_{f_1,f_2}(U_{npg}(p))$. 
In this case, for any element $x\in T_{ii}$, we have $n(x)=p$.
By Corollary \ref{maincor}, the main part of the calculation of
$Tr(R_{f_1,f_2}(\pi))$ for $\rho_{f_1,f_2}$ is to give the value of RHS of \eqref{corshiki} for $(f_1,f_2)=(0,0)$ for each $f$.  
This calculation is based on two things. One is 
the relation $2T=H+Tr(R_{0,0}(\pi))$ proved in \cite{ibutypenumber} Theorem 3.6 
between the class number $H$ of $\cL_{npg}$, the type number $T$ of
$\cL_{npg}$(the definition will be reviewed soon), and $Tr(R_{0,0}(\pi))$.  
The other is an equality between the type number $T$ of $\cL_{npg}$ and the class 
number of quinary lattices of some genus of $\det=2p$ proved in \cite{ibuquinary}
Theorem 4.5 for $p\neq 2$. 
(The case $p=2$ will be treated separately in section \ref{p23}.)   
In \cite{ibuquinary}, our calculation of $T$ 
depends on the class number formula of Asai in \cite{asai}.
But this time, we should be more careful since 
we must calculate the RHS of Corollary \ref{maincor} for each principal 
polynomial. So we will compare the class number formula in 
\cite{asai} and 
a formula of $2T$ given by $H+Tr(R_{0,0}(\pi))$ 
for the Hecke operator of $G$ for each principal polynomial. 
For that purpose, we will review some details on the type number below. 

We fix a set of representatives of 
$U_{npg}(p)\backslash G_A/G$ as in \eqref{doublecoset} and 
for a fixed 
$L\in \cL_{npg}$, put $L_i=Lg_i$ ($i=1, \ldots, H$). 
We define the right order $R_i$ of $L_i$ by 
\[
R_i = \{z\in M_2(B); L_i z \subset L_i\}.
\]
The classical meaning of the type number is the number of isomorphism classes 
of maximal orders in an algebra. But in our case, 
all maximal orders in $M_2(B)$ are conjugate to $M_2(O)$ 
since the class number of $M_2(B)$ is $1$ by the strong approximation theorem of $SL_2(B)$ (See \cite{kneser}, \cite{kneser2}). 
But here, instead of $GL_2(B)$ conjugacy, we consider 
the $G$ conjugacy of $R_i$. The ($G$-)type number $T$ of $\cL_{npg}$ 
is defined to be a number of $G$-conjugacy classes in $\{R_i\}_{1\leq i\leq H}$. 
In \cite{ibutypenumber} we proved that $2T=H+Tr(R_{0,0}(\pi))$.
So we have $T\leq H\leq 2T$.  
For each principal polynomial $\Phi$ of 
an element appearing in 
$g_i^{-1}(U_{npg}(p)\cup U_{npg}(p)\pi)g_i\cap G$ for some $i$, 
put $\phi(x)=n(g)^{-2}\Phi(x/\sqrt{n(g)})$ as before. 
Then the contribution to $T$ of the ``$\phi(x)$ and $\phi(-x)$-part'' 
in $T=(H+Tr(R_{0,0}(\pi)))/2$ 
is defined to be  
\begin{equation}\label{fpart}
T(\phi)=
\frac{1}{2}
\sum_{i=1}^{H}\frac{\#(\Gamma_i\cap G(\phi))+\#(g_i^{-1}U_{npg}(p)\pi g_i \cap G(\phi))}{\#(\Gamma_i)},
\end{equation} 
where we write $\Gamma_i=g_i^{-1}U_{ngp}(p)g_i\cap G$. 
Now we will compare this to the class number formula
of quinary lattices. 
Let $\cM=\cM(1,p)$ be the genus of lattices with determinant $2p$ in 
the $5$ dimensional positive definite quadratic space $W$ defined in 
\cite{ibuquinary} p. 215. We have shown in \cite{ibuquinary} 
that 
the class number of $\cM$ is the type number $T$ 
of $\cL_{npg}$ if $p\neq 2$. 
The class number formula for $\cM$ can be 
explained as follows. We denote by $M_i$ ($i=1,\ldots, T$) 
a complete set of representatives of classes in $\cM$. 
Although Asai in \cite{asai} used the orthogonal group $O(W)$ to define a class, 
we have $O(W)=SO(W)\cup (-1_5)SO(W)$ for the special 
orthogonal group $SO(W)$, and 
the class number for both $O(W)$ and $SO(W)$
are the same and the formulas are identical, 
so we explain the $SO(W)$ formulation.
Let $M_1$, \ldots, $M_T$ be a complete set of 
representatives of classes in $\cM$.
We denote by $Aut(M_i)$ the group of automorphisms 
of $M_i$ in $SO(W)$. Then we have the trivial identity
\[
T=\sum_{i=1}^{T}\frac{\#(Aut(M_i))}{\#(Aut(M_i))}.
\]
Let $\hh(x)$ be a principal polynomial of an element of $SO(W)$ of degree $5$. 
It is of the shape
\[
\hh(x)=(x-1)h(x)
\]
for some degree $4$ monic reciprocal polynomial $h(x)$. 
We denote by $SO(W,\hh)$ the set of elements of 
$SO(W)$ whose principal polynomial is $\hh(x)$. 
We put 
\[
T(\hh)=\sum_{i=1}^{T}\frac{\#(Aut(M_i)\cap SO(W,\hh))}{\#(Aut(M_i))}.
\]
Naturally we have $T=\sum_{\hh}T(\hh)$. 
In Asai \cite{asai}, he gave a formula for $T(\hh)$ for each $\hh$ 
as a contribution of the case $C_{\pm i}$ where $C_{\pm i}$ means 
a principal polynomial of $\pm \tg$ of some element $\tg\in O(W)$. 
Here one of $\{\tg,-\tg\}$ is an element of $SO(W)$ and Asai's formula is the same 
as the contribution of $SO(W)$ belonging to one of $C_i$ or $C_{-i}$.
Now we will interpret each $T(\hh)$ into the ^^ ^^ $\phi$-part" $T(\phi)$ of 
$T$ defined by \eqref{fpart}. To explain this, we review some parts of \cite{ibuquinary}. 
The even Clifford algebra $C_2(W)$ of $W$ can be identified 
with $M_2(B)$, $W$ with a linear subspace of $M_2(B)$ over $\Q$, and the even Clifford group
$\Gamma_2$ with $G$. The inner automorphism 
$W\ni w \rightarrow g^{-1}wg \in W$ for $g \in G$ 
induces an isomorphism $G/\{\Q^{\times}1_2\}\cong SO(W)$
(\cite{ibuquinary} p. 210).
To compare with the formulation by $G$, we must describe $Aut(M_i)$
in terms of $G$. Fix $L$ to be a representative of $\cL_{npg}$
and $R$ the right order of $L$.
Then we have 
\begin{lemma}[\cite{ibuquinary} Lemma 4.1 and Corollary 4.4]\label{keylemma}
There exists a lattice $M\in \cM$ such that 
for any $g_A=(g_v)\in G_A$, we have 
$g_vM_vg_v^{-1}=M_v$ if and only if 
$g_vR_vg_v^{-1}=R_v$ for any finite place $v$ of $\Q$,
where we put $M_v=M\otimes_{Z}\Z_v$ and 
$R_v=R\otimes_{\Z}\Z_v$.
\end{lemma}
For a representative $g_i=(g_{i,v})$ of the double coset in 
\eqref{doublecoset}, we put 
\[
R_i=g_i^{-1}Rg_i=\cap_{v<\infty}(g_{i,v}^{-1}R_vg_{i,v} \cap M_2(B)).
\]
This is the right order of $L_i=Lg_i$. 
Changing numbers $i$ if necessary, we assume that $R_1$, \ldots, $R_T$ are representatives of 
types (i.e. $G$ conjugacy classes). 
Then $M_i=g_i^{-1}Mg_i$ ($i=1,\ldots, T$) are also the 
representatives of classes in $\cM$. 
By Lemma \ref{keylemma}, we see that 
any element of $Aut(M_i)$ comes from $g\in G$ 
with $gg_i^{-1}Rg_ig^{-1}=g_i^{-1}Rg_i$. This means that 
$Rg_igg_i^{-1}$ 
is a two sided ideal of $R$, and it is well known that 
any two sided ideal of $R_v$ is spanned by $\Q_v^{\times}$ 
and besides $\pi$ if $v=p$ up to $R_v^{\times}$.
By definition we have 
$U_{npg}(p)=G_{\infty}\prod_{v<\infty}(R_v^{\times}\cap G_v)$,
and since $\Q_A^{\times}=
\Q^{\times}\R^{\times}_+\prod_{v<\infty}\Z_v^{\times}$
and $(\R^{\times} 1_2)\prod_{v<\infty}(\Z_v^{\times} 1_2) \subset U_{npg}(p)$, 
this means that $g_i(mg)g_i^{-1}\in U_{npg}(p)\cup U_{npg}(p)\pi$
for some $m\in \Q^{\times}$.  Writing the natural projection of $G$ to 
$G/{\Q^{\times} 1_2}\cong SO(W)$ by $\iota$,  
we have $\iota(m1_2)=1_5$,
so we have 
\[
Aut(M_i)=\iota(g_i^{-1}(U_{npg}(p)\cup U_{npg}(p)\pi)g_i \cap G).
\]
Here by definition we have 
\[
g_i^{-1}U_{npg}(p)g_i\cap G = \Gamma_i,
\]
so $\iota(\Gamma_i)$ is always contained in $Aut(M_i)$. 
The problem is the part $g_i^{-1}U_{npg}(p)\pi g_i\cap G$. 
To explain this part more clearly, we review the 
relation of the type number and the class number of $\cL_{npg}$ written in \cite{ibuquinary}.
We write the right orders of $L_i=Lg_i$ by $R_i$ 
($i=1$, \ldots, $H$). Assume that $gR_ig^{-1}=R_j$ 
for some $g \in G$. By the same reason we explained above, 
this means that 
\[
g_i(mg)g_j^{-1} \in U_{npg}(p) \cup U_{npg}(p)\pi
\]
for some $m\in \Q^{\times}$. So $R_i$ is $G$ conjugate to 
$R_j$ if and only if 
\begin{equation}\label{nonemptyij}
g_i^{-1}(U_{npg}(p)\cup U_{npg}(p)\pi)g_j\cap G \neq \emptyset.
\end{equation}
Now we fix $i$ and see which $j\neq i$ satisfies \eqref{nonemptyij}.
If $i\neq j$, then $g_i^{-1}U_{npg}(p)g_j\cap G=\emptyset$  
since $\{g_i\}_{1\leq i\leq h}$ is a 
set of representatives of $U\backslash G_A/G$.  
The condition 
$g_i^{-1}U_{npg}(p)\pi g_j \cap G\neq \emptyset$ is 
equivalent to 
\begin{equation}\label{ij}
\pi U_{npg}(p)g_jG=U_{npg}(p)\pi g_j G=U_{npg}(p)g_iG. 
\end{equation}
Since 
\[
U_{npg}(p)g_iG\subset G_A=\pi G_A=
\coprod_{l=1}^{H}\pi U_{npg}(p)g_lG=\coprod_{l=1}U_{npg}(p)\pi g_lG \quad (disjoint),
\] 
there exists unique $j$ that satisfies \eqref{ij}.
This means that we have two cases. 
The first one is the case that $j=i$ in \eqref{ij} and we have  
\[
U_{npg}(p)\pi g_i G=U_{npg}(p)g_i G.
\]
This means that 
$g_i^{-1}U_{npg}(p)\pi g_i\cap G \neq \emptyset$.
Besides, if 
$g$, $g' \in g_i^{-1}U_{npg}(p)\pi g_i\cap G$, then 
writing $g=g_i^{-1}u\pi g_i$ and $g'=g_i^{-1}u'\pi g_i$
for $u$, $u'\in U_{npg}(p)$, 
we have $g^{-1}g'=g_i^{-1}\pi^{-1}u^{-1}u'\pi g_i \in G$, but
we have $\pi^{-1} u^{-1}u' \pi\in \pi^{-1} U_{npg}(p)\pi=U_{npg}(p)$, 
so $g^{-1}g'\in \Gamma_i$. This means that 
\[
g_i^{-1}U_{npg}(p)\pi g_i \cap G=g_0\Gamma_i
\]
for some $g_0\in G$. We may assume that  
the set of such $i$ is $\{1,\ldots,t\}$. Then for these $i$, we have 
\[
Aut(M_i)=\iota(\Gamma_i)\cup \iota(g_0\Gamma_i).
\]
This is a disjoint union in $SO(W)$. Indeed, for any element $g_0\in G$ with $n(g_0)=p$, we have 
$\iota(g_0)\not\in \iota(G^1)$, because for any element $g_0'\in \Q^{\times}G^1$,
we have $n(g_0')\in (\Q^{\times})^2$. So we have $\#(Aut(M_i))=2\#(\iota(\Gamma_i))$. 
Next we consider the second case that $j\neq i$ in \eqref{ij}.
This case, we have a pair of right orders $R_i$ and $R_j$ 
which are $G$ conjugate 
one another, and we may take $(i,j)=(i,i+(H-t)/2)$
for $i=t+1,\ldots,(H-t)/2$. Here we have 
\[
Aut(M_i)=\iota(\Gamma_i).
\]
We also have $T=t+(H-t)/2$ and 
$t=Tr(R_{0,0}(\pi))$.  
Now for a principal polynomial $\hh$, we have 
\[
2T(\hh)=2\sum_{i=1}^{T}\frac{\#(\iota(\Gamma_i)\cap SO(W,\hh))+\#(\iota(g_i^{-1}U_{npg}(p)\pi g_i\cap G)\cap SO(W,\hh))}{\#(\iota(\Gamma_i))+\#(\iota(g_i^{-1}U_{npg}(p)\pi g_i\cap G))}.
\]
For $i=1$, \ldots, $t$, the denominator is $2\#(\iota(\Gamma_i))$. For $i=t+1$, \ldots, $T-t$, the denominator is $\#(\iota(\Gamma_i))$. 
Besides, when $gg_iRg_i^{-1}g^{-1}=g_jRg_j^{-1}$ for some $(i,j)$ with 
$1\leq i\neq j\leq H$ and $g \in G$,   
we have  $\Gamma_i\cong \Gamma_j$, so  
$\#(\iota(\Gamma_i))=\#(\iota(\Gamma_j))$. 
Now we compare principal polynomials $f(x)$ 
of $\hg=g/\sqrt{n(g)}$ for $g \in G$ and $\hh(x)$ of $\iota(g)$. 
If the eigenvalues of an element of $Sp(2)$ is $\epsilon_1$, 
$\epsilon_2$, $\epsilon_1^{-1}$, $\epsilon_2^{-1}$,
then eigenvalues of the image in $SO(5)$ is $1$, 
$\epsilon_1\epsilon_2$, $\epsilon_1\epsilon_2^{-1}$,
$\epsilon_1^{-1}\epsilon_2$, $\epsilon_1^{-1}\epsilon_2^{-1}$, so 
for the principal polynomial
\[
f(x)=x^4+c_1x^3+c_2x^2+c_1x+1
\]
of an element $\hg\in G^1_{\infty}$, the principal polynomial 
of $\iota(g)\in SO(W)$ is given by 
\[
\hf(x)=(x-1)h(x), \quad 
h(x)=x^4-(c_2-2)x^3+(c_1^2-2c_2+2)x^2-(c_2-2)x+1.
\]
Here for $\phi(x)$ and $\phi(-x)$, we have the same $\hf(x)$.
This is clear from the above calculation and also by 
$Ker(\iota|Sp(2))=\{\pm 1_2\}$. We have $\#(\Gamma_i)=2\#(\iota(\Gamma_i))$ and 
\begin{align*}
\#(\Gamma_i\cap G(\phi))& =2(\#(\iota(\Gamma_i)\cap SO(W,\hf)),\\ 
\#(g_i^{-1}U_{npg}(p)\pi g_i\cap G(\phi))& =2\#(\iota(g_i^{-1}U_{npg}(p)\pi g_i\cap G)\cap SO(W,\hf)),
\end{align*} 
so we can rewrite the above formula for $2T(\hf)$ as 
\begin{lemma}\label{quinaryvshermitian}
We have
\[
2T(\hf)=
\sum_{i=1}^{H}\left(\frac{\#(\Gamma_i \cap G(f))}
{\#(\Gamma_i)}+\frac{\#(g_i^{-1}U_{npg}(p)\pi g_i \cap G(f))}
{\#(\Gamma_i)}\right)
=2T(\phi).
\] 
\end{lemma}
This Lemma means that $T(\phi)$ can be obtained from  
Asai's formula for $T(\hf)$. 
The principal polynomials of elements of $\Gamma_i$ and 
of elements of $g_i^{-1}U_{npg}(p)\pi g_i\cap G$ are 
different, because
the constant terms $n(g)^2$ are different. But 
sometimes the principal polynomials for $\hg$ are the same.
So both contribute to the same $T(\hf)$. 
But this does not matter, since 
the character of the elements depend only on $\hf(x)$, 
or $\phi(\pm x)$. 

For more precise calculations, it would be safer to give a list of tables of principal polynomials $\Phi(x)$, $\phi(x)$ and $\hf(x)$ of 
$g\in G$, $\hg$, and $\iota(g)$.   
We consider possible principal polynomials of elements in 
$U_{npg}(p)\cup U_{npg}(p)\pi$.
The principal polynomials of elements in $\Gamma_i$ has been 
listed in \cite{hashimotoibu} I p. 590 
and also in section \ref{appendix} after Theorem \ref{nonprincipaldim}. 
So we see the case $U_{npg}(p)\pi$. 
For any prime $q\neq p$, we have $U(L_q)=G_q\cap GL_2(O_q)$.
At $p$, the prime element $\pi$ of $O$ we defined before is  
also a prime of $O_p=O\otimes_{\Z}\Z_p$.
We put 
\[
G_p^*=\left\{g\in M_2(B_p); g\begin{pmatrix} 0 & 1 \\ 1 & 0 \end{pmatrix} g^*=n(g)\begin{pmatrix} 
0 & 1 \\ 1 & 0 \end{pmatrix}\right\}.
\]
Then we have $G_p\cong G_p^*$ by $GL_2(O_p)$ conjugation, 
and  
\[
U_{npg}(L_p)\cong G_p^{*}\cap \begin{pmatrix} O_p & \pi ^{-1}O_p \\ \pi O_p & O_p \end{pmatrix}^{\times}.
\]
Hence by the condition that $g$ is conjugate to $U_{npg}\pi$ and 
$g^*=pg^{-1}$, 
the principal polynomial $F(x)$ of an element $g$ appearing in $g_i^{-1}U\pi g_i\cap G$ 
should be of the form 
\[
\Phi(x)=x^4+pax^3+pbx^2+p^2ax+p^2
\]
for some rational integers $a$, $b$.
The possible principal polynomials $\Phi$ of $g \in U_{npg}(p)\pi$ 
and corresponding $\phi(x)$ are as in Table \ref{Upi}. 
\begin{table}[htb]
\caption{Polynomials for $U_{npg}\pi$}\label{Upi}
\begin{tabular}{lll}
$p$ & $\Phi(x)$ & $\phi(x)$ \\
general & $(x^2-p)^2$ & $(x-1)^2(x+1)^2$  \\
general & $(x^2+p)^2$ & $(x^2+1)^2$  \\
general & $x^4+p^2$ &  $x^4+1$ \\
general & $x^4-px^2+p^2$ & $x^4-x^2+1$ \\
general & $x^4+px^2+p^2$ & $x^4+x^2+1$ \\
p=5 & $x^4\pm 5x^3+15x^2\pm 25x+25$  & $x^4\pm \sqrt{5}x^3+3x^2\pm \sqrt{5}x+1$ \\
p=3 & $(x^2\pm 3x+3)^2$ & $(x^2\pm \sqrt{3}x+1)^2$ \\
p=3 & $(x^2\pm 3x+3)(x^2+3)$ & $(x^2\pm \sqrt{3}x+1)(x^2+1)$\\
p=2 & $(x^2\pm 2x+2)^2$ & $(x^2\pm \sqrt{2}x+1)^2$ \\
p=2 & $(x^2\pm 2x+2)(x^2+2)$ & $(x^2\pm \sqrt{2}x+1)(x^2+1)$\\
p=2 & $x^4\pm 2x^3+2x^2\pm 4x+4$ & $x^4\pm\sqrt{2}x^3+x^2\pm\sqrt{2}x+1$  
\end{tabular}
\end{table}

This list can be obtained by using the inequalities in the following lemma.
\begin{lemma}
We fix $p$. Then a polynomial  
\[
\Phi(x)=x^4+pax^3+pbx^2+p^2ax+p^2
\]
is a principal polynomial of an element $g$ of $G$
only if the following conditions are satisfied.
\begin{align*}
|a| & \leq 4/\sqrt{p} \\
b & \leq a^2p/4+2 \\
(pa^2-4)/2 & \leq  b \\
4pa^2 & \leq (b+2)^2
\end{align*}
In particular, for a fixed $p$, there are only finitely many integers $a$, $b\in \Z$
such that $\Phi(x)$ can be a principal  polynomial of an element of $G$. 
\end{lemma}  

\begin{proof}
Since $g\in G$ and $g/\sqrt{p}\in Sp(2)$, if we put 
\[
h(X)=p^{-2}\Phi(\sqrt{p}X)=X^4+a\sqrt{p}X^3+bX^2+a\sqrt{p}X+1,
\]
then the roots of $h(x)=0$ should be of absolute value $1$.  
For any root $x$ of $h(X)=0$, writing $y=x+x^{-1}$, we have 
\begin{equation}\label{shiki1}
y^2+a\sqrt{p} y + b-2=0.
\end{equation}
Since $|x|=1$, we see $y$ is a real number, so we have 
$a^2p-4(b-2)\geq 0$, which gives the second inequality. 
If we denote by $\alpha$ and $\beta$ the roots of \eqref{shiki1}, 
then these are real and the roots of the equation
\[
X^2-\alpha X+1=0 \text{ or } X^2-\beta X+1=0
\]
are imaginary or $\pm 1$. In both cases, we have 
$\alpha^2,\beta^2\leq 4$. 
So we have 
\[
8 \geq \alpha^2+\beta^2=pa^2-2(b-2),  \qquad 
0 \leq  (4-\alpha^2)(4-\beta^2)=(b+2)^2-4pa^2,
\]
which leads to the third and the fourth inequalities.
\end{proof}

Now assume that $a$, $b \in \Z$ satisfy 
the condition in the above lemma. For any prime, 
we have $|a|\leq 4/\sqrt{p}\leq 4/\sqrt{2}$ so $|a|=0$, $1$, $2$.
If $a=0$, then we have $-2\leq b\leq 2$.
If $(a,b)=(0,\pm 2)$, then we have 
$\Phi(x)=x^4\pm 2px^2 + p^2=(x^2\pm p)^2$. 
If $(a,b)=(0,\pm 1)$, then 
$F(x)=x^4\pm px^2+p^2$.
If $(a,b)=(0,0)$, then $\Phi(x)=x^4+p^2$.
If $|a|=1$, then by the second and the third condition, we have 
$p/2-2\leq b\leq p/4+2$, so we have $p\leq 16$. If $p=7$, then 
$1.5\leq b\leq 3.75$, so $b=2$ or $3$. But the last condition  
$28\leq (b+2)^2$ is not satisfied for $b=2$ and $3$.
In the same way, for $p=11$, we have $b=4$ and for $p=13$, we have 
$b=5$, violating the last condition. So we have $p=2$, $3$, or $5$.
For $p=5$, we have $b=1$, $2$, $3$. By $20\leq (b+2)^2$, we have 
$b=3$. So we have $\Phi(x)=x^4\pm 5x^3+15x^2\pm 5x+25$. 
For $p=3$, we have $b=0$, $1$, $2$. Since $12\leq (b+2)^2$, we have 
$b=2$. This gives $\Phi(x)=x^4\pm 3x^2+6x^2\pm 9x+9=(x^2\pm 3x+3)(x^2+3)$.  
For $p=2$, we have $b=-1$, $0$, $1$, $2$. By $8\leq (b+2)^2$, we have 
$b=1$ or $b=2$. Then $\Phi(x)=x^4\pm 2^3+2x^2\pm 4x+4$ for $b=1$ and 
$\Phi(x)=(x^2\pm 2x+x)(x^2+1)$ for $b=2$. 
Finally we see the case $|a|=2$. This can happen only when $p=2$ or $3$
since $2\leq 4/\sqrt{p}$. 
When $p=3$, then $b=4$ or $5$. Since $48\leq (b+2)^2$, we should have 
$b=5$. So $\Phi(x)=x^4\pm 6x^3+15x^2\pm 18x+9=
(x^2\pm 3x+3)^2$. 
When $p=2$, we have $b=2$, $3$, $4$, and by $32\leq (b+2)^2$, we have 
$b=4$. This means that 
$\Phi(x)=x^4\pm 4x^3+8x^2\pm 8x+4=(x^2\pm 2x+2)^2$. So we have 
only polynomials $\Phi(x)$ that we have listed.

Of course Table \ref{Upi} is a possible list and we are not claiming that these really 
occur always

Now we list all the possible principal polynomials 
appearing in the calculation of $H$ and $T$  
in the following table. Here, in the first column of Table \ref{principalSO5}, 
we write the 
principal polynomials $\phi(x)$ of $\hat{g}\in G^{1}_{\infty}$ coming from 
$g\in U_{npg}(p)$ and $U_{npg}(p)\pi$. The second column is 
the principal polynomial $\tilde{\phi}(x)$ of elements of $SO(5)$ 
corresponding to the image of $\hat{g}$ by the projection 
$Sp(2)\rightarrow SO(5)$. 
The notation $C_{\pm i}$ is Asai's notation in \cite{asai}
Lemma 4.16 to indicate principal polynomials.
As we mentioned, he used $O(5)$ formulation 
instead of $SO(5)$, which causes the
suffix $\pm$ of $C_{\pm i}$, and for $SO(5)$, we need 
one of $C_{i}$ or $C_{-i}$. 

\begin{table}[htb]
\caption{Table of principal polynomials of $SO(5)$}\label{principalSO5}
\begin{tabular}{lll}
$\phi(x)$, $Sp(2)$ & $\hf(x)$, $SO(5)$ & Asai \\
$(x\pm 1)^4$ &  $(x-1)^5$ & $C_{1}$\\
$(x-1)^2(x+1)^2$ &  $(x-1)(x+1)^4$ & $C_{2}$ \\
$(x^2+1)^2$        &  $(x-1)^3(x+1)^2$ & $C_{-3}$ \\
$(x\pm 1)^2(x^2+1)$ &  $(x-1)(x^2+1)^2$ & $C_{4}$ \\
$x^4\pm 2\sqrt{2}x^3+4x^2\pm 2\sqrt{2}x+1$  & $(x-1)^3(x^2+1)$ & $C_{5}$ \\
$(x\pm 1)^2(x^2\mp x+1)$ & $(x-1)(x^2+x+1)^2$ & $C_{6}$ \\
$(x^2\pm x+1)^2$ & $(x-1)^3(x^2+x+1)$ & $C_{7}$ \\
$(x\pm1)^2(x^2\pm x+1)$ & $(x-1)(x^2-x+1)^2$ & $C_{8}$ \\
$x^4\pm 2\sqrt{3}x^3+5x^2\pm 2\sqrt{3}x+1$ & $(x-1)^3(x^2-x+1)$ & $C_{9}$ \\
$x^4+1$ &  $(x-1)(x+1)^2(x^2+1)$ & $C_{-10}$ \\
$x^4-x^2+1$  & $(x-1)(x+1)^2(x^2+x+1)$ & $C_{-11}$ \\
$(x^2+x+1)(x^2-x+1)$ & $(x-1)(x+1)^2(x^2-x+1)$ & $C_{-12}$ \\
$x^4\pm \sqrt{2}x^3+2x^2\pm \sqrt{2}x+1$ & $(x-1)(x^4+1)$ & $C_{13}$ \\
$x^4\pm x^3+x^2\pm x+1$ &  $(x-1)(x^4+x^3+x^2+x+1)$ & $C_{14}$ \\
$x^4\pm \sqrt{5}x^3+3x^2\pm \sqrt{5}x+1$ & $(x-1)(x^4-x^3+x^2-x+1)$ & $C_{15}$ \\
$(x^2+1)(x^2\pm x+1)$ & $(x-1)(x^4-x^2+1)$ & $C_{16}$ \\
$x^4\pm \sqrt{6}x^3+3x^2\pm \sqrt{6}x+1$ & $(x-1)(x^2+1)(x^2-x+1)$ & $C_{17}$ \\
$x^4\pm \sqrt{2}x^3+x^2\pm \sqrt{2}x+1$ & $(x-1)(x^2+1)(x^2+x+1)$ & $C_{18}$ \\
$x^4\pm \sqrt{3}x^3+2x^2\pm \sqrt{3}x+1$ & $(x-1)(x^2+x+1)(x^2-x+1)$ & $C_{-19}$ 
\end{tabular}
\end{table}
The polynomials in the table below are those that do not appear in the usual 
dimension formula of the class number $H$ of the non-principal 
genus coming from elements of some $\Gamma_i$ 
and newly needed for the type number.
\begin{table}[htb]
\caption{Principal polynomials of $Sp(2)$}\label{principalSp2}
\begin{tabular}{lll}
$\hat{g}\in Sp(2)$ & $g\in G$ & Asai \\
$x^4\pm 2\sqrt{2}x^3+4x^2\pm 2\sqrt{2}x+1$ & $(x^2\pm 2x+2)^2$ & $C_5$ \\
$x^4\pm 2\sqrt{3}x^3+5x^2\pm 2\sqrt{3}x+1$ & $(x^2\pm 3x+3)^2$ & $C_9$ \\
$x^4\pm \sqrt{2}x^3+2x^2\pm \sqrt{2}x+1$ & $(x^2\pm 2x+2)(x^2+2)$ & $C_{13}$ \\
$x^4\pm \sqrt{5}x^3+3x^2\pm \sqrt{5}x+1$ & $x^4\pm 5x^3+15x^2\pm 25x+25$ & $C_{15}$ \\
$x^4\pm \sqrt{6}x^3+3x^2\pm \sqrt{6}x+1$ & $x^4+6x^3+18x^2+36x+36$ & $C_{17}$ \\
$x^4\pm \sqrt{2}x^3+x^2\pm \sqrt{2}x+1$ & $x^4\pm 2x^3+2x^2\pm 4x+4$ & $C_{18}$ \\
$x^4\pm \sqrt{3}x^3+2x^2\pm \sqrt{3}x+1$ & $(x^2\pm 3x+3)(x^2+3)$  & $C_{-19}$
\end{tabular}
\end{table}
As can be seen, these appear only when $p=2$, $3$ or $5$.
Actually, the contributions to $T$ from $C_9$, $C_{13}$, 
$C_{17}$, $C_{18}$ are known to be all zero in Asai \cite{asai} Theorem 4.17 
for $d=p$ with $p\neq 2$ in the notation there,
so we do not need these. For $p=2$, we do not use Asai's result anyway.
The case $p=2$, $3$ will be treated separately in the next 
section and the case $p=5$ is included in the proof below.
 
\begin{proof}[Proof of Theorem \ref{compactdim}]
Here we prove Theorem \ref{compactdim} under the assumption that $p\neq 2$, $3$.
By Corollary \ref{maincor}, we have 
\[
Tr(R_{f_1,f_2}(R(\pi),\phi)=\chi_{f_1,f_2}(\phi)Tr(R_{0,0}(R(\pi),\phi),
\]
so if we calculate $Tr(R_{0,0}(R(\pi),\phi)$ for each principal 
polynomial $\phi$, then we obtain the case of general $(f_1,f_2)$ by the formula
for $\chi_{f_1,f_2}(\phi)$ given later.
So we assume that $f_1=f_2=0$ and $\rho_{f_1,f_2}$ is the 
trivial representation. 
All we should do is to calculate a contribution to 
$Tr(R(\pi)):=Tr(R_{0,0}(\pi))=2T-H$
for each principal polynomial $\phi(\pm x)$ or $\hf(x)$.
When $p\neq 2$, the formula for $T(\phi)$ of $\cL_{npg}$ is equal
$T(\hf)$ by Lemma \ref{quinaryvshermitian}.
To calculate $T(\hf)$, we use the class number formula  
given in Asai \cite{asai} Theorem 4.17. On the other hand, contribution of 
each conjugacy classes to the class number $H$ has been given in 
Theorem in \cite{hashimotoibu} II (which is reproduced 
in the appendix for prime discriminant case.)
The list of principal polynomials are given in Table \ref{principalSO5}, so 
we calculate $T(\tilde{\phi})$ for each $C_i$. 
For $p\neq 2$, $3$, the conjugacy classes $C_5$, $C_9$, $C_{13}$, $C_{17}$,
$C_{18}$, $C_{-19}$ in Tables \ref{principalSO5} and \ref{principalSp2}
does not appear. So the remaining cases are 
$C_{1}$, $C_2$, $C_{-3}$, $C_4$, $C_6$, $C_7$, $C_8$, $C_{-10}$, $C_{-11}$, $C_{12}$
$C_{14}$, $C_{15}$, and $C_{16}$. The corresponding characters are $\chi_1$, 
$\chi_2$, $\chi_6$, $\chi_3$, $\chi_4$, $\chi_7$, $\chi_5$, $\chi_{11}$, $\chi_{12}$,
$\chi_9$, $\chi_{10}$, $\chi_{13}$ and $\chi_8$, respectively.
We must subtract contribution of each conjugacy class to $H$ from $2T$.
This contribution to $H$ is given 
by Theorem \ref{nonprincipaldim} in section \ref{appendix}, and under the assumption that 
$p\neq 2$, $3$, only conjugacy classes $C_1$, $C_{-3}$, $C_7$, 
$C_{-10}$, $C_{-11}$ and $C_{14}$ in Table \ref{principalSO5}
contribute to $H$. The contribution of each $C_i$ to $T$ is given in \cite{asai}
Theorem 4.17, so we will see these. 
In the notation of Theorem 4.17, we have $d=p$ and $\det(L)=2p$
and $\epsilon_v$ is the Hasse invariant of $L$ for each place $v$
in the sense of Serre \cite{serre}. 
For our lattice, we have $\epsilon_p=\epsilon_2=-1$ and $\epsilon_q=1$ 
for any other prime $q$,
assuming that $p\neq 2$. Here we have $\epsilon_{\infty}=1$ since $L$ is 
positive definite and if $q$ is prime to the discriminant $2p$  
then of course we have $\epsilon_q=1$.
The fact that $\epsilon_2=-1$ 
can be checked by the fact that  
the lattice in question at $2$ is isometric to 
\[
2p\perp \begin{pmatrix} 0 & 1 \\ 1 & 0 \end{pmatrix}\perp\begin{pmatrix} 0 & 1 \\ 1 & 0 \end{pmatrix}
\]
and this is diagonalized to $diag(2p,2,-2,2,-2)$. We have $\epsilon_p=-1$ 
by the product formula of the Hasse invariant.
In \cite{asai}, the notations $\epsilon(n)$ and $\omega(n)$ mean, 
as in \cite{serre} 3.2, elements in $\{0,1\}$
such that  
\begin{align*}
& \epsilon(n)\equiv (n-1)/2\bmod 2 \\ 
& \omega(n)\equiv (n^2-1)/8 \bmod 2.
\end{align*}
So for a prime $p$ we have 
\[
\left(\frac{-1}{p}\right)=(-1)^{\epsilon(p)}, \qquad 
\left(\frac{2}{p}\right)=(-1)^{\omega(p)}.
\]
We also define 
\[
\fM(m)=\frac{h(-m)}{2^t\cdot u}
\]
where $h(-m)$ is the class number of $\Q(\sqrt{-m})$, $u$ is the order of the 
unit group of the ring of integers of $Q(\sqrt{-m})$ and 
$t$ be the number of the divisors of the discriminant of $\Q(\sqrt{-m})$.
Below, we use freely the notation used in 
\cite{asai} from p. 284 to p. 290 for each conjugacy class $C_i$.
For the general formulas for $h(L,C_i)$ we use here, please refer the corresponding 
pages of \cite{asai}
since we do not copy them here to avoid complication.
For $C_1$, we have $d=p$, $\epsilon_p=-1$, and $e=1$, so 
the contribution to $T$ is $2/(2^8\cdot 3^2\cdot 5)(p^2-1)=(p^2-1)/5760$ 
and the contribution to $H$ is $(p^2-1)/2880$ 
as in Theorem \ref{nonprincipaldim} from the coefficient of $\chi_1$, 
so for $2T-H$, we have   
\[
2\frac{p^2-1}{5760}-\frac{p^2-1}{2880}=0.
\]
For $C_2$, in the notation of \cite{asai} p. 284 (2), we 
have $e=\delta=1$ and the contribution to $T$ is 
\[
\begin{array}{ll}
\dfrac{1}{2^6\cdot 3}\left(9-2\left(\frac{2}{p}\right)\right) 
& \text{ if }p\equiv 1 \bmod 4 \\[1ex]
\dfrac{1}{2^6\cdot 3}B_{2,\chi} & \text{ if } p\equiv 3 \bmod 4,
\end{array}
\]
and the contribution to $H$ of finite order elements with principal 
polynomial $(x-1)^2(x+1)^2$ is $0$ if $p\neq 2$, so 
the contribution to $2T-H$ is the twice of the above values.
For $C_3$, we have $\delta=1$ or $p$ and 
$S_1=S_2=\emptyset$ in p.285 of \cite{asai}.
We have $a=1$ and $-\left(\frac{2}{p}\right)$ for 
$\delta=1$, $p$, respectively. (By definition, the product on the empty set 
is $1$ in the definition of $a$ in \cite{asai},) 
We have 
\[
\fM(1)=\frac{1}{8},\quad \fM(p)=h(-p)\times 
\left\{\begin{array}{cc} 1/8 & p\equiv 1 \bmod 4, \\
1/4 & \text{ if }p\equiv 3 \bmod 4.
\end{array}\right.
\]
For $\delta=1$, we have $(-1)^{\epsilon(p)+\omega(1)}=\left(\frac{-1}{p}\right)$. 
So in this case we have 
\[
A=\frac{1}{4}\left(4+\left(\frac{-1}{p}\right)\right)) 
\]
where $A$ is as in \cite{asai} p. 285 line 6, 
and 
\begin{align*}
\fM(1)\times \frac{1}{6}\times A\frac{p+\epsilon_p(-1/p)}{2}
& =
\frac{1}{8}\times \frac{1}{12}\left(p-\left(\frac{-1}{p}\right)\right)
\times \left(1+\frac{1}{2^2}\left(\frac{-1}{p}\right)\right)
\\ 
& = \frac{1}{2^5\cdot 3}\left(p-\left(\frac{-1}{p}\right)\right)+\frac{1}{2^7\cdot 3}
\left(p\left(\frac{-1}{p}\right)-1\right).
\end{align*}
The contribution of the part $\delta=1$ to $2T$ is the four times of the above and 
is equal to  
\[
\frac{1}{2^3\cdot 3}\left(p-\left(\frac{-1}{p}\right)\right)+\frac{1}{2^5\cdot 3}\left(p\left(\frac{-1}{p}\right)-1\right),
\]
which is exactly the coefficients of $\chi_6$ in $H$ in Theorem \ref{nonprincipaldim}.
So these parts cancel with each other.
The remaining part is the case $\delta=p$. 
When $\delta=p\equiv 1 \bmod 4$, We have 
\[
a(-1)^{\epsilon(p)+\omega(p)}=-\left(\frac{2}{p}\right)\left(\frac{-2}{p}\right)=-1,
\]
so 
\[
A=2^{-2}(4-1)=\frac{3}{2^2}
\]
in this case.
If $p\equiv 3 \bmod 4$, then we see $1+a(-1)^{\omega(p)}=0$, so we have 
\[
A=\frac{1}{2^3\cdot 3}\left(1-\left(\frac{2}{p}\right)\right).
\]
Gathering these data, we have 
\[
h(L.C_3,p)=h(-p)\times 
\left\{\begin{array}{ll}2^{-6} & \text{ if } p\equiv 1 \bmod 4 \\
2^{-6}\left(1-\left(\frac{2}{p}\right)\right) & \text{ if } p\equiv 3 \bmod 4.
\end{array}\right.
\]
So the contribution $4h(L,C_3,p)$ to $2T$ is exactly the coefficient of 
$\chi_6$ in Theorem \ref{compactdim}.
In the case of $C_4$, both conditions (i) and (ii) of (4) in \cite{asai} p. 285 
do not hold, so the contribution to $T$ is $0$, The same for $H$. So this term gives $0$.
In the case $C_6$ either, the conditions (i) and (ii) do not hold and 
the contributions to $T$ and $H$ are $0$.
For $C_7$, noting that $\left(\frac{p}{3}\right)=\left(\frac{-3}{p}\right)$,
we have $A=3+\left(\frac{-3}{p}\right)$ in  (7) of \cite{asai} p. 286 and 
\[
h(L,C_7)=\frac{1}{2^5\cdot 3^2}\left(p-\left(\frac{-3}{p}\right)\right)
\left(3+\left(\frac{-3}{p}\right)\right).
\]
We see that $4h(L.C_7)$ (the contribution to $4T$) is exactly 
the coefficient of $\chi_7$ of $H$ in Theorem \ref{nonprincipaldim},
so the contribution to $2T-H$ is $0$. 
In the case of $C_8$, the condition (ii) in (8) of \cite{asai} p.286 is not
satisfied and the contribution to $T$ is $0$. The contribution of $\phi_5(x)$ to 
$H$ is also $0$, so no contribution to $2T-H$.
In the case of $C_{-10}$, we have $\delta=1$ or $p$ in \cite{asai} p, 287 (10) (i) and 
if $\delta=1$, then by the condition (ii), the contribution is $0$ unless 
$\left(\frac{2}{p}\right)=-1$, i,e, unless $p\equiv 3, 5 \bmod 8$.
In this case, we have $a=0$ and $\fM(2)= h(\sqrt{-2})/2\cdot 2=1/4$. 
So the contribution of the case $\delta=1$ to $2T$ is 
$4h(L,C_{10},2)=4\times (1/2^2)\times (1/4)=1/4$. 
for $p\equiv \pm 3 \bmod 8$. But the coefficient of 
$\chi_{11}$ in $H$ is $(1/8)(1-(2/p))$ so this is $0$ and $1/4$ for $p\equiv \pm 1 
\bmod 8$ and $p\equiv \pm 3 \bmod 8$, respectively. 
So the contribution of these parts to $2T-H$ is $0$.
The case $\delta=p$ remains, and in this case, (ii) in \cite{asai} p. 287 (10) 
means no condition.
We have $\fM(2p)=h(\sqrt{-2p})/2^2\cdot 2$ and 
$h(L,C_{10},2p)=h(\sqrt{-2p})/2^5$, so the contribution to 
$2T$ is $4h(L.C_{10},2p)=h(\sqrt{-2p})/2^3$. This is nothing but the 
coefficient of $\chi_{11}$ in Theorem \ref{compactdim}.
In the case of $C_{11}$, in \cite{asai} p. 287 (11), we have 
$a=b=0$, $c=-1$. We have $\delta=1$ or $p$. If $\delta=p$, then 
$A=1-1=0$, so $h(L,C_{11},p)=0$. So we assume $\delta=1$ hereafter.
Then $A=1-\left(\frac{p}{3}\right)=1-\left(\frac{-3}{p}\right)$,
and by the condition (ii), we have $-1=\left(\frac{3}{p}\right)$.
We have $\fM(1)=1/8$, so the contribution to $2T$ is   
\[
4h(L,C_{11},1)= \frac{1}{12}\left(1-\left(\frac{-3}{p}\right)\right)
\]
if $(3/p)=-1$ and $0$ otherwise.
In other words, this is equal to $1/6$ if $p\equiv 5 \bmod 12$ and 
$0$ otherwise.
On the other hand, the coefficient of $\chi_{12}$ in $H$ is 
\[
\frac{1}{24}\left(1-\left(\frac{3}{p}\right)+\left(\frac{-1}{p}\right)-\left(\frac{-3}{p}\right)\right).
\]
This is also $1/6$ if $p\equiv 5 \bmod 12$ and $0$ otherwise. So 
the contribution to $2T-H$ is $0$. 
In the case of $C_{12}$, the coefficient of $\chi_9$ in $H$ is $0$. 
In \cite{asai} p. 288 (12), we always have $a=b=0$ and $c=-1$ in our setting.
By the condition (i), we have $\delta=1$, $3$, $p$ or $3p$.
If $3\nmid \delta$, we have $A=0$ by definition,
so we should have $\delta=3$ or $3p$. If $\delta=3$, this contradicts to 
the condition (ii) since $\epsilon_p=-1$ but $(3^2/p)=1$. 
So the only possible case is $\delta=3p$. In this case we have $A=2$.
We have 
\[
\fM(3p)=h(\sqrt{-3p})\times 
\left\{\begin{array}{cc} 2^{-3} & \text{ if } p\equiv 1 \bmod 4, \\
2^{-4} & \text{ if } p\equiv 3 \bmod 4.
\end{array}\right.
\]
We first consider the case $p\equiv 1\bmod 4$. 
Then since we have $(-1)^{\omega(3p)}=-\left(\frac{2}{p}\right)$, 
we have 
\[
h(L,C_{12},3p)=\frac{h(\sqrt{-3p})}{2^4\cdot 3}
\left(3+\left(\frac{2}{p}\right)\right).
\]
If $p\equiv 3 \bmod 4$, then 
\[
h(L,C_{12},3p)=\frac{h(\sqrt{-3p})}{2^4\cdot 3}.
\]
So the contribution to $2T-H$ given by $4h(L,C_{12},3p)$ is 
as in Theorem \ref{compactdim}.
In the case of $C_{14}$, then in \cite{asai} p. 289 (14), we have 
$a=0$ and if $p\neq 5$, then we should have $(5/p)=-1$, that is, 
$p\equiv 2$, $3 \bmod 5$ and in this case $h(L.C_{14})=1/10$.
If $p=5$, then $h(L,C_{14})=1/20$ 
So the contribution to $2T$ is 
$4h(L,C_{14})=1/5$, $2/5$ and $0$ for $p=5$, $p\equiv 2,3 \bmod 5$, 
$p\equiv \pm 1 \bmod 5$, respectively. 
This cancels with the corresponding 
coefficients of $\chi_{10}$ in $H$.
In the case of $C_{-15}$, by \cite{asai} p, 289 (15), 
the contribution to $T$ vanishes 
if $p\neq 5$ and is $1/10$ if $p=5$. 
So the contribution to $2T$ for $p=5$ is $1/5$ and $0$ otherwise. 
On the other hand, the contribution to $H$ is $0$ since 
the polynomial $x^4\pm \sqrt{5}x^3+3x^2\pm \sqrt{5}x+1$ does not appear for $H$.
So the contribution to $2T-H$ is as stated in Theorem \ref{compactdim}
for $\chi_{13}$. In the case of $C_{16}$, by \cite{asai} p. 289 (16), 
there is no contribution to $T$, By Theorem \ref{nonprincipaldim},
there is no contribution 
from $\chi_8$ to $H$ either. So the contribution to $2T-H$ is $0$. 
\end{proof}

\section{The case $p=2$ and $3$}\label{p23}
In this section, we give a formula for $Tr(R_{f_1,f_2}(\pi))$ for general 
$(f_1,f_2)$ for $p=2$ and $p=3$ by a different method. Applying \cite{dummigan}
and Theorem \ref{paramodular}, this also gives 
a dimension formula for $S_{k,j}^{\pm}(K(p))$ for $p=2$, $3$ and $k\geq 3$ as before.
For the scalar valued cases $\dim S_k^{\pm}(K(2))$ and 
$\dim S_k^{\pm}(K(3))$, this means that we reprove the formula  
that has already been known 
in \cite{ibuonodera}, \cite{dern}, \cite{pooryuenibu}. 
 
For the case $p=3$, we can calculate $Tr(R_{f_1,f_2}(\pi))$ 
by the same method as in 
the proof of Theorem \ref{compactdim}, but if $p=2$, we have 
a problem since the correspondence between $T$ and the class
number of quinary lattices explained before is not known in 
\cite{ibutypenumber}. So we need a different method.
Here we use more direct method for both $p=2$ and $3$.
We know that the class number $H=1$ for the non-principal genus $\cL_{npg}$  
for both $p=2$ and $3$ (See \cite{hashimotoibu}), so 
$G_A=U_{npg}G$ and $\Gamma_1=G\cap U_{npg}$. We have 
$\#(\Gamma_1)=1920$ for $p=2$ and $720$ for $p=3$. 
Besides, we have $R(\pi)=U_{npg}\pi$ for $\pi \in O$, 
and since $\pi 1_2 \in G$, we see that  
\[
G\cap U_{npg}\pi =\pi(G\cap U_{npg})= \pi\Gamma_1.
\]
So in order to obtain $Tr(R(\pi))$, all we should do is to 
count the number of elements 
$\gamma \in \pi \Gamma_1$ for each fixed principal polynomial.
We can concretely describe these elements by a direct calculation.
First we give a table of principal polynomials and number of corresponding elements 
in $\pi \Gamma$, and then state our theorem 

\begin{lemma}\label{table23}
For a fixed polynomial $\Phi(x)$, the number of elements $\gamma \in \pi \Gamma_1$ 
such that $\Phi(x)$ is their principal polynomial is given in the second 
column in Tables \ref{p2} and \ref{p3} for $p=2$, $3$. 
We put $\phi(x)=p^{-2}\Phi(x/\sqrt{p})$ as before. 
The last column is the character $Tr(\rho_{f_1,f_2}(g))$ of elements
of the corresponding row, where the notation $\chi_{i}$ is explained in section \ref{appendix}. 
\end{lemma}

\begin{table}[htb]
\caption{Number of elements for $p=2$}\label{p2}
\begin{tabular}{lcll}
$\Phi(x)$ & $\pi\Gamma_1$ & $\phi(x)$ & character \\
$(x^2-2)^2$ & $40$ & $(x-1)^2(x+1)^2$ & $\chi_2$ \\
$(x^2+2)^2$ & $120$ & $(x^2+1)^2$ & $\chi_6$ \\
$x^4+2x^2+4$ & $320$ & $x^4+x^2+1$ & $\chi_9$ \\
$x^4+4$ & $600$ & $x^4+1$ & $\chi_{11}$ \\
$(x^2\pm 2x+2)^2$ & $40$ & $(x^2\pm \sqrt{2}x+1)^2$ & $\chi_{14}$ \\
$x^4\pm 2x^3+2x^2\pm 4x+4$ & $320$ & $x^4\pm \sqrt{2}x^3+x^2+\sqrt{2}x+1$ & $\chi_{15}$
\\
$(x^2\pm 2x+2)(x^2+2)$ & $480$ & $(x^2\pm \sqrt{2}x+1)(x^2+1)$ & $\chi_{16}$
\end{tabular}
\end{table}

\begin{table}[htb]
\caption{Number of elements for $p=3$}\label{p3}
\begin{tabular}{lcll}
$\Phi(x)$ & $\pi \Gamma_1$ & $\phi(x)$ & character \\
$(x^2-3)^2$ & $30$ & $(x-1)^2(x+1)^2$ & $\chi_2$ \\
$(x^2+3)^2$ & $30$ & $(x^2+1)$ & $\chi_6$  \\
$(x^2+3x+3)(x^2+3)$ & $120$ & $(x^2+\sqrt{3}x+1)(x^2+1)$ & $\chi_{17}$ \\
$(x^2-3x+3)(x^2+3)$ & $120$ & $(x^2-\sqrt{3}x+1)(x^2+1)$ & $\chi_{17}$ \\
$x^4+9$ & $180$ & $x^4+1$ & $\chi_{11}$ \\
$(x^4+3x^2+9)$ & $240$ & $x^4+x^2+1$ & $\chi_{9}$  
\end{tabular}
\end{table}

\begin{thm}\label{compactp2p3}
For $p=2$, we have 
\[
Tr(R_{f_1,f_2}(\pi))=\frac{1}{48}\chi_2+\frac{1}{16}\chi_6+\frac{1}{6}\chi_9 
+\frac{5}{16}\chi_{11}
+\frac{1}{48}\chi_{14}+\frac{1}{6}\chi_{15} + \frac{1}{4}\chi_{16}.
\]
For $p=3$, we have 
\[
Tr(R_{f_1,f_2}(\pi))=\frac{1}{24}\chi_2+\frac{1}{24}\chi_6+\frac{1}{3}\chi_9
+\frac{1}{4}\chi_{11}+\frac{1}{3}\chi_{17}.
\]
\end{thm}

Theorem \ref{compactp2p3} is an easy corollary of Lemma \ref{table23},
so we prove the lemma. 

Before proving Lemma \ref{table23},
we write a formula for principal polynomial 
by using matrix coefficients.
\begin{lemma}\label{charpoly}
For $g=\begin{pmatrix} a & b  \\ c  & d \end{pmatrix} \in G$ with 
$gg^*=n(g)1_2$, the principal polynomial $\Phi(x)$ of $g$ is given by 
\begin{multline*}
\Phi(x)=x^4-(Tr(a)+Tr(d))x^3+
\\ (Tr(a)Tr(d)-N(b+\overline{c})+2n(g))x^2
-(Tr(a)+Tr(d))n(g)x+n(g)^2
\end{multline*}
We also have 
\[
\Phi(x)=x^4+Tr(g)x^3+(1/2)(Tr(g)^2-Tr(g^2))x^2-Tr(g)n(g)x+n(g)^2.
\]
\end{lemma}
\begin{proof}
Put $\Phi(x)=x^4+a_1x^3+a_2x^2+a_3x+a_4$. 
Here $g^*$ has the same principal polynomial $\Phi(X)$,  and since 
$g^*=n(g)g^{-1}$, we have $x^4\Phi(n(g)x^{-1})=\Phi(x)$. 
So we have $a_4=n(g)^2$ and $a_3=a_1n(g)$,
Now put $y=x+n(g)x^{-1}$. Then we have 
\[
x^{-2}\Phi(x)=y^2+a_1y+a_2-2n(g).
\]
Since 
\[
g+n(g)g^{-1}=g+g^*=\begin{pmatrix} Tr(a) & b+\overline{c} \\ 
c+\overline{b} & Tr(d) \end{pmatrix},
\]
and the components of this matrix are commutative, the characteristic
polynomial of $g+n(g)g^{-1}$ is given by 
\[
y^2-(Tr(a)+Tr(d))y+Tr(a)Tr(d)-N(b+\overline{c}),
\]
Comparing this with $x^{-2}\Phi(X)$, we prove the first formula.
Noting that 
\begin{align*}
Tr(g)^2 & =N(a)^2+N(d)^2+2Tr(a)Tr(d), \\
Tr(g^2) & = Tr(a^2)+Tr(d^2)+2Tr(bc), \\
N(b+\overline{c}) & = N(b)+N(c)+Tr(bc) \\
N(a)+N(b) & = N(b)+N(c)=n(g),
\end{align*}
we have the second formula.
\end{proof}

\begin{proof}[Proof of Lemma \ref{table23}]
When $p=2$, a maximal order $O$ is taken to be 
\[
O=\Z+\Z i+\Z j+ \Z\frac{1+i+j+k}{2},
\]
where $i^2=j^2=-1$, $ij=-ji=k$. 
We have 
\[
O^{\times}=\{\pm 1, \pm i, \pm j, \pm k, (\pm 1\pm i\pm j\pm k)/2\}.
\]
The quaternion hermitian matrix corresponding 
to a lattice in the non-principal genus $\cL_{npg}$ is given by   
\[
gg^*=\begin{pmatrix} 2 & r \\ -r & 2 \end{pmatrix} \text{ where }r=i-k, \quad 
g=\begin{pmatrix} 1 & -1 \\ 0 & r \end{pmatrix}.
\]
The unit group $\Gamma_1$ is a set of matrices $g^{-1}\epsilon g$ such that 
$\epsilon \in M_2(O)$ and $\epsilon gg^*\epsilon^*=gg^*$. 
This is given by the following set of elements (1) to (5) (see \cite{ibueuler} p.592. Here we slightly modify the notation, but essentially the same).
\[
\begin{array}{ll}
(1) \quad \begin{pmatrix} r^{-1}a & -r^{-1}aa_0 \\
r^{-1}a & r^{-1}aa_0 \end{pmatrix}, 
& 
(2) \quad \begin{pmatrix} r^{-1}a & r^{-1}aa_0 \\
-r^{-1}a & r^{-1}aa_0 \end{pmatrix},
\\
(3) \quad \begin{pmatrix} a & 0  \\ 0 & aa_0 \end{pmatrix},
& 
(4)\quad \begin{pmatrix} 0 & aa_0 \\
a & 0 \end{pmatrix}, 
\\
(5) \quad 
\begin{pmatrix} (1+r^{-1}x)a & r^{-1}xaa_0 \\
r^{-1}xa & (1+r^{-1}x)aa_0 \end{pmatrix},
& 
\end{array}
\]
where $a \in O^{\times}$,
$a_0\in \{\pm 1,\pm i, \pm j, \pm k\}$ and 
$x\in \{-i,k,(\pm 1-i\pm j+k)/2\}$. 
The number of elements are 192,192,192,192 and 1152 
for each (1), (2), (3), (4) and (5), respectively.
We have $|\Gamma_1|=1920$. 

For $\pi=r=i-k$, by calculating the principal polynomials of elements in 
$\pi \Gamma_1$, we have the result in Lemma \ref{table23}.
Actual calculation is quite long, 
so we sketch the point.
Using Lemma \ref{charpoly}, we can count the elements having each principal polynomial.
For example, if $\gamma \in \Gamma_1$ is as in (4) above, 
then for $g=\pi \gamma=r\gamma$, we have 
\[
g=\begin{pmatrix} 0 & raa_0 \\ ra & 0 \end{pmatrix}.
\]
So we have $Tr(g)=0$ and $Tr(g^2)=Tr(raa_0ra)+Tr(raraa_0)=2Tr((ra)^2a_0)$. 
Here $N(ra)=2$ and the norm $2$ elements in $O$ are the following 24 elements.
\[
\pm(1\pm i),\pm(1\pm j),\pm(1\pm k), \pm(i\pm j),\pm(i\pm k),\pm (j\pm k)
\]
Here we have $(ra)^2=\pm 2i, \pm 2j, \pm 2k$ for the first 12 elements,
(each for two $a$), 
and $-2$ for the last 12 elements.
First we assume $(ra)^2=-2$. Then
 $Tr((ra)^2a_0)=0$ for $a_0=\pm i$, $\pm j$, $\pm k$.
In this case we have $\Phi(X)=x^4+4$. This is for $12\times 6=72$ elements.
For $a_0=\pm 1$ we have $Tr((ra)^2a_0)=\mp 2$.
This case, we have $\Phi(X)=(x^2\mp 2)^2$ for 12 elements for each.
Next assume that $(ra)^2=\pm 2i$. Then we have $Tr(g^2)\neq 0$ only when 
$a_0=\pm i$. For $a_0=\pm i$, we have $(ra)^2a_0=\pm 2$.
So we have $\Phi(x)=x^4+4$ for $ 4 \times 6=24$ elements and
$\Phi(x)=x^4+4x^2+4$ for $4$ elements,
$\Phi(x)=x^4-4x^2+4$ for $4$ elements.
This is the same for $(ra)^2=\pm 2j$ and $\pm 2k$. 
So as a total, the principal polynomials of $\pi \gamma$ for $\gamma \in \Gamma_1$
in (4) is given by the following table.
\[
\begin{array}{lr}
x^4+4 & 144 \\
x^4+4x^2+4=(x^2+2)^2 & 24 \\
x^4-4x^2+4=(x^2-2)^2 & 24 
\end{array}
\]
The case (4) is the simplest case, but continuing the same sort of 
(more complicated) calculation for (1), (2), (3), (5),
we obtain the number of elements for each principal polynomial as in the 
following table.
\[
\begin{array}{lrrrrrrc}
\text{principal poly.} & (3) & (4) & (1) & (2) & (5) & total & total/1920 \\
x^4+4 & 24 & 144 & 54 & 54 & 324 & 600 & 5/16\\
x^4+4x^3+8x^2+8x+4 & 12 & 0 & 1 & 1 & 6 & 20 & 1/96 \\
x^4-4x^3+8x^2-4x+4 & 12 & 0 & 1 & 1 & 6 & 20 & 1/96 \\
x^4-2x^3+4x^2-4x+4 & 48 & 0 & 24 & 24 & 144 & 240 & 1/8 \\
x^4+2x^2+4x^2+4x+4 & 48 & 0 & 24 & 24 & 144 & 240 & 1/8 \\
x^4+4x^2+4 & 48 & 24 & 6 & 6 & 36 & 120 & 1/16 \\
x^4-4x^2+4 & 0 & 24 & 2 & 2 & 12 & 40 & 1/48 \\
x^4-2x^3+2x^2+4x+4 & 0 & 0 & 20 & 20 & 120 & 160 & 1/12 \\
x^4+2x^3+2x^2+4x+4 & 0 & 0 & 20 & 20 & 120 & 160 & 1/12 \\
x^4+2x^2+4 & 0 & 0 & 40 & 40 & 240 & 320 & 1/6 \\
total & 192 & 192 & 192 & 192 & 1152 & 1920 & 1 
\end{array}
\]
From this table, we obtain Table \ref{p2}.
We omit the details.  

When $p=3$, we can take
\[
O=\Z+\Z\frac{1+\alpha}{2}+\Z\beta+\Z\frac{(1+\alpha)\beta}{2}, 
\qquad \alpha^2=-3, \beta^2=-1,\alpha\beta=-\beta\alpha.
\]
We have 
\[
O^{\times}=\{\pm 1,\pm \beta,(\pm1\pm \alpha)/2,(\pm 1\pm \alpha)\beta/2\}
\]
The quaternion hermitian matrix corresponding to a lattice in $\cL_{npg}$
is given by 
\[
gg^*=\begin{pmatrix} 3 & -(1+\beta)\alpha \\ (1+\beta)\alpha & 3 \end{pmatrix},
\quad g=\begin{pmatrix} 1 & 1+\beta \\ 0 & \alpha \end{pmatrix}.
\]
Then $U_{npg}(p)\pi \cap G$ is given by the following elements.
\[
\begin{array}{ll}
(1) \quad \begin{pmatrix} \beta\alpha a & 0 \\ 0 & \alpha aa_0 \end{pmatrix} 
& a \in O^{\times}, a_0\in \{1,(-1\pm \alpha)/2\}, \\[2ex]
(2) \quad \begin{pmatrix} 0 & \beta\alpha aa_0 \\ \alpha a & 0 \end{pmatrix}
& a \in O^{\times}, a_0\in \{1,(-1\pm \alpha)/2\}, \\[2ex]
(3) \quad \begin{pmatrix} \epsilon_1 & -\epsilon_1\overline{c_2}\epsilon_2 
\\ c_2 & \epsilon_2 \end{pmatrix}
& \begin{array}{l}c_2\in O, N(c_2)=2, \epsilon_1, \epsilon_2 \in O^{\times}, 
\\ \epsilon_1\equiv -(1-\beta)c_2\bmod \alpha, 
\epsilon_2\equiv c_2(1+\beta) \bmod \alpha, 
\end{array}
\\[2ex]
(4) \quad \begin{pmatrix} c_2 & \epsilon_2 \\ \epsilon_1 & -\epsilon_2\overline{c_2}\epsilon_2
\end{pmatrix}
& \begin{array}{l}
c_2\in O, N(c_2)=2, \epsilon_1,\epsilon_2 \in O^{\times}, \\
\epsilon_1\equiv (1+\beta)c_2\bmod \alpha, \epsilon_2 \equiv c_2(1+\beta) 
\bmod \alpha. 
\end{array}\\
\end{array}
\]
The number of elements in (1), (2), (3), (4) are 36, 36, 324, 324, respectively.
By similar method as in the case $p=2$, we have the following number of 
principal polynomials in each case. 
\[
\begin{array}{lrrrrrc}
principal\ poly.      & (1) & (2) &  (3)  & (4) & total & total/720 \\
(x^2+3)^2             & 12  &  0   & 18 & 0 & 30  & 1/24\\ 
(x^2+3)(x^2+3x+3)     & 12   &  0  & 36 & 72 & 120 & 1/6 \\
(x^2+3)(x^2-3x+3)     & 12   &  0  & 36 & 72 & 120 & 1/6 \\
(x^2-3)^2             & 0   &  12  & 18 & 0  & 30 & 1/24 \\
x^4+9                 & 0   &  0  & 144 & 36 & 180 & 1/4 \\
x^4+3x^2+9            & 0   &  24  &   72  & 144 & 240 & 1/3 \\ 
total                 & 36  & 36  & 324 & 324 & 720 & 1
\end{array}
\]

\end{proof} 

For $p=2$, $3$, we can compare dimensions of scalar valued paramodular 
cusp forms of weight
$k$ and  
algebraic modular forms of weight $\rho_{k-3,k-3}$ ($k\geq 3$) with involutions 
directly by \cite{ibuonodera}, \cite{dern}, \cite{pooryuenibu} 
without using \cite{dummigan}.
We note that $S_{j+2}(\Gamma_0(p))=0$ for $j=0$ and $p=2$, $3$, 
so the Yoshida lift part does not appear in the comparison and the relation becomes 
much more simple. 
For algebraic modular forms for $p=2$, $3$, we have the following result from the calculation of 
$Tr(R(\pi))$ given as above and the result in \cite{hashimotoibu} (see section \ref{appendix}).
\begin{align*}
\sum_{f=0}^{\infty}\dim \fM_{f,f}(U_{npg}(2))t^f
& = \dfrac{(1+t^5)(1+t^{20})}{(1-t^4)(1-t^6)(1-t^8)(1-t^{10})}
\\
\sum_{f=0}^{\infty}\dim \fM_{f,f}^{+}(U_{npg}(2))t^f
& = \dfrac{1+t^{25}}{(1-t^4)(1-t^6)(1-t^8)(1-t^{10})}
\\
\sum_{f=0}^{\infty}\dim \fM_{f,f}^{-}(U_{npg}(2))t^f
& = \dfrac{t^5(1+t^{15})}{(1-t^4)(1-t^6)(1-t^8)(1-t^{10})}\\
\sum_{f=0}^{\infty}\dim \fM_{f,f}(U_{npg}(3))t^f
& = 
\frac{(1+t^5)(1+t^{15})}
{(1-t^3)(1-t^4)(1-t^6)(1-t^{10})}
\\
\sum_{f=0}^{\infty}\dim \fM_{f,f}^{+}(U_{npr}(3))t^f
& =\frac{(1+t^8)(1+t^{15})}
{(1-t^4)(1-t^6)^2(1-t^{10})}
\\
\sum_{f=0}^{\infty}\dim \fM_{f,f}^{-}(U_{npr}(3))t^f
& = \frac{(t^3+t^5)(1+t^{15})}{(1-t^4)(1-t^6)^2(1-t^{10})}
\end{align*}
On the other hand, for paramodular forms of level $2$ and $3$, the dimensions for 
plus and minus space (including non-cusp forms) has been given in 
\cite{ibuonodera}, \cite{dern}, \cite{pooryuenibu}. 
The structure of cusps of paramodular varieties is well known
(see for example \cite{ibucusp} Proposition 5.1 to 5.3). 
For level $p$, there are two one-dimensional cusps isomorphic to the compactification of $H_1/SL_2(\Z)$ that are crossing at one point. 
For a fixed $k$, the modular forms of one variable on this boundary consist of 
a pair of elliptic cusp forms of $SL_2(\Z)$, 
which gives one to plus part and other to minus part of $A_k(K(p))$, and 
the Eisenstein series besides, which belong to the plus part. 
This fact and the surjectivity of generalized Siegel $\Phi$-operator 
by Satake \cite{satake} show that $S_k^{\pm}(K(p))$ is given by 
\begin{align*}
\dim S_{k}^+(K(p))  & = \dim A_k^+(K(p))-\dim A_k(SL_2(\Z)), \\
\dim S_k^+(K(p)) & =\dim A_k^-(K(p))-\dim S_k(SL_2(\Z)).  
\end{align*}

Here we give concrete results only for cusp forms.
For the dimensions of the whole space including non cusp forms,
 see \cite{pooryuenibu} p. 113. 

\begin{align*}
\sum_{k=0}\dim S_{k}^{+}(K(2)) t^k & = 
\dfrac{t^8 + t^{10} + t^{12} - t^{20} + t^{23} + t^{33}}
{(1-t^4)(1-t^6)(1-t^8)(1-t^{12})}
\\
\sum_{k=0}\dim S_k^{-}(K(2)) t^k & = 
\dfrac{t^{11} + t^{20} + t^{21} + t^{22} + t^{24} - t^{32}}
{(1-t^4)(1-t^6)(1-t^8)(1-t^{12})}
\\
\sum_{k=0}^{\infty}\dim S_k^+(K(3))t^k & = 
\frac{t^6+t^8+t^{10}+t^{12}-t^{18}+t^{21}+t^{23}+t^{31}}
{(1-t^4)(1-t^6)^2(1-t^{12})}
\\
\sum_{k=0}^{\infty}\dim S_k^-(K(3))t^k & =
\frac{t^9+t^{11}+t^{18}+t^{19}+t^{20}+t^{22}+t^{24}-t^{30}}
{(1-t^4)(1-t^6)^2(1-t^{12})}
\end{align*}

Observing these, we have the following dimensional relations without 
assuming the results of 
\cite{dummigan}.  
(Actually these relations had been obtained before independently of \cite{dummigan}.)

\begin{prop}
For $p=2$ and $p=3$, and any integer $k\geq 3$, we have 
\begin{align*}
& \dim S_{k}^{-}(K(p))-\dim S_{k}(Sp(2,\Z))+\delta_{k,even}\cdot\dim S_{2k-2}(SL_2(\Z))
\\ & \qquad =\dim \fM_{k-3,k-3}^{+}(U_{npr}(p))-\delta_{k,3}
-\delta_{k,odd}\cdot \dim S_{2k-2}(SL_2(\Z)).\\
& \dim S_{k}^{+}(K(p))-\dim S_{k}(Sp(2,\Z)) 
 = \dim \fM_{k-3,k-3}^{-}(U_{npr}(p)),
\end{align*}
where $(\delta_{k,even},\delta_{k,odd})=(1,0)$ if $k$ is even and 
$=(0,1)$ if $k$ is odd. 
Here LHS are the dimensions of new forms in the sense of \cite{robertschmidt}.
The term $\dim S_{2k-2}(SL_2(\Z))$ is an adjustment for 
Saito--Kurokawa lift or 
Ihara lift (which is the compact version of Saito--Kurokawa lift 
though it had been introduced much earlier). 
\end{prop}

\section{Appendix on dimensions and characters}\label{appendix}
We review formulas for $\dim \fM_{f_1,f_2}(U_{npg}(p))$ from \cite{hashimotoibu} II
p. 696 Theorem for readers convenience. . 
We also give explicit formulas of characters of irreducible representations of $Sp(2)$ 
for elements of $Sp(2)$ that we need in the theorems.
\begin{thm}\label{nonprincipaldim}
We assume that $p$ is any prime including $2$, $3$, $5$. Then we have 
\begin{align*}
& \dim \fM_{k+j-3,k-3}(U_{npg}(p)) \\
& = \frac{p^2-1}{2880}\chi_1 +\frac{\delta_{p2}}{192}\chi_2 +\frac{\delta_{p2}}{16}\chi_3
+\frac{\delta_{p3}}{9}\chi_4
\\ & \quad
+\left(\frac{1}{2^3\cdot 3}\left(p-\left(\frac{-1}{p}\right) \right)+
\frac{1}{2^5\cdot 3}\left(p\left(\frac{-1}{p}\right)-1\right)\right)\chi_6 
\\  &  \quad +
\left(\frac{1}{2^3\cdot 3}\left(p-\left(\frac{-3}{p}\right)\right)+
\frac{1}{2^3\cdot 3^2}\left(p\left(\frac{-3}{p}\right)-1\right)\right)\chi_7
\\ & \quad +\frac{\delta_{p2}}{6}\chi_9 
+\frac{\chi_{10}}{5}\left(1-\left(\frac{p}{5}\right)\right)
+\frac{\chi_{11}}{8}\left(1-\left(\frac{2}{p}\right)\right)
\\
 & 
\quad +\frac{\chi_{12}}{24}
\left(1-\left(\frac{3}{p}\right)+\left(\frac{-1}{p}\right)-\left(\frac{-3}{p}\right)\right)
\end{align*}
\end{thm}

Here we understand $(n/p)=0$ if $p$ ramifies in $\Q(\sqrt{n})$, $=-1$ if $p$ remains prime, and $=1$ otherwise. So for example 
we have 
\[
\left(\frac{-1}{2}\right)=\left(\frac{3}{2}\right)=\left(\frac{-3}{3}\right)=\left(\frac{5}{5}\right)=\left(\frac{2}{2}\right)
=0, \qquad \left(\frac{-3}{2}\right)=-1.
\]
Here $\chi_i$ are characters of the elements of $Sp(2)$ whose principal polynomials 
are $\phi_i(\pm x)$, where $\phi_i(x)$ are given as follows.
\[
\begin{array}{ll}
\phi_1(x) = (x-1)^4 & \phi_2(x)=(x-1)^2(x+1)^2 \\
\phi_3(x)= (x-1)^2(x^2+1) & \phi_4(x) = (x-1)^2(x^2+x+1) \\
\phi_5(x)=(x-1)^2(x^2-x+1) & \phi_6(x)=(x^2+1)^2 \\
\phi_7(x) = (x^2+x+1)^2 & \phi_8(x)= (x^2+1)(x^2+x+1) \\
\phi_9(x)=(x^2+x+1)(x^2-x+1) & \phi_{10}=x^4+x^3+x^2+x+1 \\
\phi_{11}(x)  = x^4+1 & \phi_{12}(x)=x^4-x^2+1 \\
\phi_{13}(x) =x^4+\sqrt{5}x^3+3x^2+\sqrt{5}x+1 & 
\phi_{14}(x) = (x^2+\sqrt{2}x+1)^2 \\
\phi_{15}(x)= x^4+\sqrt{2}x^3+x^2+\sqrt{2}x+1 &
\phi_{16}(x)=(x^2+\sqrt{2}x+1)(x^2+1) \\
\phi_{17}(x) = (x^2+\sqrt{3}x+1)(x^2+1) 
\end{array}
\]
These polynomials are possible principal polynomials of all the elements of 
finite order of $G$ for $n=2$ and of 
$g/\sqrt{p}$ with $n(g)=p$ appearing in $U_{npg}\pi$.
Of course the character formula is classical and written in \cite{weyl}, 
but explicit calculation is sometimes complicated and we give them below for 
readers' convenience. We write the Young diagram parameter $(f_1,f_2)$ for 
an irreducible representation of $Sp(2)$. 
By the general formula of \cite{weyl} p. 219 Theorem 7.8E, 
for a principal polynomial $\phi(x)$ of 
any element in $Sp(2)$, 
the character of $\rho_{f_1,f_2}$ is given by 
\begin{equation}\label{characterformula}
p_{f_1}(p_{f_2}+p_{f_2-2})-p_{f_2-1}(p_{f_1+1}+p_{f_1-1})
\end{equation}
where $p_f$ is defined by 
\[
\frac{1}{\phi(x)}=\sum_{f=0}^{\infty}p_f x^f
\]
with $p_f=0$ for any $f<0$. 

Since it would be convenient to use the notation suitable for the calculation of 
paramodular forms $S_{k,j}(K(p))$ of weight $\det^kSym(j)$, 
we put $(f_1,f_2)=(j+k-3,k-3)$ ($k\geq 3$, $j\geq 0$, $j$ even) and 
write the characters below for these parameters $k$, $j$.
In the following table, $\chi_i$ are up to constant the same as $C_i(k,j)$ in \cite{ibukitayama} p. 604,
and also $\chi_i$ for $i=2,6,9,11$ are reproductions of those in the first section, noting that $f_1=k+j-3$, $f_2=k-3$. 
The concrete results for $\chi_{13}$ to $\chi_{17}$ are newly calculated by 
using the above formula \eqref{characterformula}, writing $p_f$ explicitly 
for each $\phi$.  The details of the calculation are omitted here but 
we will add hints for the calculation later.
\begin{align*}
\chi_1 & = \frac{(j+1)(k-2)(j+k-1)(j+2k-3)}{6}, \\
\chi_2 & = (-1)^{k-3}\frac{(k-2)(k+j-1)}{2}, \\
\chi_3 & = \frac{1}{2}[(-1)^{j/2}(k-2),-(j+k-1),-(-1)^{j/2}(k-2),j+k-1;4]_k, \\
\chi_4 & = \frac{1}{3}\bigl((j+k-1)[1,-1,0;3]_k+(k-2)[1,0,-1;3]_{j+k}\bigr), \\
\chi_5 & = (j+k-1)[-1,-1,0,1,1,0;6]_k+(k-2)[1,0,-1,-1,0,1;6]_{j+k}, \\
\chi_6 & = \frac{1}{2}(-1)^{(2k+j-6)/2}\times [-k+2, j+k-1;2]_k, \\
\chi_7 & = \left\{\begin{array}{ll}
[(2k+j-3),(2k+2j-2), (2k-4);3]_k/3 & \text{ if } j\equiv 0 \bmod 3, \\
{}[-(2k+2j-2),-(2k+j-3),-(2k-4);3]_k/3 & \text{ if } j\equiv 1 \bmod 3, \\
{}[j+1,-(j+1),0;3]_k/3 & \text{ if } j\equiv 2\bmod 3,
\end{array}\right.
\\ 
\chi_{8} & = 
\left\{
\begin{array}{ll}
[-1,0,0,1,1,1,1,0,0,-1,-1,-1;12]_k & \text{ if } j\equiv 0 \bmod 12, \\{}
[1,-1,0,-1,-1,0,-1,1,0,1,1,0; 12]_k &  \text{ if } j\equiv 2 \bmod 12, \\{}
[-1,1,0,0,1,-1,1,-1,0,1,-1,1;12]_k & \text{ if } j\equiv 4 \bmod 12, \\{}
[1,0,0,1,-1,1,-1,0,0,-1,1,-1; 12]_k & \text{ if} j\equiv 6 \bmod 12, \\{}
[-1,-1,0,-1,1,0,1,1,0,1,-1,0; 12]_k & \text{ if } j\equiv 8 \bmod 12, \\{}
[1,1,0,0,-1,-1,-1,-1, 0,0,1,1; 12]_k & \text{ if } j\equiv 10 \bmod 12, 
\end{array}\right.
\\ 
\chi_9 &  = 
\left\{\begin{array}{ll}
[-1,0,0,1,0,0;6]_k & \text{ if } j\equiv 0 \bmod 6, \\{}
[1,-1,0,-1,1,0;6]_k & \text{ if } j\equiv 2 \bmod 6, \\{}
[0,1,0,0,-1,0;6]_k & \text{ if } j\equiv 4 \bmod 6, 
\end{array}\right.
\\
\chi_{10}
& =\left\{\begin{array}{ll}
[-1,0,0,1,0;5]_k & \text{ if } j\equiv 0 \bmod 10, \\{}
[1,-1,0,0,0; 5]_k & \text{ if } j\equiv 2 \bmod 10, \\{}
0 & \text{ if } j\equiv 4 \bmod 10, \\{}
[0,0,0,-1,1; 5]_k & \text{ if } j\equiv 6 \bmod 10, \\{}
[0,1,0,0,-1;5]_k & \text{ if }j\equiv 8 \bmod 10, 
\end{array}\right.
\\
\chi_{11} & = \left\{\begin{array}{ll}
[-1,0,0,1;4]_k & \text{ if } j\equiv 0 \bmod 8, \\{}
[1,-1,0,0;4]_k & \text{ if } j\equiv 2 \bmod 8, \\{}
[1,0,0,-1;4]_k & \text{ if } j\equiv 4 \bmod 8, \\{}
[-1,1,0,0;4]_k & \text{ if } j\equiv 6 \bmod 8, 
\end{array}\right.
\\
\chi_{12} & = (-1)^{j/2}\times \left\{
\begin{array}{ll}
[-1,0,0,1,-2,2;6]_k & \text{ if }j\equiv 0 \bmod 6, \\{} 
[-1,1,0;3]_k & \text{ if }j\equiv 2 \bmod 6, \\{}
[2,-1,0,0,1,-2;6]_k & \text{ if }j\equiv 4 \bmod 6,
\end{array}\right.
\\
\chi_{13} & = 
\left\{\begin{array}{ll}
[-1,0,0,1,2,1,0,0,-1,-2;10]_{k} & \text{ if } j\equiv 0 \bmod 10, \\{}
[1,-1,0,2,0,-1,1,0,-2,0;10]_{k} & \text{ if } j\equiv 2 \bmod 10, \\{}
[-2,-2,0,-2,-2,2,2,0,2,2;10]_{k} & \text{ if }j\equiv 4 \bmod 10, \\{}
[0,2,0,-1,1,0,-2,0,1,-1;10]_{k} & \text{ if }j\equiv 6 \bmod 10, \\{}
[2,1,0,0,-1,-2,-1,0,0,1;10]_{k} & \text{ if }j\equiv 8 \bmod 10,
\end{array}\right.
\\
\chi_{14} & = 
\left\{
\begin{array}{ll}
(-1)^{j/4}[j+k-1,j+k-1,k-2,k-2;4]_k & \text{ if } j\equiv 0 \bmod 4, \\
(-1)^{(j-2)/4}[j+k-1,k-2,k-2,j+k-1;4]_k & \text{ if } j\equiv 2 \bmod 4, 
\end{array}\right.
\\
\chi_{15}& =(-1)^{[j/12]}\times 
\\ & 
\left\{\begin{array}{ll} 
[-1,0,0,1,0,-2,1,2,-2,-1,2,0;12]_{k}& \text{ if }j\equiv 0 \bmod 12,\\{}
[1,-1,0;3]_{k} & j\equiv 2 \bmod 12, \\{}
[0,-1,0,2,-1,-2,2,1,-2,0,1,0;12]_{k} & j\equiv 4 \bmod 12, \\{}
[1,-2,0,1,0,0,-1,0,2,-1,-2,2;12]_{k} & j\equiv 6 \bmod 12, \\{}
[1,-1,0;3]_{k} & j\equiv 8 \bmod 12, \\{}
[0,-1,0,0,1,0,-2,1,2,-2,-1,2;12]_{k} & j\equiv 10 \bmod 12, 
\end{array}\right.
\\
\chi_{16} &  
=\left\{
\begin{array}{ll}
[-1,0,0,1,1,0,0,-1;8]_{k} & j \equiv 0 \bmod 8, \\{}
[1,-1,0,0,-1,1,0,0;8]_{k} & j \equiv 2 \bmod 8, \\{}
[-1,0,0,-1,1,0,0,1;8]_{k} & j \equiv 4 \bmod 8, \\{}
[1,1,0,0,-1,-1,0,0;8]_{k} & j \equiv 6 \bmod 8,
\end{array}\right.
\\
\chi_{17} & = 
\left\{
\begin{array}{ll}
[-1,0,0,1,1,-1;6]_{k} & j \equiv 0 \bmod 12, \\{}
[1,-1,0;3]_{k} & j \equiv 2 \bmod 12, \\{}
[-1,-1,0,0,1,1;6]_{k} & j \equiv 4 \bmod 12, \\{}
[1,0,0,-1,-1,1;6]_{k} & j \equiv 6 \bmod 12, \\{}
[-1,1,0;3]_{k} & j \equiv 8 \bmod 12, \\{}
[1,1,0,0,-1,-1;6]_{k} & j \equiv 10 \bmod 12.
\end{array}\right.
\end{align*}

Calculation for $\chi_{13}$ to $\chi_{17}$ are based on the following observation.
Multiplying suitable polynomials to the numerators and 
denominators of $1/\phi_i(x)$, we have  
\begin{align*}
\frac{1}{\phi_{13}(x)} & =
\frac{1+2x^2-2x^4-x^6+\sqrt{5}(-x+x^5)}{1-x^{10}},
\\
\frac{1}{\phi_{14}(x)}& =\frac{(x^4+4x^2+1)-2\sqrt{2}x(x^2+1)}{(x^4+1)^2},
\\
\frac{1}{\phi_{15}(x)}&=\frac{1++x^2+2x^4+x^6+x^8\sqrt{2}(x+x^3+x^5+x^7)}
{1+x^{12}},
\\
\frac{1}{\phi_{16}(x)}& = =\frac{1-x^4+\sqrt{2}(-x+x^3)}{1-x^8}
\\
\frac{1}{\phi_{17}(x)}& = =\frac{1+x^2-x^6-x^8+\sqrt{3}(-x+x^7)}
{1-x^{12}}.
\end{align*}
Noting that 
\[
\frac{1}{(1+x^4)^2}=\sum_{n=0}^{\infty}(-1)^n(n+1)x^{4n},
\]
the formula for $p_f$ defined by $\phi_{14}(x)=\sum_{f=0}^{\infty}p_fx^f$
is given as follows.
\[
p_f=\left\{\begin{array}{ll}
(-1)^{f/4} & f\equiv 0 \bmod 4, \\
-2^{-1}\sqrt{2}(-1)^{(f-1)/4}(f+3) & f\equiv 1 \bmod 4, \\
(-1)^{(f-2)/4}(f+2) & f\equiv 2 \bmod 4, \\
-2^{-1}\sqrt{2}(-1)^{(f-3)/4}(f+1) & f\equiv 3 \bmod 4.
\end{array}\right.
\]
Then calculating \eqref{characterformula} for each $f\bmod 4$ separately, we can show that 
$\chi_{14}$ is given as above. 
In the other cases, the formula for $p_f$ are simpler 
and depend on $f \bmod 10$, $f\bmod 24$, $f\bmod 8$ and $f\bmod 12$ for 
$\phi_{13}$, $\phi_{15}$, $\phi_{16}$ and $\phi_{17}$, respectively.
The characters $\chi_i$ can be calculated similarly as before, divided by cases.

\section{Bias of the dimensions of plus and minus}\label{bias}
In \cite{martin}, for square free $N$, the dimensions of 
$S_k^{new,+}(\Gamma_0(N))$ and 
$S_k^{new,-}(\Gamma_0(N))$ of elliptic cusp forms 
have been compared and it has been shown that we always have 
\[
(-1)^{k/2}\bigl(\dim S_k^{new,+}(\Gamma_0(N))-\dim S_k^{new,-}(\Gamma_0(N))\bigr)\geq 0.
\]
(See \cite{martin} Corollary 2.3. 
Note that the definition of the sign there is the sign of the functional 
equation, which is the Atkin--Lehner sign times $(-1)^{k/2}$.
Since we used the Atkin--Lehner sign for the definition, 
we put $(-1)^{k/2}$ in the above.)

In this section, we similarly compare $\dim S_k^{+}(K(p))$ and $\dim S_k^{-}(K(p))$, 
 and give the same sort of results. 
By Theorem \ref{paramodular} and \ref{compactdim}, 
it is almost trivial that $\dim S_k^+(K(p))$ and 
$\dim S_k^{-}(K(p))$ grow  proportionally to $p^2k^3$ as $k$ and $p$ go to 
infinity. But if we see the difference $\dim S_k^{+}(K(p))-\dim S_k^{-}(K(p))$,
then by Theorem \ref{paramodular}, this is close to  
\[
\dim \fM^{-}_{k-3,k-3}(U_{npg}(p))-\dim \fM^{+}_{k-3,k-3}(U_{npg}(p)).
\]
This is $-Tr R_{k-3,k-3}(\pi)$ by definition, 
and its formula is given in Theorem \ref{compactdim}.
As we will see later, $B_{2,\chi}$ and class numbers 
are approximately $p^{3/2}$ and $\sqrt{p}\log(p)$ up to constant, respectively,
so we can see that the main term of the above difference is 
a positive constant times $-\chi_2B_{2,\chi}$ for big $p$ or $k$, 
which is approximately 
$(-1)^{k}p^{3/2}(k-2)(k-1)$ up to constant.
So it is almost clear that 
$(-1)^{k}\bigl(\dim S_k^+(K(p))-\dim S_k^{-}(K(p))\bigr)$ 
goes to infinity for $k\rightarrow \infty$ or $p\rightarrow \infty$.
By seeing details of the difference, 
we have more exact result given as follows.
\begin{thm}\label{plusminusdiff}
For any integer $k\geq 3$ and any prime $p$, we have 
\[
(-1)^k\bigl(\dim S_k^{+}(K(p))-\dim S_k^{-}(K(p))\bigr)
\geq 0,
\]
The list of $k\geq 3$ and $p$ such that $\dim S_k^+(K(p))=\dim S_k^{-}(K(p))$ 
is given as follows.
\begin{align*}
(p,k) =
& (2,3), (2,4), (2,5),(2,6),(2,7), (2,9), (2,13)
\\ & (3,3), (3,4), (3,5), (3,7), (5,3), (5,4), (7,3), (11,3).
\end{align*}
These are exactly the cases such that 
$S_k(K(p))=0$.
\end{thm}

Now the rest of this section is devoted to the proof of this theorem.

\begin{proof}
We have 
\[
\dim S_2^{new,-}(\Gamma_0(p))-\dim S_2^{new,+}(\Gamma_0(p))=\frac{a_ph(\sqrt{-p})}{2}-1
\]
for $a_p$ defined in \eqref{dimgamma0pm}.
We put 
\[
f(p,k)=(-1)^k\bigl(\dim S_k^+(K(p))-\dim S_k^{-}(K(p))\bigr).
\]
Then by Theorem \ref{paramodular}, we have 
\[
f(p,k)=(-1)^{k-1}\biggl(Tr(R_{k-3,k-3}(\pi))
-\frac{a_p h(\sqrt{-p})}{2}\times \dim S_{2k-2}(SL_2(\Z))\biggr)-\delta_{k3}.
\]

By Theorem \ref{compactdim}, noting that $\chi_2$ is a second order polynomial of $k$ and the other characters $\chi_i$ in Theorem 
\ref{compactdim} are at most linear with respect to $k$, 
we see that $f(p,k)$ is a second order 
polynomial of $k$ depending on $k \bmod 60$ for a fixed $p$.
So for any concrete $p$, 
the evaluation of $f(p,k)$ is just an elementary calculus on 
quadratic polynomials. Indeed for small $p$ like 
$p=2$, $3$, $5$, the result in Theorem 
\ref{plusminusdiff} is easily obtained.
For example, for $p=5$, we have $B_{2,\chi}=4/5$
and $h(\sqrt{-5})=h(\sqrt{-10})=h(\sqrt{-15})=2$, $\dim 
S_2(\Gamma_0(5))=0$, the formula for $f(5,k)$ is explicitly written down as  
 \begin{align*}
  f(5,k) = & \frac{11(k-1)(k-2)}{2^4\cdot 3\cdot 5}
+ \frac{1}{2^4}\left\{\begin{array}{cc} -(k-2) & k:even 
\\ (k-1) & k:odd \end{array}\right.
\\ & +\frac{(-1)^{k-1}}{2^2}[-1,0,,0,1;4]_k
 +\frac{(-1)^{k-1}}{3}[-1,0,0,1,0,0;6]_k
 \\ & +\frac{(-1)^{k-1}}{5}
 [-1,0,0,1,2,1,0,0,-1,-2;10]_k
 \\ & +(-1)^{k}\dim S_{2k-2}(SL_2(\Z))-\delta_{k3}. 
\end{align*}
The third and the fourth terms are always non-negative and 
the fifth term is not less than $-2/5$. It is easy to see that 
\[
\frac{k-8}{6}\leq \dim S_{2k-2}(SL_2(\Z))\leq \frac{k-1}{6}
\]
for any $k$, and by using this evaluation, 
we see that $f(5,k)>0$ for $k\geq 7$.
For $k=3$, $4$, $5$, $6$, by calculating the formula directly,
we see $f(5,3)=f(5,4)=0$ and $f(5,5)=f(5,6)=1$. 
So we proved the result for $p=5$. The case $p=2$ and $3$ are similarly done
and we omit the details.
So in the rest of the proof, we assume that $7\leq p$.
Our strategy is to show first that $f(p,k)>0$ for any prime 
$p\geq 7$ if $k$ is big enough, and that $f(p,k)>0$ 
for any $k\geq 3$ if  
$p$ is big enough. If we can say this, then there remain only 
finitely many $(p,k)$ 
that are not covered by the above range.
Then we calculate $f(p,k)$ directly for these remaining $(p,k)$ and 
show $f(p,k)\geq 0$ directly.   
First, for $p\geq 7$, we have 
\begin{align*}
f(p,k)& =c_2(p)B_{2,\chi}(k-1)(k-2)+c_6(p)h(\sqrt{-p})\left\{\begin{array}{cc}
-(k-2) & k: even \\
k-1 & k: odd 
\end{array}\right.
\\ & +c_9(p)h(\sqrt{-2p})(-1)^{k-1}\chi_9+c_{11}(p)h(\sqrt{-3p})(-1)^{k-1}\chi_{11}
\\ & +
(-1)^{k}\frac{a_ph(\sqrt{-p})}{2} \dim S_{2k-2}(SL_2(\Z))-\delta_{k3}
\end{align*}
where $c_i(p)$ are non-negative constant independent of $k$ 
but depending on $p \bmod 12$.
Here by the formula of $\chi_9$ and $\chi_{11}$, we see that 
$(-1)^{k-1}\chi_9\geq 0$ and $(-1)^{k-1}\chi_{11}\geq 0$, 
so we may ignore these terms for rough evaluation.
We evaluate $c_2(p)B_{2,\chi}$ first.
We denote by $D_0$ the discriminant of $\Q(\sqrt{p})$ and by $\chi$ the 
real character corresponding to $\Q(\sqrt{p})$.
Then we have 
\[
B_{2,\chi}=\frac{L(2,\chi)D_0\sqrt{D_0}}{\pi^2}
\]
by \cite{bernoulli} Theorem 9.6, where $L(s,\chi)$ is the Dirichlet $L$ function 
with respect to the character $\chi$. 
Since 
\[
L(2,\chi)=\prod_{p}(1-\chi(p)p^{-2})^{-1}
\geq \prod_{p}(1+p^{-2})^{-1}=\frac{\zeta(4)}{\zeta_(2)}=
\frac{\pi^2}{15},
\]
we have 
\[
B_{2,\chi}=\frac{L(2,\chi)D_0\sqrt{D_0}}{\pi^2}\geq 
\frac{D_0\sqrt{D_0}}{15}.
\]
For $p\equiv 1 \bmod 4$, we have 
\[
c_2(p)=\frac{1}{2^6\cdot 3}\left(9-\left(\frac{2}{p}\right)\right)
\geq \frac{7}{2^6\cdot 3}
\]
and $D_0=p$, so we have 
\[
c_2(p)B_{2,\chi}\geq \frac{7p^{3/2}}{2^6\cdot 3^2\cdot 5}
\]
On the other hand, if $p\equiv 3 \bmod 4$, then 
$c_2(p)=1/2^6\cdot 3$ and $D_0=4p$, so 
\[
c_2(p)B_{2,\chi}\geq \frac{8p^{3/2}}{2^6\cdot 3^2\cdot 5}.
\]
So we have $c_2(p)B_{2,\chi}\geq 7p^{3/2}/(2^6\cdot 3^2\cdot 5)$ 
in any case.

We explain 
an evaluation of the class numbers needed later.
Let $D<0$ be the fundamental discriminant of 
an imaginary imaginary quadratic field
$\Q(\sqrt{D})$, and $\chi_D$ is the character associated to $\Q(\sqrt{D})$
Then it is well known that for $|D|>3$, we have 
\[
L(1,\chi_D)=\frac{\pi h(\sqrt{D})}{\sqrt{|D|}} 
\]
(\cite{bernoulli} p. 164). On the other hand, by Siegel \cite{siegel} section 15, 
we have
\[
|L(1,\chi_D)|<2\log(|D|)+\frac{1}{2}
\]
so 
\begin{equation}\label{imaginaryclassno}
h(\sqrt{D})<\frac{1}{\pi}\biggl(2\sqrt{|D|}\log|D|+\frac{1}{2}\sqrt{|D|}\biggr).
\end{equation}

For further evaluation, it is convenient to treat the case
$k$ even and odd separately.
First we assume that $k$ is even.  
Then we have  $a_p\geq 1$, $c_6(p)\leq 1/2^4$ and 
$\dim S_{2k-2}(SL_2(\Z))\geq (k-8)/6$, so   
\begin{align*}
f(p,k)& \geq \frac{7p^{3/2}(k-1)(k-2)}{2^6\cdot 3^2\cdot 5}
-\frac{h(\sqrt{-p})(k-2)}{16}
+\frac{h(\sqrt{-p})(k-8)}{12}
\\ & = \frac{7p^{3/2}(k-1)(k-2)}{2^6\cdot 3}
+\frac{h(\sqrt{-p})(k-26)}{48}.
\end{align*}
The RHS is positive if $k\geq 26$. So we assume 
that $4\leq k\leq 26$. The last term is bigger than
$(-11/24)h(\sqrt{-p})$, which is the value at $k=4$. 
The first term is not smaller than the value at $k=4$.  
By \eqref{imaginaryclassno}, for both $p\equiv \pm 1 \bmod 4$, we have 
\[
h(\sqrt{-p})\leq \frac{1}{\pi}(4\sqrt{p}\log(4p)+\sqrt{p}).
\]
So if we put, 
\[
g(x)=\frac{7x^{3/2}(4-1)(4-2)}{2^6\cdot 3\dot 5}
-\frac{23}{48\pi}(4\sqrt{x}\log(x)+\sqrt{x}),
\]
then $f(x,k)\geq g(x)$ for $7\leq x$.
We have 
\[
\frac{d g(x)}{dx} =\frac{-1980+21\pi x-880\log(4x)}{960\pi x}\\
\]
If we write the numerator 
$h(x)=-1980+21\pi x-880\log(4x)$, then 
\[
\frac{d h(x)}{dx}  = 
\frac{-880+21\pi x}{x}.
\]
So $h'(x)$ is positive if $x\geq 14>\frac{880}{21\pi}$. 
This means that $h(x)$ is increasing monotonously
for $x\geq 14$. 
Here we have  $h(112)=36.80/..>0$, so 
$h(x)>0$ and $g'(x)>0$ for $x\geq 112$. 
So $g(x)$ increases monotonously for $x>112$. 
But we have $g(300)=1.58...>0$, so 
$g(x)>0$ if $x\geq 300$. 
As a whole, for any $4\leq k$ and $x\geq 300$,  
and for $k\geq 26$ and any $7\leq x$, we have 
$f(x,k)>0$. For $4\leq k\leq 25$ and $x<300$, 
we see $f(x,k)>0$ by checking directly by the 
explicit formula.

Next we consider the case when $k$ is odd. 
We assume $p\geq 7$ as before.
We have $c_6(p)\geq 0$, so in this case, we have 
\[
f(p,k)\geq \frac{7p^{3/2}}{2^6\cdot 3^2\cdot 5}
-\frac{a_ph(\sqrt{-p})}{2}\dim S_{2k-2}(SL_2(\Z)).
\]
For $k\leq 6$, we have $S_{2k-2}(SL_2(\Z))=0$ 
and $f(p,k)>0$, so  we may assume $k\geq 7$. 
Since $(k-1)/6\geq \dim S_{2k-2}(SL_2(\Z))$ and 
$a_p\leq 4$, we have
\[
f(p,k)\geq \frac{7p^{3/2}}{2^6\cdot 3^2\cdot 5}(k-1)(k-2)
-\frac{k-1}{3}h(\sqrt{-p}).
\]
So if we put 
\begin{align*}
g(x,k)& =\frac{7x^{3/2}(k-2)}{2^6\cdot 3^2\cdot 5}-
\frac{1}{3\pi}(4\sqrt{x}\log(4x)+\sqrt{x})
\\ & =
\frac{\sqrt{x}(-960+7(k-2)\pi x-3840\log(4x))}{288 \pi},
\end{align*}
then $f(p,k)\geq (k-1)g(x,k)$. 
We put 
\[
h(x,k)=-960+7(k-2)\pi x-3940 \log(4x).
\]
Then 
\[
h_x(x,k)=\frac{d h(x,k)}{dx}=\frac{-3840+7\pi(k-2)x}{x}
\]
Since we assumed $k\geq 7$, this increases monotonously 
with respect to $x$, and bigger than $0$ for $x>35$.
So $g(x,k)$ increases monotonously for $x>35$.
We have 
\[
g(250,7)=0.0055...>0
\]
so we have 
\[
0<g(x,7)\leq g(x,k)
\]
for any $x>250$ and $k\geq 7$.
On the other hand, we have 
\[
h_x(7,27)=1.2...>0
\]
so if $k\geq 27$, then $h_x(x,k)>0$ for any $x\geq 7$.
So for any $k\geq 27$, $h(x,k)$ is increasing monotonously 
for $x$. We have 
\[
h(7,92)=98.75...>0
\]
so $g(x,92)>0$ for any $x\geq 7$, but for a fixed $x>0$ and 
$k_1\leq k_2$, we see easily by definition that 
$g(x,k_1)\leq g(x,k_2)$, so $g(x,k)>0$ for any $x\geq 7$ and $k\geq 92$. 
So the remaining cases are $(x,k)$ with 
$x<250$ and $k<92$. For these finitely many cases,
by calculating $f(p,k)$ directly, we see that 
$f(p,k)\geq 0$ and $=0$ only when $(k,p)$ are listed 
in Theorem \ref{plusminusdiff}.
\end{proof}

\section{Numerical Examples}\label{example}
We give several numerical examples of dimensions.
The calculation in this section are mostly done by using
computer software Mathematica \cite{mathematica} and PARI-GP
\cite{pari}.

\subsection{The case $p=5$ and $p=7$}
First we assume $p=5$ for a while.
By Theorem \ref{compactdim} and \ref{nonprincipaldim}, we have 
\begin{align*}
\sum_{f=0}^{\infty}Tr_{f,f}(R(\pi))t^f
& =
\frac{1+t^{11}}{(1-t^2)(1-t^4)(1+t^3)(1+t^5)},
\\
\sum_{f=0}^{\infty}\dim \fM_{f,f}(U_{npg}(5)) t^f
& =
\frac{1+t^{11}}{(1-t^2)(1-t^3)(1-t^4)(1-t^5)},
\end{align*}
so we have 
\begin{align*}
\sum_{f=0}^{\infty}\dim \fM_{f,f}^{+}(U_{npg}(5)) t^f
& = \frac{(1+t^8)(1+t^{11})}{(1-t^2)(1-t^4)(1-t^6)(1-t^{10})},\\
\sum_{f=0}^{\infty}\dim \fM_{f,f}^{-}(U_{npg}(5)) t^f 
& = \frac{(t^3+t^5)(1+t^{11})}{(1-t^2)(1-t^4)(1-t^6)(1-t^{10})}.
\end{align*}
Next, from these result, we will give dimension formulas for 
paramodular forms of plus and minus signs.
We have  
\[
\sum_{k=0}^{\infty}\dim S_{k}(Sp(2,\Z))t^k=\frac{t^{10}+t^{12}-t^{22}+t^{35}}
{(1-t^4)(1-t^6)(1-t^{10})(1-t^{12})},
\]
and 
\[
\sum_{k=0}^{\infty}\dim S_{2k-2}(SL_2(\Z)))t^{k}=
\frac{t^7}{(1-t^2)(1-t^3)}=\frac{t^7+t^{10}}{(1-t^2)(1-t^6)},
\]
and $S_2(\Gamma_0(5))=0$. Using these and Theorem \ref{paramodular}, we have 
\begin{align*}
\sum_{k=0}^{\infty}\dim S_{k}^+(K(5))t^k
& =
\frac{Q_{+}^{(5)}(t)}{(1-t^4)(1-t^6)(1-t^{10})(1-t^{12})},
\\
\sum_{k=0}^{\infty}\dim S_k^{-}(K(5))t^k
& = 
\frac{Q_{-}^{(5)}(t)}{(1-t^4)(1-t^6)(1-t^{10})(1-t^{12})},
\end{align*}
where 
\begin{align*}
Q_{+}^{(5)}(t) = & t^6 + 2 t^8 + 3 t^{10} + 3 t^{12} + 2 t^{14} + 2 t^{16} + t^{17} + t^{18}
\\ &  +  2 t^{19} + 2 t^{21} - t^{22} + 2 t^{23} + 2 t^{25} + 2 t^{27} + t^{29} 
+ t^{35},
\\
Q_{-}^{(5)}(t) = & 
t^5 + t^7 + t^9 + 2 t^{11} + 2 t^{13} + t^{14} + 
 2 t^{15} + t^{16} + t^{17} + t^{18} 
 \\ & + t^{19} + 2 t^{20} + t^{21} + 3 t^{22} + t^{23} + 
 3 t^{24} + t^{26} + t^{28} + t^{30} - t^{34}.
\end{align*}
By the way, adjusting non cusp forms by \cite{satake},  we have 
\begin{align*}
\sum_{k=0}^{\infty}\dim A_k^+(K(5))t^k 
& =
\frac{P_+^{(5)}(t)}{(1-t^4)(1-t^6)(1-t^{10})(1-t^{12})},
\\
\sum_{k=0}^{\infty}\dim A_k^-(K(5))t^k & = 
\frac{P_-^{(5)}(t)}{(1-t^4)(1-t^6)(1-t^{10})(1-t^{12})}, 
\\
\sum_{k=0}^{\infty}\dim A_{k}(K(5))t^k & =
\frac{P^{(5)}(t)}{(1-t^4)(1-t^5)(1-t^6)(1-t^{12})},
\end{align*}
where 
\begin{align*}
P^{(5)}_+(t) = & 1 + t^6 + 2 t^8 + 2 t^{10} + 2 t^{12} + 2 t^{14} + 2 t^{16} + t^{17} + t^{18} + 
\\ &  2 t^{19} + 2 t^{21} + 2 t^{23} + 2 t^{25} + 2 t^{27} + t^{29} + t^{35},
\\
P^{(5)}_-(t)= & t^5 + t^7 + t^9 + 2 t^{11} + t^{12} + 2 t^{13} + t^{14} + 
 2 t^{15} + t^{16} + t^{17} + t^{18} 
 \\ & + t^{19} + 2 t^{20} + t^{21} + 2 t^{22}
  + t^{23} +  2 t^{24} + t^{26} + t^{28} + t^{30},
\\
P^{(5)}(t)= & 1 + t^6 + t^7 + 2 t^8 + t^9 + 2 t^{10} + t^{11} + 2 t^{12} + 2 t^{14} +  2 t^{16} + 2 t^{18} 
\\ & + t^{19} + 2 t^{20} + t^{21} + 2 t^{22} + t^{23} + t^{24} 
 + t^{30}.
\end{align*}
The last formula is easily deduced also from the dimension formula 
in \cite{ibuparamodular} and has been explicitly written in \cite{marschner}.
Note that the denominator for $A_k(K(5))$ above is different from those for 
$A_k^{\pm}(K(5))$. 

Next we consider the case $p=7$. 
By using the formula in Theorem \ref{compactdim} and Theorem \ref{nonprincipaldim}
for $Tr_{k-3,\, k-3}(R(\pi))$ and $\dim \fM_{k-3,\, k-3}(U_{npg}(7))$, we have 
\begin{align*}
\sum_{f=0}^{\infty}\dim\fM_{f,f}^+(U_{npg}(7))t^f
& =
\frac{1 + t^4 + t^6 + t^8 + t^{11} + t^{13} + t^{15} + t^{19}}
{(1-t^2)(1-t^4)(1-t^6)(1-t^{10})}, \\
\sum_{f=0}^{\infty}\dim\fM_{f,f}^-(U_{npg}(7))t^f
& =
\frac{t + t^3 + t^5 + t^9 + t^{10} + t^{14} + t^{16} + t^{18}}
{(1-t^2)(1-t^4)(1-t^6)(1-t^{10})}. 
\end{align*}
In the same way as before, we can calculate 
$\dim S_k^{\pm}(K(7))$ and $\dim A_k^{\pm}(K(7))$ for $k\geq 3$ using 
the above formulas.
We know that $A_1(K(p))=0$ and $A_2(K(p))=S_2(K(p))$ for general prime $p$.
Since $\dim S_8(K(7))=2$ and $\dim A_6(K(7))=3$, we have 
$A_2(K(7))=S_2(K(7))=0$. By $S_2(\Gamma_0(7))=0$, using Theorem \ref{paramodular},  we have 

\begin{align*}
\sum_{k=0}^{\infty}\dim S_k^{+}(K(7))t^k & = 
\frac{Q^{(7)}_{+}(t)}{(1-t^4)^2(1-t^6)(1-t^{12})},
\\
\sum_{k=0}^{\infty}\dim S_k^{-}(K(7))t^k & = 
\frac{Q^{(7)}_{-}(t)}{(1-t^4)^2(1-t^6)(1-t^{12})},
\end{align*}
where 
\begin{align*}
Q_{+}^{(7)}(t)= &  t^4 + 2 t^6 + 2 t^8 + 2 t^{10} + 2 t^{12} + t^{13} + t^{14} 
\\ & + t^{15} + t^{17} + 2 t^{19} + 2 t^{21} + 2 t^{23} + t^{29}, \\
Q_{-}^{(7)}(t) = &  t^5 + 2 t^7 + 2 t^9 + 2 t^{11} + t^{13} + t^{14} + t^{15}
\\ &  + 2 t^{16} + t^{17} + 2 t^{18} + 2 t^{20} + 2 t^{22} + 2 t^{24} - t^{28}.\\
\end{align*}
Counting the modular forms on the boundary, we have 
\begin{align*}
\sum_{k=0}^{\infty}\dim A_k^{+}(K(7))t^k & = \frac{P_{+}^{(7)}(t)}
{(1-t^4)^2(1-t^6)(1-t^{12})},
\\
\sum_{k=0}^{\infty}\dim A_k^{-}(K(7))t^k & = 
 \frac{P_{-}^{(7)}(t)}
{(1-t^4)^2(1-t^6)(1-t^{12})},
\\
\sum_{k=0}^{\infty}\dim A_k(K(7))t^k & = 
 \frac{P^{(7)}(t)}
{(1-t^4)^2(1-t^6)(1-t^{12})}.\end{align*}

\begin{align*}
P_{+}^{(7)}(t) = & 1 + 2 t^6 + 2 t^8 + 2 t^{10} + t^{12} + t^{13} + t^{14} + t^{15} + t^{16} 
\\ & + t^{17} + 2 t^{19} + 2 t^{21} + 2 t^{23} + t^{29},
\\
P_{-}^{(7)}(t) = & t^5 + 2 t^7 + 2 t^9 + 
 2 t^{11} + t^{12} + t^{13} + t^{14} + t^{15} + t^{16}
 \\ &  + t^{17} + 2 t^{18} + 2 t^{20} + 2 t^{22} + t^{24},
 \\
 P^{(7)}(t) = & 
 1 + t^5 + 2 t^6 + 2 t^7 + 2 t^8 + 2 t^9 + 2 t^{10} + 2 t^{11} + 2 t^{12} + 
 2 t^{13}
\\ &  + 2 t^{14} + 2 t^{15} 
  + 2 t^{16} + 2 t^{17} + 2 t^{18} + 2 t^{19} + 
 2 t^{20} + 2 t^{21} 
\\ & + 2 t^{22} + 2 t^{23} + t^{24} + t^{29}.
\end{align*} 

The last formula is easily deduced also from \cite{ibuparamodular} and 
has been explicitly written in \cite{williams}.

Explicit generators of the rings  
$\oplus_{k=0}^{\infty}A_k^{+}(K(p))$ for $p=5$ and $7$ 
and their relations have been given in \cite{williams}, 
though the generating functions of dimensions 
have not been given in the paper \cite{williams}.
Replying to the author's question, Professor Williams 
kindly informed the author that the above generating 
functions for $\dim A_k^+(K(p))$ for 
$p=5$, $7$ can be obtained 
also by his method and the results are completely
the same. 

\subsection{Small primes $p$ and palindromic Hilbert series}
We consider the graded rings $A(K(p))=\oplus_{k=0}^{\infty}A_k(K(p))$ and 
$A^+(K(p))=\oplus_{k=0}^{\infty}A_k^+(K(p))$. 
For any Noetherian graded integral domain 
$B=\oplus_{k=0}^{\infty}B_k$ over $\C$ 
with $B_0=\C$, we call $F(B,t)=\sum_{k=0}^{\infty}\dim B_k t^k$ the Hilbert 
series of $B$. 
We say that $F(B,t)$ is palindromic if $F(B,1/t)=(-1)^mt^{\ell}F(B,t)$ 
for some $\ell>0$, where $m$ is the Krull dimension of $B$.
Under the assumption that $B$ is a Cohen-Macaulay ring, 
it is known that $B$ is a Gorenstein ring 
if and only if $F(B,t)$ is palindromic (\cite{stanley1} Theorem 4.4, \cite{stanley2} p. 503. I learned these references 
from R. Schmidt).
For $p=2$, $3$, we see that 
$A(K(p))$ and $A^+(K(p))$ are Cohen-Macaulay 
since the rings are free over rings generated by $4$ algebraically independent 
forms by \cite{ibuonodera} Theorem 1, Corollary 1 and \cite{dern}  Lemma 5.4, 5.5.
They are also Gorenstein since their Hilbert series are palindromic.
For $p=5$, $7$, the ring $A(K(p))$ are Gorenstein by 
B. Williams \cite{williams} Corollary 16 and 24. He also checked that 
$A^+(K(5))$ and $A^+(K(7))$ are Gorenstein (private communication).

Cris Poor asked the author for which primes $p$, the 
Hilbert series of $A(K(p))$ and $A^+(K(p))$ 
are palindromic. 
The proposition below answer this for small primes.
We denote by $J_{2,p}$ the space of Jacobi forms of 
$SL_2(\Z)$ of weight $2$ of 
index $p$.

\begin{prop}\label{palindromic}
Let $p$ be a prime less than $100$. Among this range, 
we have the following results.
\\
(i) $F(A(K(p)),t)$ is palindromic if and only if $p=2$, $3$, $5$, $7$, $13$. 
\\
(ii) $F(A^+(K(p)),t)$ is palindromic if and only if 
$p=2$, $3$, $5$, $11$, $13$, $17$, $19$, $23$, $29$, $31$, $41$, $47$, 
$59$, $71$. \\
(iii) The primes in (1) are exactly those such that $S_2(\Gamma_0(p))=0$.
The primes in (ii) are exactly those such that 
$J_{2,p}=0$. 
\end{prop}

By \cite{stanley1}, \cite{stanley2} and the 
palindromic property, it is natural to ask 
if $A(K(p))$ and $A^+(K(p))$ for primes 
in the above Proposition are
Gorenstein. Seeing the Hilbert series,  our guess is that 
the rings $A(K(p))$ and $A^+(K(p))$ are not  
Cohen-Macaulay for $p<100$ except for the above listed primes.
The theoretical meaning of (iii) is not clear.

The proof of Proposition \ref{palindromic} merely consists of 
explicit calculations of all the Hilbert series 
for this range by Theorem \ref{compactdim} and \ref{paramodular} except for the point
that we need $\dim A_1(K(p))$ and $\dim A_2(K(p))$. 
We know that we always have $A_1(K(p))=0$ (e.g. \cite{ibuweightthree}). 
We see easily that $A_2(K(p))=S_2(K(p))$ by the boundary structure of the 
Satake compactification of $K(p)\backslash H_2$ and by the fact $A_2(SL_2(\Z))=0$.
For the range of $p<277$, the space $S_2(K(p))$ are all spanned 
by the Gritsenko lift from $J_{2,p}$ by \cite{poor}, so  
we have $\dim S_2(K(p))=\dim J_{2,p}$ for such $p$.
These are all Atkin--Lehner plus.
The latter dimension is given by the table
\[
{\footnotesize
\begin{array}{|c|ccccccccccccccc|} \hline
p& \leq 31 & 37 & 41 & 43 & 47& 53 & 59 & 61 & 67 & 71 & 73 & 79 &  83 & 89 & 97  \\ 
\hline 
\dim J_{2,p} & 0 & 1 & 0 & 1 & 0 & 1 & 0 & 1 & 2 & 0 & 2 & 1 & 1 & 1 & 3  \\ \hline
\end{array}
}
\]
(\cite{eichlerzagier} p. 132).
We list below only palindromic Hilbert series 
for the range of primes $11\leq p\leq 71$. 
Non-palindromic cases has been calculated similarly but 
omitted here to avoid lengthy statement 
except for one example $A(K(11))$. 
It seems there is no 
way to reasonably adjust the numerator of $F(A(K(11)),t)$ as positive coefficients,  
so maybe $A(K(11))$ is not Cohen-Macaulay. 
The other non-palindromic Hilbert series for this range suggest the same thing.
In the list below, sometimes we give non-irreducible polynomials as the numerators.
We have the following reasons.
We prefer that factors $1-t^a$ of the denominator suggests that there
exists a form of weight $a$. But for example, $P^{(13)}(t)$ is divisible by 
$(1+t)(1+t^2)$ and if we divide the denominator by this, then the denominator contains 
the factor $1-t$ in spite of the fact that there is no paramodular form of weight $1$. 
The second reason is that we prefer to give polynomials with non-negative 
coefficients if possible in order to guess easily 
possible module structures over the rings generated by four 
algebraically independent forms corresponding with the denominator.
We note that $P^{(11)}_+(t)$, $P^{(19)}_+(t)$, $P^{(29)}_+(t)$, $P^{(31)}_+(t)$ 
and $P^{(41)}_+(t)$ below are all divisible by $1+t^5$ but divided polynomials
contain negative coefficients.

\begin{align*}
\sum_{k=0}^{\infty}\dim A_k^+(K(11))t^k& =\frac{P^{(11)}_+(t)}
{(1-t^4)(1-t^6)(1-t^{10})(1-t^{12})}
\\
\sum_{k=0}^{\infty}\dim A_k(K(11))t^k & = \frac{P^{(11)}(t)}
{(1-t^4)(1-t^5)(1-t^6)(1-t^{12})}
\end{align*}
\begin{align*}
P^{(11)}_+(t)=& 1 + t^4 + 3 t^6 + 5 t^8 + 7 t^{10} + t^{11}
 + 9 t^{12} + 3 t^{13}
\\ &  + 10 t^{14} +5 t^{15} + 9 t^{16} + 7 t^{17} + 7 t^{18}
  + 9 t^{19} + 5 t^{20}
\\ &  + 10 t^{21}  + 3 t^{22} + 9 t^{23} + t^{24} 
 + 7 t^{25} + 5 t^{27}
 + 3 t^{29} + t^{31} + t^{35}
\\ 
P^{(11)}(t) = & 1 + t^4 + t^5 + 3 t^6 + 3 t^7 + 5 t^8 + 4 t^9 + 6 t^{10} 
+ 6 t^{11} +  8 t^{12} + 7 t^{13} 
\\ & + 10 t^{14}+ 9 t^{15} + 9 t^{16}  + 8 t^{17} + 8 t^{18}
 +7 t^{19} + 7 t^{20}
\\ & 
  + 6 t^{21} + 6 t^{22} + 4 t^{23} + 
 3 t^{24} + t^{25} - t^{27} - t^{29} + t^{30}
\end{align*}
\begin{align*}
\sum_{k=0}^{\infty}\dim A_k^+(K(13))t^k& =\frac{P^{(13)}_{+}(t)}
{(1-t^4)^2(1-t^6)(1-t^{12})}
\\
\sum_{k=0}^{\infty}\dim A_k(K(13))t^k & =
\frac{P^{(13)}(t)}{(1-t^4)^2(1-t^6)(1-t^{12})} 
\end{align*}
\begin{align*}
P^{(13)}_+(t) = & 1 + t^4 + 5 t^6 + 6 t^8 + 6 t^{10} + t^{11} + 5 t^{12} 
+ 4 t^{13} + 5 t^{14}
\\ &  + 
 5 t^{15} + 4 t^{16} + 5 t^{17} + t^{18} + 6 t^{19} + 6 t^{21} + 
 5 t^{23} + t^{25} + t^{29}  
\\ 
P^{(13)}(t) = & 1 + t^3 + t^4 + 3 t^5 + 5 t^6 + 5 t^7 + 6 t^8 + 6 t^9 + 7 t^{10} 
+ 7 t^{11} + 8 t^{12}
\\ &  + 9 t^{13} + 9 t^{14} 
 + 9 t^{15} + 9 t^{16} + 8 t^{17} + 
 7 t^{18}  + 7 t^{19}+ 6 t^{20} 
\\ & + 6 t^{21} + 5 t^{22} + 5 t^{23} + 
 3 t^{24} + t^{25} + t^{26} + t^{29}
\end{align*}
\begin{align*}
\sum_{k=0}^{\infty}\dim A_k^+(K(17))t^k& =\frac{P^{(17)}_{+}(t)}
{(1-t^4)^2(1-t^6)(1-t^{12})}
\end{align*}
\begin{align*}
P^{(17)}_+(t)  = & 1 + t^4 + 6 t^6 + 9 t^8 + t^9 + 10 t^{10} + 4 t^{11} 
+ 9 t^{12} + 8 t^{13}+ 9 t^{14} + 9 t^{15}
\\ &  + 8 t^{16} + 9 t^{17} + 4 t^{18} + 10 t^{19}
  + t^{20} +  9 t^{21} + 6 t^{23} + t^{25} + t^{29} 
\end{align*}
\begin{align*}
\sum_{k=0}^{\infty}\dim A_k^+(K(19))t^k& =\frac{P^{(19)}_{+}(t)}
{(1-t^4)(1-t^6)(1-t^{10})(1-t^{12})}
\end{align*}
\begin{align*}
P^{(19)}_+(t) = &
1 + 3 t^4 + 8 t^6 + 14 t^8 + t^9 + 20 t^{10} + 4 t^{11} + 26 t^{12} + 
 10 t^{13} + 29 t^{14}
\\ &  + 16 t^{15} + 27 t^{16} + 22 t^{17} + 22 t^{18} + 
 27 t^{19} + 16 t^{20} + 29 t^{21} + 10 t^{22}
\\ &  + 26 t^{23} + 4 t^{24} + 
 20 t^{25} + t^{26} + 14 t^{27} + 8 t^{29} + 3 t^{31} + t^{35},
\end{align*}
\begin{align*}
\sum_{k=0}^{\infty}\dim A_k^+(K(23))t^k& =\frac{P^{(23)}_{+}(t)}
{(1-t^4)^2(1-t^6)(1-t^{12})}
\end{align*}
\begin{align*}
P^{(23)}_+(t) = & 1 + 2 t^4 + 9 t^6 + t^7 + 14 t^8 + 4 t^9 + 17 t^{10}
 + 9 t^{11} + 17 t^{12}
\\ &  + 15 t^{13} + 17 t^{14} + 17 t^{15} + 15 t^{16} 
 + 17 t^{17}  + 9 t^{18} + 17 t^{19} 
\\ &+ 4 t^{20} + 14 t^{21} + t^{22} + 9 t^{23} + 2 t^{25} + t^{29}
 - 2 t^{28} + t^{29}
\end{align*}
\begin{align*}
\sum_{k=0}^{\infty}\dim A_k^+(K(29))t^k& =\frac{P^{(29)}_{+}(t)}
{(1-t^4)(1-t^6)(1-t^{10})(1-t^{12})}
\end{align*}
\begin{align*}
P^{(29)}_+(t) = & 
1 + 4 t^4 + 14 t^6 + t^7 + 27 t^8 + 5 t^9 + 41 t^{10} + 15 t^{11}
\\ & +  55 t^{12} + 29 t^{13} + 65 t^{14} + 43 t^{15} + 65 t^{16} 
 + 56 t^{17} 
\\ & + 
 56 t^{18} + 65 t^{19} + 43 t^{20} + 65 t^{21} + 29 t^{22} + 55 t^{23}
\\ &  + 
 15 t^{24} + 41 t^{25} + 5 t^{26} + 27 t^{27} + t^{28} + 14 t^{29} + 4 t^{31} 
+ t^{35}.
\end{align*}
\begin{align*}
\sum_{k=0}^{\infty}\dim A_k^+(K(31))t^k& =\frac{P^{(31)}_{+}(t)}
{(1-t^4)(1-t^6)(1-t^{10})(1-t^{12})}
\end{align*}
\begin{align*}
P^{(31)}_+(t) = & 1 + 6 t^4 + 17 t^6 + t^7 + 32 t^8 + 6 t^9 + 48 t^{10}
  + 16 t^{11} \\ &  + 64 t^{12} + 32 t^{13} + 74 t^{14} + 48 t^{15} + 73 t^{16}
  + 63 t^{17} \\ & + 63 t^{18} + 73 t^{19} + 48 t^{20} + 74 t^{21} + 32 t^{22}
  + 64 t^{23} \\ & +  16 t^{24} + 48 t^{25} + 6 t^{26} + 32 t^{27}
 + t^{28} + 17 t^{29} + 6 t^{31} + t^{35}.\\
\end{align*}
\begin{align*}
\sum_{k=0}^{\infty}\dim A_k^+(K(41))t^k & = 
\frac{P^{(41)}_{+}(t)}{(1-t^4)(1-t^6)(1-t^{10})(1-t^{12})}
\end{align*}
\begin{align*}
P^{(41)}_+(t) & = 
1 + 7 t^4 + 24 t^6 + 3 t^7 + 49 t^8 + 14 t^9 + 77 t^{10} + 35 t^{11} + 
 105 t^{12}
\\ &  + 63 t^{13} + 126 t^{14} + 91 t^{15} + 130 t^{16} + 116 t^{17} + 
 116 t^{18} + 130 t^{19} 
\\ & + 91 t^{20} + 126 t^{21} + 63 t^{22} + 105 t^{23} + 
 35 t^{24} + 77 t^{25} + 14 t^{26} 
\\ & + 49 t^{27} + 3 t^{28} + 24 t^{29} + 
 7 t^{31} + t^{35}.
\end{align*}
\begin{align*}
\sum_{k=0}^{\infty}\dim A_k^+(K(47))t^k & = 
\frac{P^{(47)}_{+}(t)}{(1-t^4)^2(1-t^6)(1-t^{12})}
\end{align*}
\begin{align*}
P^{(47)}_+(t) & = 1 + 7 t^4 + t^5 + 27 t^6 + 8 t^7 + 49 t^8 + 25 t^9
 + 66 t^{10} + 47 t^{11}
 \\ &  + 72 t^{12} + 66 t^{13} + 73 t^{14} + 73 t^{15} + 66 t^{16} + 
 72 t^{17} + 47 t^{18}
 \\ &  + 66 t^{19} + 25 t^{20} + 49 t^{21} + 8 t^{22} + 
 27 t^{23} + t^{24} + 7 t^{25} + t^{29}
\end{align*}
\begin{align*}
\sum_{k=0}^{\infty}\dim A_k^+(K(59))t^k & = 
\frac{P^{(59)}_{+}(t)}{(1-t^4)(1-t^5)(1-t^6)(1-t^{12})}
\end{align*}
\begin{align*}
P^{(59)}_+(t)  =& 1 + 11 t^4 + 40 t^6 + 12 t^7 + 87 t^8 + 30 t^9 + 144 t^{10} + 48 t^{11} + 
 190 t^{12}
 \\ &  + 59 t^{13} + 219 t^{14} + 59 t^{15} + 219 t^{16} + 59 t^{17} + 
 190 t^{18}
 \\ & + 48 t^{19} + 144 t^{20} + 30 t^{21} + 87 t^{22} + 12 t^{23} + 
 40 t^{24} + 11 t^{26} + t^{30}
\end{align*}
\begin{align*}
\sum_{k=0}^{\infty}\dim A_k^+(K(71))t^k & = 
\frac{P^{(71)}_{+}(t)}{(1-t^4)(1-t^5)(1-t^6)(1-t^{12})}\\
\end{align*}
\begin{align*}
P^{(71)}_{+}(t) = & 1 + 15 t^4 + t^5 + 56 t^6 + 19 t^7 + 123 t^8 + 47 t^9 + 204 t^{10}
\\ &  + 75 t^{11} + 270 t^{12} + 92 t^{13} + 311 t^{14} + 93 t^{15} + 311 t^{16}
\\ &  + 
 92 t^{17} + 270 t^{18} + 75 t^{19} + 204 t^{20} + 47 t^{21} + 123 t^{22} + 
 19 t^{23} 
\\ & + 56 t^{24}  + t^{25} + 15 t^{26} + t^{30}
\end{align*}

\subsection{The case of $j>0$}
When $j=2$, we have the following results.
\begin{align*}
\sum_{f=0}^{\infty}\dim \fM_{f+2,f}^{+}(U_{npg}(2))t^f
& = \frac{t^8(1+t^2+2t^4-t^5+t^6-t^9-t^{12}+t^{13})}
{(1-t^4)(1-t^5)(1-t^6)(1-t^8)}, \\
\sum_{f=0}^{\infty} \dim \fM_{f+2,f}^{-}(U_{npg}(2))t^f
& = \frac{t^7}{(1-t^2)(1-t^4)(1-t^5)(1-t^8)}, \\
\sum_{f=0}^{\infty} \dim \fM_{f+2,f}^{+}(U_{npg}(3))t^f
& = \frac{t^6(1+t^2+t^3-t^5+t^6-t^7-t^8+t^9)}
{(1-t^2)(1-t^4)(1-t^5)(1-t^6)},
\\
\sum_{f=0}^{\infty}\dim \fM_{f+2,f}^{-}(U_{npg}(3))t^f
& = \frac{t^5}{(1-t^2)^2(1-t^5)(1-t^6)}.
\end{align*}

Here we have
 $\dim S_4(\Gamma_0(2))=\dim S_4(\Gamma_0(3))=0$, so 
 there is no Yoshida lifting. Since $j>0$, there is no Saito--Kurokawa lift. So for $p=2$ and $3$ and $k\geq 3$. we have
\begin{align*}
\dim S_{k,2}^{+}(K(p)) & = \dim S_{k,2}(Sp(2,\Z))+\dim \fM_{k-1,k-3}^{-}
(U_{npg}(p)), \\
\dim S_{k,2}^{-}(K(p)) & = \dim S_{k,2}(Sp(2,\Z))+\dim \fM_{k-1,k-3}^{+}
(U_{npg}(p)).
\end{align*}
Here we have 
\begin{multline*}
\sum_{k=0}^{\infty}\dim S_{k,2}(Sp(2,\Z))t^k
\\ = \frac{t^{14}+2t^{16}+t^{18}+t^{22}-t^{26}-t^{28}
+t^{21}+t^{23}+t^{27}+t^{29}-t^{33}}{(1-t^4)(1-t^6)(1-t^{10})(1-t^{12})}.
\end{multline*}
(See \cite{tsushima}, \cite{satoh}, \cite{ibuvect246}).
So we can give the following generating functions.

\begin{align*}
\sum_{k=0}^{\infty}\dim S_{k,2}^+(K(2))t^{k}
& = \frac{P_{+,2}^{(2)}(t)}{(1-t^2)(1-t^6)(1-t^8)(1-t^{12})},
\\
\sum_{k=0}^{\infty}\dim S_{k,2}^-(K(2))t^{k}
& = \frac{P_{-,2}^{(2)}(t)}{(1-t^2)(1-t^6)(1-t^8)(1-t^{12})},
\\
\sum_{k=0}^{\infty}\dim S_{k,2}^+(K(3))t^{k}
& = \frac{P_{+,2}^{(3)}(t)}{(1-t^2)(1-t^4)(1-t^6)(1-t^{12})},
\\
\sum_{k=0}^{\infty}\dim S_{k,2}^-(K(3))t^{k}
& = \frac{P_{-,2}^{(3)}(t)}{(1-t^2)(1-t^4)(1-t^6)(1-t^{12})},
\end{align*}
where 
\begin{align*}
P_{+,2}^{(2)}(t)  = & t^{10} + 2 t^{14} + t^{15} + t^{18} + t^{19}
 + t^{23} - t^{24} + t^{27} - t^{29},
\\
P_{-,2}^{(2)}(t)  = &  t^{11} + t^{14} + 
 2 t^{15} + t^{16} - t^{17} + t^{18} + t^{19} + t^{22} - t^{24} + t^{26} - t^{28},
\\
P_{+,2}^{(3)}(t)  = &  t^8 + t^{10} + t^{13} + t^{14} + t^{15} + t^{16} 
- t^{20} + t^{21} + t^{23} - t^{25},
\\
P_{-,2}^{(3)}(t) =  & t^9 + t^{11} + t^{12} + t^{14} + t^{15} + t^{16} + t^{22} 
- t^{24}.
\end{align*}

The result for $k=0$, $1$, $2$ are obtained 
by the dimension formula for higher $k$.
For example, we have $\dim A_4(K(2))=1$, so 
if $\dim S_{2,2}(K(2))\neq 0$, then this contradicts 
to the formula $\dim S_{6,2}(K(2))=0$. Arguments for the other 
cases are similar. 

Next we consider the case $j=4$ and $p=2$, $3$. 
There is one different point here. 
We have $j+2=6$ and $2k+j-2=2k+2$.
We have $S_{6}(\Gamma_0(2))=0$ and for $p=2$ we have no
new phenomenon and we have 
\begin{align*}
\dim S_{k,4}^+(K(2))=\dim \fM_{k+1,k-3}^{-}(U_{npg}(2))+\dim S_{k,4}(Sp(2,\Z)), \\
\dim S_{k,4}^-(K(2))=\dim \fM_{k+1,k-3}^{+}(U_{npg}(2))+\dim S_{k,4}(Sp(2,\Z)).
\end{align*}
But we have $\dim S_6(\Gamma_0(3))=\dim S_6^{-}(\Gamma_0(3))=1$. 
This means that we have the Yoshida lifting part 
$S_6^{-}(\Gamma_0(3))\times S_{2k+2}(SL_2(\Z))$ in 
$M_{k+1,k-3}^{+}(U_{npg}(3))$. So we have 
\begin{align*}
\dim S_{k,4}^{+}(K(3))  = & \dim \fM_{k+1,k-3}^{-}(U_{npg}(3)) 
+\dim S_{k,4}(Sp(2,\Z)), \\
\dim S_{k,4}^{-}(K(3)) = & \dim \fM_{k+1,k-3}^{+}(U_{npg}(3))
\\ & +\dim S_{k,4}(Sp(2,\Z))-\dim S_{2k+2}(SL_2(\Z)).
\end{align*}
By \cite{tsushima} and \cite{ibuvect246}, we have 
\begin{multline}
\sum_{k=0}^{\infty}\dim S_{k,4}(Sp(2,\Z))t^k
\\ =
\frac{t^{10} + t^{12} + t^{14} + t^{15} + t^{16} + t^{17} + t^{18} + t^{19} 
+ t^{20} + t^{21} +t^{23} - t^{30}}{(1-t^4)(1-t^6)(1-t^{10})(1-t^{12})}. 
\end{multline}

Anyway, by Theorem \ref{compactp2p3} and \ref{nonprincipaldim}, we have 
\begin{align*}
\sum_{f=0}^{\infty}\dim \fM_{f+4,f}^{+}(U_{npg}(2))t^f & 
=\frac{t^4+t^9}{(1-t^2)(1-t^4)(1-t^6)(1-t^8)}, \\
\sum_{f=0}^{\infty}\dim \fM_{f+4,f}^{-}(U_{npg}(2))t^f & 
= \frac{t^5+t^8}{(1-t^2)(1-t^4)(1-t^6)(1-t^8)},\\
\sum_{f=0}^{\infty}\dim \fM_{f+4,f}^{+}(U_{npg}(3))t^f 
& = \frac{t^2+t^6}{(1-t^2)(1-t^3)(1-t^4)(1-t^6)},\\
\sum_{f=0}^{\infty}\dim \fM_{f+4,f}^{-}(U_{npg}(3))t^f & 
= \frac{t^3+t^5}{(1-t^2)(1-t^3)(1-t^4)(1-t^6)}.\\
\end{align*}

Then we have 
\begin{align*}
\sum_{k=0}^{\infty}\dim S_{k,4}^+(K(2))t^k & = \frac{P_{+,4}^{(2)}(t)}
{(1-t^2)(1-t^6)(1-t^8)(1-t^{12})}, \\
\sum_{k=0}^{\infty}\dim S_{k,4}^-(K(2))t^k & = \frac{P_{-,4}^{(2)}(t)}
{(1-t^2)(1-t^6)(1-t^8)(1-t^{12})},\\
\sum_{k=0,k\neq 2}^{\infty}\dim S_{k,4}^+(K(3))t^k & = \frac{P_{+,4}^{(3)}(t)}
{(1-t^2)(1-t^4)(1-t^6)(1-t^{12})},\\
\sum_{k=0}^{\infty}\dim S_{k,4}^-(K(3))t^k & = \frac{P_{-,4}^{(3)}(t)}
{(1-t^2)(1-t^4)(1-t^6)(1-t^{12})},\\
\end{align*}
where 
\begin{align*}
P_{+,4}^{(2)}(t) = & 
t^8 + t^{10} + t^{11} + t^{12} + t^{14} + 2 t^{15} + t^{16} + 
  2t^{19} + t^{20} - t^{22} + t^{24} - t^{26}, \\
P_{-,4}^{(2)}(t)= & 
t^7 + t^{10} + t^{11} + t^{12} + t^{14} + 2 t^{15} + t^{16} + t^{19} + 
 2 t^{20} - t^{22} + t^{24} - t^{26},\\
P_{+,4}^{(3)}(t) = & 
t^6 + t^8 + t^9 + t^{10} + t^{11} + t^{12} + t^{14} + 
 2 t^{15} + t^{17} + t^{20} - t^{22}, \\
 P_{-,4}^{(3)}(t) = & 
 2 t^9 + t^{10} + t^{11} + 2 t^{12} + t^{14} + 2 t^{15} + t^{17} + t^{18} + 
 2 t^{20} - t^{21} - t^{22} - t^{24}.
\end{align*}

Here we do not know if $S_{2,4}^{+}(K(3))=0$ or not.
\\ 
 
\subsection{The case of small $k$}
For several small $k\geq 3$ with $j=0$ and some primes $p$, 
we give tables for $\fM_{k-3,\,k-3}^{\pm}(U_{npg}(p))$ and 
$\dim S_k^{\pm}(K(p))$. 
In all the tables below, we put 
\begin{align*}
& H=\dim \fM_{k-3,\,k-3}(U_{npg}(p)), 
\quad R=Tr_{k-3,\,k-3}(R(\pi)),
\\
& M^+=\dim \fM_{k-3,\,k-3}^{+}(U_{npg}(p)), \quad 
M^-=\dim \fM_{k-3,\,k-3}^{-}(U_{npg}(p)),
\\ 
&S_k^{\pm}=\dim S_k^{\pm}(K(p)),\quad  
s_2^{\pm}=\dim S_{2}^{\pm,new}(\Gamma_0(p)),
\end{align*}  
where $k$ is fixed for each table.

Numerical examples for the case $k=3$ has been given in 
in \cite{ibuquinary} p. 218.  We denote the class number and 
the type number of discriminant $p$ of the non-principal 
genus by $H(p)$ and $T(p)$. Since 
we have $S_{2k-2}(SL_2(\Z))=0$ for $k=3$, we have 
$\dim S_3^+(K(p))=H(p)-T(p)$ and 
$\dim S_3^-(K(p))=T(p)-1$ by Theorem \ref{paramodular}. 
Now we determine
all primes $p$ such that $\dim S_3^{+}(K(p))=0$.
This has some geometric meaning.
Let $K(N)^*$ be the maximal extension of $K(N)$ in $Sp(2,\R)$ of order $2^{\nu(N)}$
where $\nu(N)$ is the number of prime divisors of $N$.
In \cite{gritsenkohulek}, it was proved that
$K(N)^{*}\backslash {\mathcal H}_2$ is the moduli 
space of the Kummer surfaces associated to $(1,N)$  
polarizations. In particular, elements of $S_3(K(N)^*)$ give canonical differential 
forms of the Satake compactification of $K(N)^* \backslash {\mathcal H}$ 
(See \cite{freitag}).
It has been  
that $\dim S_3(K(N)^*)=0$ for $N\leq 40$ in   
\cite{breedingpooryuen} and also that 
 $\dim S_3(K(N)^*)\geq 1$ for $N=167$, $173$, $197$,
$213$, $285$ in \cite{gritsenkopooryuen}.
When $N=p$ is a prime, then we have 
$S_3(K(p)^*)=S_3^+(K(p))$. 
Now our new result for $k=3$ is given as follows.
\begin{prop}\label{weight3}
Let $p$ be a prime. Then we have 
\[
\dim S_3^{+}(K(p))=0
\]
if and only if $p$ is any prime such that $p\leq 163$ or 
$p=179$, $181$, $191$, $193$, $199$, $211$, $229$, $241$.
\end{prop}

By the way, we have $\dim S_3^+(K(p))=1$ if 
$p=167$, $173$, $197$, $223$, $233$, $239$, $251$, $271$, $277$, $281$, $313$, $331$, $337$ and 
$\dim S_3^{+}(K(p))=2$ if $p=227$, $257$, $263$, $269$, $283$, $349$, $379$. $409$, $421$.

\begin{proof}[Proof of Proposition \ref{weight3}]
For $p=2$, $3$, $5$, we know $S_3^+(K(p))=0$ for direct calculation 
of the formula in Theorem \ref{paramodular}.
Now assuming $p>5$, by rough estimation first we show that 
$\dim S_3^+(K(p))=H(p)-T(p)=(H(p)-Tr(R_{0,0}(\pi))/2>0$ 
for any $p\geq 3327$.  
On the other hand, for each prime $p<3327$,
we calculate $\dim S_3^+K(p))=(H(p)-Tr(R_{0,0}(\pi)))/2$ directly 
from Theorem \ref{compactdim} and 
Theorem \ref{nonprincipaldim} and can show concretely that 
it is non-negative. This completes the proof.
 
Now we explain how to obtain the estimation for big $p$
given above.
First of all, for $k=3$, any $\chi_i$ in the formula in 
$H(p)$ and $Tr(R_{0,0}(\pi))$ is $1$. 

Denote by $D_0$ the discriminant of the real quadratic field $\Q(\sqrt{p})$ and by $\chi$ the 
character associated to $\Q(\sqrt{p})$. Then by \cite{bernoulli} Theorem 9.6, we have 
\[
L(2,\chi)=\frac{\pi^2}{D_0\sqrt{D_0}}B_{2,\chi}.
\]
By evaluating the Euler product, we have  
$0<L(2,\chi)<\zeta(2)=\pi^2/6$, 
so we have  
\[
\begin{array}{ll}
B_{2,\chi}<\dfrac{p^{3/2}}{6} & \text{ if $p\equiv 1 \bmod 4$,} \\[2ex]
B_{2,\chi}<\dfrac{4p^{3/2}}{3} & \text{ if $p\equiv 3 \bmod 4$.}
\end{array}
\]
So since 
\[
\frac{4}{3}\leq \frac{11}{6} \quad \text{ and } \quad 
\frac{1}{6}\left(9-2\left(\frac{2}{p}\right)\right)\leq \frac{11}{6},
\]
the term in $Tr R_{0,0}(\pi)$ containing $B_{2,\chi}$ is bounded by 
$11p^{3/2}/(2^6\cdot 3^2)$ from the above.
Considering the evaluation \eqref{imaginaryclassno} 
for cases $p\equiv 1 \bmod 4$ 
and $p\equiv 3 \bmod 4$, we can give a common upper bound for 
$h(\sqrt{-p})$, $h(\sqrt{-2p})$ and $h(\sqrt{-3p})$.  
By these inequalities we can estimate $Tr(R_{0,0}(\pi))$ 
in Theorem \ref{compactdim} from the  above as follows. 
\begin{align*}
Tr(R_{0,0}(\pi))& <
\frac{11p^{3/2}}{2^6\cdot 3^2}
 +\frac{1}{2^4\pi}(4\sqrt{p}\log(p)+\sqrt{p})
\\ & \quad +\frac{1}{2^3\pi}(4\sqrt{2p}\log(8p)+\sqrt{2p})
 + \frac{1}{3\pi}(4\sqrt{3p}\log(12p)+\sqrt{3p}).
\end{align*}
Here $Tr(R_{0.0}(\pi))=2T(p)-H(p)$. 
On the other hand, by Theorem 
\ref{nonprincipaldim}, comparing the part for $(-1/p)=1$ and $-1$, 
and also for $(-3/p)=+1$ and $-1$, for $p>5$ we see that 
\[
H(p)>\frac{p^2-1}{2880}+\frac{1}{36}(p+1)+\frac{1}{32}(p+1)
=\frac{p^2-1}{2880}+\frac{17(p+1)}{288}.
\]
So if we put 
\begin{align*}
g(x) &=\frac{x^2-1}{2880}+\frac{17(x+1)}{288}
-\frac{11x^{3/2}}{2^6\cdot 3^2}
 -\frac{1}{2^4\pi}(4\sqrt{x}\log(x)+\sqrt{x})
\\ & \quad -\frac{1}{2^3\pi}(4\sqrt{2x}\log(8x)+\sqrt{2x})
 - \frac{1}{3\pi}(4\sqrt{3x}\log(12x)+\sqrt{3x}),
\end{align*}
then we have $2\dim S_3^+(K(p))>g(p)$ for $p\geq 7$.
For $x>0$, we have 
\[
\frac{dg(x)}{dx}=\frac{h(x)}{5760\pi x},
\]
where 
\begin{align*}
h(x)= & -1620 - 3240 \sqrt{2} - 4320 \sqrt{3} + 340 \pi \sqrt{x]}- 
  165 \pi x + 4 \pi x^{3/2} 
  \\ & - 720 \log(x) - 
  1920 \sqrt{3}\log(3 x) - 1440 \sqrt{2} \log(8x).
\end{align*}
We also have 
\[
h'(x):=\frac{d h(x)}{dx}=\frac{-720-1440\sqrt{2}-1920\sqrt{3}
+\pi(170\sqrt{x}-165 x+6x^{3/2})}{x}.
\]
It is easy to see that the numerator of $h'(x)$ increases 
monotonously for $x>10$ by elementary calculus, and 
since we have $h'(800)=26.06...>0$, we have 
$h'(x)>0$ for $x>800$. This means that 
$g(x)$ increases monotonously for $x>800$. We have 
$g(3327)=0.60...>0$ so we have $g(x)>0$ for 
$x\geq 3327$. 
So we are done.
\end{proof}

Next we study the cases $k=4$, $5$, $6$, $8$ for small $p$.
We have $S_k(Sp(2,\Z))=0$ and 
$S_{2k-2}(SL_2(\Z))=0$ in these cases,  
so we have the following simple relations.
\begin{align*}
\dim S_{k}^{+}(K(p)) & =\dim \fM_{k-3,\,k-3}^{-}(U_{npg}(p)), \\
\dim S_{k}^{-}(K(p)) & = \dim \fM_{k-3,\,k-3}^{+}(U_{npg}(p)).
\end{align*} 
A table of $\dim S_4^{\pm}(K(p))$ has been given in 
\cite{poor} Table 4 in p. 28 by elaborate calculation of 
constructing paramodular forms, but Table \ref{k4} below was obtained 
by our theoretical result independently. 
It is nice to see that the results coincide. 
(We added $p=601$ and $607$ as a small tip to their table.)

The numerical tables in the cases $k=5$, $6$, $8$ are given in 
Table \ref{k5}, \ref{k6}, \ref{k8}, respectively.

By the way, when $k=8$, using our formula, we have  
$\dim S_8^+(K(277))=1761$ and 
$\dim S_8^-(K(277))=768$.
These do not coincide with the numbers given in 
\cite{poor} page 31.
The authors are aware of the error and will publish a 
correction~\cite{website}. 

When $k=7$, we have $\dim S_{2k-2}(SL_2(\Z))=1$ and 
$S_7(Sp(2,\Z))=0$, so we have 
\begin{align*}
\dim S_7^{+}(K(p)) & =\dim \fM_{4,4}^{-}(U_{npg}(p))
-\dim S_2^{+,new}(\Gamma_0(p)), \\
\dim S_7^{-}(K(p)) & = \dim \fM_{4,4}^{+}(U_{npg}(p))
-1-\dim S_2^{-,new} (\Gamma_0(p)).
\end{align*}
The numerical table for $k=7$ is given in Table \ref{k7}.

When $k=10$, we have $\dim S_{10}(Sp(2,\Z))=\dim S_{18}(SL_2(\Z))=1$, 
so the old form in $S_{10}(K(p))$ comes from the Saito--Kurokawa lift and 
we have also the Yoshida lift part in $\fM_{7,7}$. So we have 
\begin{align*}
\dim S_{10}^{+}(K(p)) & = 1+\dim \fM_{7,7}^{-}(U_{npg}(p))-
\dim S_2^{+,new}(\Gamma_0(p)), \\
\dim S_{10}^{-}(K(p)) & = \dim\fM_{7,7}^{+}(U_{npg}(p))
-\dim S_2^{-,new}(\Gamma_0(p)).
\end{align*}
The numerical table for $k=10$ is given in Table \ref{k10}.

{\bf Acknowledgement:}
We would like to thank Neil Dummigan for informing the interesting result in his joint work \cite{dummigan} and for answering author's question in detail,  
as well as his big interest around the problem, and also to Cris Poor and David S. Yuen for constantly asking the author the dimension formula 
with involution and sometimes assuring numerical correctness of our calculation by their examples.

\begin{table}[h]
\caption{The case $k=5$.} \label{k5}
\begin{center}
{\small
\begin{tabular}{crrrrrrrrrrrrrrrrr}\hline
p & 7 & 11 & 13 & 17 & 19 & 23 & 29 & 31 & 37 & 41 & 43 & 47 & 53 & 59 & 61 & 67 
& 71 \\ 
H & 1 & 2 & 3 & 4 & 5 & 5 & 9 & 10 & 14 & 15 & 16 & 16 & 22 & 24 & 31 & 33 & 33 \\
R & 1 & 2 & 3 & 4 & 5 & 5 & 9 & 10 & 14 & 15 & 16 & 14 & 20 & 22 & 31 & 31 & 29 \\ 
$S_5^+$ & 0 & 0 & 0 & 0 & 0 & 0 & 0 & 0 & 0 & 0 & 0 & 1 & 1 & 1 & 0 & 1 & 2\\
$S_5^-$ & 1 & 2 & 3 & 4 & 5 & 5& 9 & 10 & 14 & 15 & 16 & 15 & 21 & 23 & 31 & 32 & 31 \\ \hline
\end{tabular}
}
\end{center}
\end{table}

\begin{table}[htbp]
\caption{The case $k=6$.}\label{k6}
\begin{center}
{\small
\begin{tabular}{crrrrrrrrrrrr}\hline 
p & 7 & 11 & 13 & 17 & 19 & 23 & 29 & 31 & 37 & 41 & 43 & 47 \\ 
H & 2 & 3 & 5 & 6 & 8 & 9 & 14 & 17 & 24 & 25 & 29 & 30 \\
R &-2 & -3 & -5 & -6 & -8 & -9 & -14 & -17 & -24 & -23 & -27 
& -24 \\ 
$S_6^+$ & 2 & 3 & 5 & 6 & 8 & 9 & 14 & 17 & 24 & 24 & 28 & 27 \\
$S_6^-$ & 0 & 0 & 0 & 0 & 0 & 0 & 0 & 0 & 0 & 1 & 1 & 3 
\\ \hline
\end{tabular}
}
\end{center}
\end{table}

\begin{table}[htbp]
\caption{The case $k=8$.}\label{k8} 
\begin{center}
{\small
\begin{tabular}{crrrrrrrrrrrr}\hline 
p & 7 & 11 & 13 & 17 & 19 & 23 & 29 & 31 & 37 & 41 & 43 & 47 \\ 
H & 4 & 6 & 10 & 14 & 17 & 22 & 34 & 40 & 57 & 64 & 72 & 80 \\
R & -4 & -6 & -10 & -12 & -17 & -18 & -28 & -36 & -49 & -48 & -56 & -50 \\ 
$S_8^+$ & 4 & 6 & 10 & 13 & 17 & 20 & 31 & 38 & 53 & 56 & 64 & 65 \\ 
$S_8^-$ & 0 & 0 & 0 & 1 & 0 & 2 & 3 & 2 & 4 & 8 & 8 & 15  \\ \hline
\end{tabular}
}
\end{center}
\end{table}

\begin{table}[!h]
\caption{The case $k=4$}\label{k4}.
\begin{center}
{\tiny
\begin{tabular}{crrrrrrrrrrrrr}\hline
p & 7 & 11 & 13 & 17 & 19 & 23 & 29 & 31 & 37 & 41 & 43 & 47 \\ 
H & 1 & 1 & 2 & 2 & 3 & 3 & 4 & 6 & 8 & 7 & 9 & 8  \\
R & -1 & -1 & -2 & -2 & -3 & -3 & -4 & -6 & -8 & -7 & -9 & -8 \\
$S_4^+$ & 1 & 1 & 2 & 2 & 3 & 3 & 4 & 6 & 8 & 7 & 9 & 8 \\
$S_4^-$ & 0 & 0 & 0 & 0 & 0 & 0 & 0 & 0 & 0 & 0 & 0 & 0 \\ \hline
p & 53 & 59  & 61 & 67 & 71 & 73 & 79 & 83 & 89 & 97 & 101 & 103 \\ 
H & 10 & 11& 16 & 17 & 15 & 21 & 22 & 19 & 23 & 32 & 28 & 33  \\
R & -10 & -11 & -16 & -17 & -15 & -21 & -22 & -17 & -23 & -32 & -26 & -31  \\
$S_4^+$ & 10 & 11 & 16 & 17 & 15 & 21 & 22 & 18 & 23 & 32 & 27 & 32 \\
$S_4^-$ & 0 & 0 & 0 & 0 & 0 & 0 & 0 & 1 & 0 & 0 & 1 & 1 \\ \hline 
p & 107 & 109 & 113 & 127 & 131 & 137 & 139 & 149 & 151 & 157 & 163 & 167 \\
H &  29 & 38 & 34 & 46 & 41 & 47 & 53 & 54 & 61 & 68 & 69 & 63 \\
R & -25 & -38 & -32 & -42 & -35 & -43 & -49 & -48 & -57 & -62 & -61 & -47 \\
$S_4^+$ & 27 & 38 & 33 & 44 & 38 & 45 & 51 & 51 & 59 & 65 & 65 & 55 \\
$S_4^-$ & 2 & 0 & 1 & 2 & 3 & 2 & 2 & 3 & 2 & 3 & 4 & 8 \\ \hline
p & 173 & 179 & 181 & 191 & 193 & 197 & 199 & 211 & 223 & 227 & 229 & 233 \\
H & 70 & 71 & 86 & 80 & 96 & 88 & 97 & 107 & 118 & 109 & 128 & 119 \\
R & -54 & -59 & -80 & -66 & -88 & -68 & -85 & -93 & -94 & -73 & -114 & -93 \\ 
$S_4^+$ & 62 & 65 & 83 & 73 & 92 & 78 & 91 & 100 & 106 & 91 & 121 & 106 \\
$S_4^-$ & 8 & 6 & 3 & 7 & 4 & 10 & 6 & 7 & 12 & 18 & 7 & 13 \\
 \hline
p & 239 & 241 & 251 & 257 & 263 & 269 & 271 & 277 & 281 & 283 & 293 & 307 \\
H & 120 & 140 & 131 & 142 & 143 & 154 & 166 & 178 & 167 & 179 & 180 & 207  \\
R & -90 & -126 & -95 & -106 & -97 & -114 & -132 & -144 & -131 & -131 & -118 & -147  \\ $S_4^+$ & 105 & 133 & 113 & 124 & 120 & 134 & 149 & 161 & 149 & 155 & 149 & 177  \\
$S_4^-$ & 15 & 7 & 18 & 18 & 23 & 20 & 17 & 17 & 18 & 24 & 31 & 30 \\ 
\hline 
p & 311 &  313 & 317 & 331 &  337 & 347 & 349 & 353 & 359 & 367 & 373 & 379  \\
H & 195 & 221 & 208 & 237 & 252 & 239 & 268 & 254 & 255 & 286 & 302 & 303 \\
 R& -131 & -179 & -140 & -185 & -202 & -145 &  -210 & -170 & -165 & -196 & -224 & -225  \\
$S_4^+$ & 163 &  200 & 174 & 211 & 227 & 192 & 239 & 212 & 210 & 241 & 263 & 264 \\
$S_4^-$ & 32 & 21 & 34 & 26 & 25 & 47 & 29 & 42 & 45 & 45 & 39 & 39 \\
\hline 
p &  383 & 389 & 397 &  401 & 409 & 419 & 421 & 431 & 433 & 439 & 443 & 449  \\
H &  288 & 304 & 338 & 322 & 357 &  341 & 376 & 360 & 396 &  397 & 379 & 398 \\
R &  -164 & -208 & -240 & -226 & -279 & -203 & -290 & -214 & -290 &-269 & -215 & -268   \\
$S_4^+$ & 226 & 256 & 289 & 274 & 318 & 272 & 333 & 287 & 343 & 333 & 297 & 333  \\
$S_4^-$ & 62 & 48 & 49 & 48 & 39 & 69 & 43 & 73 & 53 & 64 & 82 & 65 \\ \hline 
p & 457 & 461 & 463 & 467 & 469 & 487 & 491 & 499 & 503 & 509 & 521 & 523  \\
H & 437 & 418 & 438 & 419 & 440 & 481 & 461 & 503 & 483 & 504 & 527 & 549 \\
R & -319& -252 & -286 & -223 & -242  &-305  & -265 & -341 & -243 & -296 & -325 & -325 \\
$S_4^+$ & 378 & 335 & 362 & 321 & 341 & 393 & 363 & 422 & 363 & 400 & 426 & 437 \\
$S_4^-$ & 59 & 83 & 76 & 98 & 99 & 88 & 98 & 81 & 120 & 104 & 101 & 112 \\ \hline
p & 541 & 547 & 557 & 563 &  569 & 571 & 577 & 587 & 593 & 599 & 601 & 607 \\
H &  596 & 597 & 598 & 599 & 623 & 647 & 672 & 649 & 674 & 675 & 725 & 726 \\
R & -416 & -359 & -322 & -287 & -381 & -413 & -444 & -311 & -362 & -363 & -497 & -404 \\
$S_4^+$ & 506 & 478 & 460 & 443 & 502 & 530 & 558 & 480 & 518 & 519 & 611 & 565 \\
$S_4^-$ & 90 & 119 & 138 & 156 & 121 & 117 & 114 & 169 & 156 & 156 & 114 & 161 \\ \hline
\end{tabular}
}
\end{center}
\end{table}

\begin{table}[!hbt]
\caption{The case $k=7$}\label{k7}
\begin{center}
{\small
\begin{tabular}{crrrrrrrrrrrr} \hline 
p & 7 & 11 & 13 & 17 & 19 & 23 & 29 & 31 & 37 & 41 & 43 & 47 \\ 
H & 3 & 5  & 8 & 10 & 13 & 15 & 24 & 28 & 40 & 43 & 49 & 52 \\
R & 3 & 5 & 8 & 10 & 13 & 13 & 22 & 26 & 36 & 37 & 41 & 36 \\ 
$M^+$ & 3 & 5 & 8 & 10 & 13 & 14 & 23 & 27 & 38 & 40 & 45 & 44 \\
$M^-$ & 0 & 0 & 0 & 0 & 0 & 1 & 1 & 1 & 2 & 3 & 4 & 8 \\
$s_2^+$ & 0 & 0 & 0 & 0 & 0 & 0 & 0 & 0 & 1 & 0 & 1 & 0 \\
$s_2^-$ & 0 & 1 & 0 & 1 & 1 & 2 & 2 & 2 & 1 & 3 & 2 & 4 \\
$S_7^+$ & 0 & 0 & 0 & 0 & 0 & 1 & 1 & 1 & 1 & 3 & 3 & 8 \\
$S_7^-$ & 2 & 3 &  7 & 8 & 11 & 11 & 20 & 24 & 36 & 36 & 42 & 39 \\
\hline
\end{tabular}
}
\end{center}
\end{table}

\begin{table}[!htbp]
\caption{The case $k=10$.}\label{k10}
\begin{center}
{\small
\begin{tabular}{crrrrrrrrrrrr}\hline 
p & 7 & 11 & 13 & 17 & 19 & 23 & 29 & 31 & 37 & 41 & 43 & 47 \\ 
H & 6 & 12 & 18 & 26 & 34 & 44 & 70 & 82 & 116 & 134 & 150 & 170 \\ 
R & -6 & -10 & -16 & -20 & -28 & -30 & -48 & -60 & -82 & -82 & -94 
& -84 \\ 
$M^+$ & 0 & 1 & 1 & 3 & 3 & 7 & 11 & 11 & 17 & 26 & 28 &  43  
\\
$M^-$ & 6 & 11 & 17 & 23 & 31 & 37 & 59 & 71 & 99 & 108 & 122 & 127 \\
$s_2^+$ & 0 & 0 & 0 & 0 & 0 & 0 & 0 & 0 & 1 & 0 & 1 & 0 \\
$s_2^-$ & 0 & 1 & 0 & 1 & 1 & 2 & 2 & 2 & 1 & 3 & 2 & 4 \\
$S_{10}^{+}$ & 7 & 12 & 18 & 24 & 32 & 38 & 60 & 72 & 99 & 109 & 122 & 128 \\
$S_{10}^{-}$ & 0 & 0 & 1 & 2 & 2 & 5 & 9 & 9 & 16 & 23 & 26 & 39 \\
\hline 
\end{tabular}
}
\end{center}
\end{table}


\begin{thebibliography}{99}
\bibitem{bernoulli}
T. ~Arakawa, T. ~Ibukiyama and M. ~Kaneko,
\textit{Bernoulli numbers and zeta functions}, 
Springer Verlag (2014), vi+274 pp.  

\bibitem{asai}
T. ~Asai, The class numbers of positive definite quadratic forms, 
Japan. J. Math. \textbf{3}, (1977), 239--296.
\url{https://doi.org/10.4099/math1924.3.239}.

\bibitem{breedingpooryuen}
J. ~Breeding II, C. ~Poor, D. S. ~Yuen, 
Computations of spaces of paramodular
forms of general level, J. Korean Math. Soc. \textbf{53}
(2016) 645--689.
https://doi.org/10.4134/JKMS.j150219

\bibitem{brumerkramer}
A. ~Brumer and K. ~Kramer, 
Paramodular abelian varieties of odd conductor, 
Trans. Amer. Math. Soc. 366 (2014), 2463--2516.

\bibitem{dern}
T. ~Dern,
Paramodular forms of degree 2 and level 3. Comment. Math. Univ. St. Paul. \textbf{51}
 (2002), 157--194.
\url{http://doi.org/10.14992/00008717}.

\bibitem{dummigan}
N. ~Dummigan, A. Pacetti, G. Rama and G. Tornar\'ia, 
Quinary forms and paramodular forms, preprint (2021), 
52 pp. ArXiv: 2112.03797v1. 

\bibitem{eichlerzagier}
M. ~Eichler and D. ~Zagier,
\textit{The theory of Jacobi forms}, Progr. Math. \textbf{55}, 
Birkh\"{a}user Boston, Inc. Boston, MA, 1985,v+148 pp/ 


\bibitem{freitag}
E. ~Freitag, \textit{Siegelsche Modulfunktionen},  
Grundlehren math. Wissenschaften \textbf{254}
Springer-Verlag, Berlin, (1983), x+341 pp.  

\bibitem{gritsenko}
V. A. ~Gritsenko, Arithmetical lifting and its applications,
in Number theory(Paris 1992-1993),
London Math. Soc. Lecture Note Ser., \textbf{215}, 
Cambridge Univ. Press, Cambridge, 
(1995), 103--126.
\url{https://doi.org/10.1017/CBO9780511661990.008}.

\bibitem{gritsenkohulek}
V. ~Gritsenko and K. ~Hulek, 
Minimal Siegel modular threefolds.
Math. Proc. Cambridge Philos. Society \textbf{123}
(1998), 461--485.
\url{https://doi.org/10.1017/S0305004197002259}.

\bibitem{gritsenkopooryuen}
V. ~Gritsenko, C. ~Poor and D. S. ~Yuen, 
Antisymmetric paramodular forms of weight $2$ and $3$, 
 Int. Math. Res. Not. IMRN (2020), no. 20, 6926--6946.
\url{https://doi.org/10.1093/imrn/rnz011}.

\bibitem{gross}
B. H. ~Gross, Algebraic modular forms. Israel J. Math. \textbf{113} 
(1999), 61--93.
\url{https://doi.org/10.1007/BF02780173}. 

\bibitem{hashimoto}
K. ~Hashimoto, 
On Brandt matrices associated with the positive definite quaternion Hermitian forms. 
J. Fac. Sci. Univ. Tokyo Sect. IA Math. \textbf{27} (1980), 227--245.

\bibitem{hashimotoibu}
K. ~Hashimoto and T. ~Ibukiyama, 
On class numbers of positive definite binary quaternion hermitian forms 
I, J. Fac. Sci. Univ. Tokyo Sec. IA \textbf{27} (1980), 549--601;
II, ibid., \textbf{28} (1982), 695-699;
III ibid., \textbf{30} (1983), 393--401. 

\bibitem{hashimotoibudimII}
K. ~Hashimoto and T. ~Ibukiyama, 
On relations of dimensions of automorphic forms of 
$Sp(2,\R)$ and its compact twist $Sp(2)$ (II). 
Advanced Studies in Pure Math., \textbf{7}, (1985), 30--102.
\url{https://doi.org/10.2969/aspm/00710031}.

\bibitem{ibueuler}
T. ~Ibukiyama, On symplectic Euler factors of genus two,
J. Fac. Sci. Univ. Tokyo Sec. IA \textbf{30}(1984), 587--614.
\url{https://doi.org/10.15083/00039549}.

\bibitem{ibuparamodular}
T. ~Ibukiyama, 
On relations of dimensions of automorphic forms of 
$Sp(2,\R)$ and its compact twist $Sp(2)$ (I). 
Advanced Studies in Pure Math., \textbf{7}, (1985), 7--29. 
\url{https://doi.org/10.2969/aspm/00710007}. 


\bibitem{iharaibu}
T. ~Ibukiyama and Y. ~Ihara, 
On automorphic forms on the unitary symplectic group 
$Sp(n)$ and $SL_2(R)$, Math. Ann. \textbf{278} (1987), 307--327. 
\url{https://doi.org/10.1007/BF01458073}.

\bibitem{ibucusp}
T. ~Ibukiyama, 
On some alternating sums of dimensions of Siegel modular forms of general degree 
and cusp configurations. J. Fac. Sci. Univ. Tokyo Sect. IA Math. \textbf{40}(1993), 
245--283.
\url{https://doi.org/10.15083/00039320}.

\bibitem{ibuonodera}
T. ~Ibukiyama and F. ~Onodera, 
On the graded ring of modular forms of the Siegel paramodular group of level 2. 
Abh. Math. Sem. Univ. Hamburg \textbf{67} (1997), 297--305. 
\url{https://doi.org/10.1007/BF02940837}.

\bibitem{ibuparavector}
T. ~Ibukiyama, Paramodular forms and compact twist, Proceedings of 
9-th Hakuba autumn workshop 2006, {\it Automorphic Forms on $GSp(4)$}. 
(2007), 37--48.

\bibitem{ibuweightthree}
T. ~Ibukiyama, Dimension formulas of Siegel modular forms of weight 3 and supersingular 
abelian surfaces, Proceedings of the $4$-th Spring Conference on modular forms and 
related topics, {\it Siegel modular forms and abelian varieties}, 
(2007), 39--60. Revised version is on the following web:
\url{http://www4.math.sci.osaka-u.ac.jp/~ibukiyama/pdf/2007weightthreeprocrevised2.pdf}.

\bibitem{ibuvect246}
T. ~Ibukiyama, Vector valued Siegel modular forms of 
symmetric tensor weight of small degrees, 
Comment. Math. Univ. St. Pauli \textbf{61}(2012), 87--102.
\url{http://doi.org/10.14992/00008607}

\bibitem{pooryuenibu}
T. ~Ibukiyama, C. ~Poor and D. S. ~Yuen, 
 Jacobi forms that characterize paramodular forms. 
 Abh. Math. Semin. Univ. Hambg. \textbf{83} (2013), 111--128.
\url{DOI 10.1007/s12188-013-0078-y}.

\bibitem{ibukitayama}
T. ~Ibukiyama and H. ~Kitayama, 
Dimension formulas of paramodular forms of squarefree level and comparison with inner twist. 
J. Math. Soc. Japan \textbf{69} (2017), 597--671.
\url{https://doi.org/10.2969/jmsj/06920597}.

\bibitem{ibumiddle}
T. ~Ibukiyama,  Conjectures on correspondence of symplectic modular forms 
of middle parahoric type and Ihara lifts, 
 Res. Math. Sci. \textbf{5} (2018), no. 2, Paper No. 18, 36 pp.
 \url{https://doi.org/10.1007/s40687-018-0136-2}.

\bibitem{ibutypenumber}
T. ~Ibukiyama, Type numbers of quaternion hermitian forms 
and supersingular abelian varieties, 
Osaka J. Math. \textbf{55} (2018), 369--384.

\bibitem{ibuquinary}
T. ~Ibukiyama, Quinary lattices and binary quaternion hermitian lattices, 
Tohoku Math. J. (2) \textbf{71} (2019), 207--220. 
\url{DOI: 10.2748/tmj/1561082596}.

\bibitem{ibulattice}
T. ~Ibukiyama, 
Supersingular abelian varieties and quaternion hermitian lattices, 
RIMS Kokyuroku Bessatsu B90 (2022), 17--37.

\bibitem{igusa}
J. ~Igusa, 
On Siegel modular forms of genus two. Amer. J. Math. \textbf{84} (1962), 175--200.
\url{https://doi.org/10.2307/2372812}.

\bibitem{igusa2}
J. ~Igusa, On Siegel modular forms genus two II. Amer. J. Math. 
\textbf{86} (1964), 392--412.
\url{https://doi.org/10.2307/2373172}.

\bibitem{ihara}
Y. ~Ihara, On certain arithmetical Dirichlet series, 
J. Math. Soc. Japan \textbf{16}, (1964), 214--225. 
\url{DOI:10.2969/jmsj/01630214}.

\bibitem{kneser}
M. ~Kneser, Starke Approximation in algebraischen Gruppen I. 
J. Reine Angew. Math. \textbf{218} (1965), 190--203. 
\url{https://doi.org/10.1515/crll.1965.218.190}.

\bibitem{kneser2}
M. ~Kneser, Strong approximation, 
Algebraic Groups and Discontinuous 
Subgroups (Proc. Sympos. Pure Math. Boulder, Colo. 1965) 
(1966), 187--196 Amer. Math. Soc., Providence.
\url{http://dx.doi.org/10.1090/pspum/009/0213361}.

\bibitem{ladd}
W. B. ~Ladd, Algebraic modular forms on $SO_5(\Q)$ and the computation of 
paramodular forms, PhD thesis, Berkeley 2018,  \\
\url{https://digitalassets.lib.berkeley.edu/etd/ucb/text/Ladd_berkeley_0028E_17895.pdf}.

\bibitem{marschner}
A. ~Marschner, 
Paramodular forms of degree 2 with particular emphasis on level
$t=5$, Dissertation (advisors A. Krieg , J. M\H{u}ller), RWTH Aachen, 2004.


\bibitem{martin}
K. ~Martin, Refined dimensions of cusp forms, equidistribution and bias of signs, J. Number Theory \textbf{188}(2018),1--17.
\url{https://doi.org/10.1016/j.jnt.2018.01.015}.

\bibitem{pari}
    The PARI~Group, PARI/GP version 2.13.4, 
    Univ. Bordeaux, 2022, 
    \url{http://pari.math.u-bordeaux.fr}.

\bibitem{petersen}
D. Petersen, Cohomology of local systems on the moduli 
of principally polarized abelian surfaces, 
Pacific J. Math. \textbf{275} (2015), 39--61.
\url{DOI: 10.2140/pjm.2015.275.39}.

\bibitem{poor}
C. ~Poor and D. ~Yuen, 
Paramodular cusp forms. Math. Comp. \textbf{84} (2015), 
1401--1438
\url{https://doi.org/10.1090/S0025-5718-2014-02870-6}. 

\bibitem{website}
C. ~Poor, J. ~Schurman and D. S. ~Yuen, 
Siegel modular forms computation pages, 
\url{http://siegelmodularforms.org}.

\bibitem{robertschmidt}
B. ~Roberts and R. ~Schmidt, 
Local new forms for $GSp(4)$,
Lecture Notes in Math. \textbf{1918}, Springer Verlag (2007).

\bibitem{royschmidtyi}
M. ~Roy, R. ~Schmidt and S. ~Yi,
On counting cuspidal automorphic representations for 
$GSp(4)$, Forum Math. \textbf{33}, (2021), 821--843.
\url{https://doi.org/10.1515/forum-2020-0313}.

\bibitem{weissauer}
M. ~R\"osner and R. ~Weissauer, 
Global liftings between inner forms of 
$GSp(4)$, preprint 2021, arXiv:2103.14715.

\bibitem{satake}
I. ~Satake, 
Surjectivit\'e global de op\'erateur $\Phi$, in S\'eminare Cartan, 
E. N. E. 1957/58, Expos\'e 16, 1--17.

\bibitem{satoh}
T. ~Satoh, On certain vector valued Siegel modular forms of 
dehgree two, Math. Ann. \textbf{274}(1986), 335--352.
\url{https://doi.org/10.1007/BF01457078}.

\bibitem{schmidtiwahori}
R. ~Schmidt, 
Iwahori-spherical representations of $GSp(4)$ and Siegel modular forms of 
degree 2 with square-free level. J. Math. Soc. Japan \textbf{57}, (2005), 
259--293.
\url{DOI: 10.2969/jmsj/1160745825}.

\bibitem{schmidtSK}
R. ~Schmidt, 
On classical Saito-Kurokawa liftings, J. Reine Angew. Math. \textbf{604},
(2007), 211-236.
\url{DOI 10.1515/CRELLE.2007.024}.

\bibitem{schmidtpacket}
R. ~Schmidt, Packet structure and paramodular forms, 
Trans. Amer. Math. Soc. \textbf{370} (2018), no. 5, 3085--3112.
\url{http://dx.doi.org/10.1090/tran/7028}.

\bibitem{schmidt}
R. ~Schmidt, Paramodular forms in CAP representations of 
$GSp(4)$, Acta Arith. \textbf{194}(4),(2020), 319--340.
\url{https://doi.org/10.1090/tran/7028}. 

\bibitem{serre}
J. P. ~Serre, Cours d'arithm\'{e}tique, 
Collection SUP: "Le Math\'{e}maticien'', 
2 Presses Universitaires de France, Paris 1970 188 pp.

\bibitem{siegel}
C. L. ~Siegel, Analytische Zahlentheorie I (Vorlesung, gehalten im 
Wintersemester 1963/64, an der Universit\"{a}t G\"{o}ttingen).


\bibitem{shimura}
G. ~Shimura, An arithmetic of alternating forms and quaternion hermitian forms, 
J. Math. Soc. Japan \textbf{15} (1963), 33--65.
\url{DOI: 10.2969/jmsj/01510033}.

\bibitem{stanley1}
R. P. ~Stanley, Hilbert functions of graded algebras,
Advances Math. \textbf{28}(1978), 57--83.
\url{https://doi.org/10.1016/0001-8708(78)90045-2}.

\bibitem{stanley2}
R. P. ~Stanley, 
Invariants of finite groups and their applications to combinatorics, 
Bull. AMS, \textbf{1}(1979). 475--511. 
\url{https://doi.org/10.1090/s0273-0979-1979-14597-x}.

\bibitem{tsushima}
R. ~Tsushima, 
An explicit dimension formula for the spaces of generalized automorphic forms with 
respect to 
$Sp(2,Z)$. Proc. Japan Acad. Ser. A Math. Sci. \textbf{59} (1983), 139--142
\url{DOI: 10.3792/pjaa.59.139}.

\bibitem{vanhoften}
P. van ~Hoften, 
A geometric Jacquet-Langlands correspondence for paramodular Siegel threefolds. 
Math. Z. \textbf{299} (2021), 2029--2061. 
\url{https://doi.org/10.1007/s00209-021-02756-0}.

\bibitem{weyl}
H. Weyl, 
The classical groups. Their invariants and representations. 
Fifteenth printing. 
Princeton Landmarks in Mathematics. Princeton Paperbacks. Princeton University Press, 
Princeton, NJ, 1997. xiv+320 pp.

\bibitem{williams}
B. ~Williams, Grated rings of paramodular forms of levels 5 and 7, 
J. Number Theory \textbf{209}, (2020), 483--515. 
\url{https://doi.org/10.1016/j.jnt.2019.08.009}.

\bibitem{mathematica}
Wolfram Research, Inc., Mathematica, Version 13.2, Champaign, 
IL (2022).

\bibitem{yamauchi}
M. ~Yamauchi, On the traces of Hecke operators for a normalizer of 
$\Gamma_0(N)$, J. Math. Kyoto Univ. \textbf{13-2}(1973), 403--411.
\url{DOI: 10.1215/kjm/1250523380}.

\bibitem{yoshida}
H. ~Yoshida, On Siegel modular forms obtained 
from theta series.
J. Reine Angew. Math.352(1984), 184--219.
\url{https://doi.org/10.1515/crll.1984.352.184}.

\end{thebibliography}
\end{document}